\documentclass[reqno,a4paper,11pt]{amsart}
\usepackage{amsfonts,amsmath,amsthm,amssymb,amscd}
\usepackage{amsthm}
\usepackage{comment}

\usepackage{chngcntr}

\usepackage[mathscr]{euscript}

\usepackage{frcursive}
\usepackage[T1]{fontenc}  %%% Must be with \usepackage{frcursive}
\usepackage{helvet}
\usepackage{anyfontsize}

\usepackage{epsf}

\usepackage{graphicx,psfrag}
\usepackage{tabularx}
\usepackage{multicol}
\usepackage{color}
\font\sans=cmss12

\font\sc=cmcsc10 at 12truept 
 at 9truept

\def \SaK{\text{\sans K}}

\def \ScptA{\mathscr{A}}
\def \ScptB{\mathscr{B}}

\def \ScptS{\mathscr{S}}

\def \Ak{\mathscr{A}}
\def \Bk{\mathscr{B}}
\def \Ck{\mathscr{C}}

\def \Fk{\mathscr{F}}
\def \Gk{\mathscr{G}}
\def \Hk{\mathscr{H}}

\def \Lk{\mathscr{L}}
\def \Mk{\mathscr{M}}
\def \Nk{\mathscr{N}}

\def \Pk{\mathscr{P}}

\def \Sk{\mathscr{S}}

\def \Uk{\mathscr{U}}
\def \Vk{\mathscr{V}}

\def \Xk{\mathscr{X}}

\def \Zk{\mathscr{Z}}

\def \BFq{\mathcal{B}}

\def \GFq{\mathcal{G}}
\def \HFq{\mathcal{H}}

\def \LFq{\mathcal{L}}
\def \MFq{\mathcal{M}}

\def \PFq{\mathcal{P}}
\def \QFq{\mathcal{Q}}
\def \RFq{\mathcal{R}}

\def \TFq{\mathcal{T}}
\def \UFq{\mathcal{U}}
\def \VFq{\mathcal{V}}
\def \WFq{\mathcal{W}}

\def \ZFq{\mathcal{Z}}

\def \SamB{\text{{\fontfamily{phv}\selectfont{B}}}}

\def \SamG{\text{{\fontfamily{phv}\selectfont{G}}}}

\def \SamP{\text{{\fontfamily{phv}\selectfont{P}}}}

\def \SamS{\text{{\fontfamily{phv}\selectfont{S}}}}

\def \SamZ{\text{{\fontfamily{phv}\selectfont{Z}}}}

\def \mB{\text{\rm{\bf{B}}}}

\def \mG{\text{\rm{\bf{G}}}}

\def \mH{\text{\rm{\bf{H}}}}

\def \mM{\text{\rm{\bf{M}}}}

\def \mP{\text{\rm{\bf{P}}}}

\def \mS{\text{\rm{\bf{S}}}}

\def \mU{\text{\rm{\bf{U}}}}

\def \mk{k}
\def \mK{K}

\def \myadm{\text{\rm{{\,}sm}}}

\def \Ad{\text{\rm{Ad}}}

\def \myBall{\text{\rm{Ball}}}
\def \myBruhat{\text{\rm{Bru}}}
\def \myCoxeter{\text{\rm{W}}}
\def \myFT{\text{\rm{FT}}}

\def \myGal{\text{\rm{Gal}}}
\def \mygrad{\text{\rm{grad}}}
\def \myht{\text{\rm{ht}}}

\def \myInf{\text{\rm{Inf}}}
\def \myId{\text{\rm{Id}}}

\def \mypfacet{\text{\rm{recess}}}
\def \mypfacets{\text{\rm{recesses}}}
\def \myopp{\text{\rm{opp}}}
\def \myrank{\text{\rm{rank}}}
\def \myred{\text{\rm{red}}}
\def \myun{\text{\rm{un}}}
\def \myVect{\text{\rm{Vect}}}

\def \End{\text{\rm{End}}}

\def \meas{\text{\rm{meas}}}

\def \sgn{\text{\rm{sgn}}}
\def \trace{\text{\rm{trace}}}

\def \bC{{\mathbb C}}
\def \bF{{\mathbb F}}

\def \bN{{\mathbb N}}

\def \bR{{\mathbb R}}

\def \bZ{{\mathbb Z}}

\def \fkg{{\mathfrak g}}

\def \frcA{ {\text{\rm{\fontsize{9}{10}{\fontfamily{frc}\selectfont{A}}}}} }

\def \frcL{ {\text{\rm{\fontsize{9}{10}{\fontfamily{frc}\selectfont{L}}}}} }

\def \frcQ{ {\text{\rm{\fontsize{9}{10}{\fontfamily{frc}\selectfont{Q}}}}} }
\def \frcR{ {\text{\rm{\fontsize{9}{10}{\fontfamily{frc}\selectfont{R}}}}} }

\def \frcLm{ {\text{\rm{\fontsize{7}{8}{\fontfamily{frc}\selectfont{L}}}}} }
\def \frcSm{ {\text{\rm{\fontsize{7}{8}{\fontfamily{frc}\selectfont{S}}}}} }

\def \frcAs{ {\text{\rm{\fontsize{5}{6}{\fontfamily{frc}\selectfont{A}}}}} }
\def \frcBs{ {\text{\rm{\fontsize{5}{6}{\fontfamily{frc}\selectfont{B}}}}} }
\def \frcLs{ {\text{\rm{\fontsize{5}{6}{\fontfamily{frc}\selectfont{L}}}}} }

\def \frcQs{ {\text{\rm{\fontsize{5}{6}{\fontfamily{frc}\selectfont{Q}}}}} }
\def \frcRs{ {\text{\rm{\fontsize{5}{6}{\fontfamily{frc}\selectfont{R}}}}} }
\def \frcSs{ {\text{\rm{\fontsize{5}{6}{\fontfamily{frc}\selectfont{S}}}}} }

\def \myFFsl2{{\mathfrak s}{\mathfrak l}  (2,{\mathbb F}_{q})}
\def \myFFGL2{{\text{\rm{GL}}}  (2,{\mathbb F}_{q})}
\def \myFFSL2{{\text{\rm{SL}}}  (2,{\mathbb F}_{q})}

\def \myauthor{Dan Barbasch, Dan Ciubotaru and Allen Moy}

\catcode`\@=11

\@addtoreset{equation}{section} \@addtoreset{equation}{subsection}
\def\theequation{\ifnum\value{subsection}>0\relax
\thesubsection.\arabic{equation}\relax
\else\ifnum\value{section}>0\relax
\thesection.\arabic{equation}\relax \else\arabic{equation}\fi\fi}

\newtheorem{thm}[equation]{Theorem}
\newtheorem{lemma}[equation]{Lemma}
\newtheorem{prop}[equation]{Proposition}

\newtheorem*{thm*}{Theorem}
\newtheorem*{prop*}{Proposition}

\newtheorem{cor}[equation]{Corollary}

\newcommand \reBna{B}

\newcommand \reBD{BD}
\newcommand \reBKV{BKV}

\newcommand \reDta{Dt}
\newcommand \reDna{Dn}

\newcommand \reHCb{HC}
\newcommand \reIM{IM}

\newcommand \reKota{K}

\newcommand \reMS{MS}

\newcommand \reMPa{MPa}
\newcommand \reMPb{MPb}

\newcommand \reSSa{SS}

\newcommand \reTa{T}

\newcommand{\xRightarrow}[2][]{\ext@arrow 0359\Rightarrowfill@{#1}{#2}}
\newcommand*{\medcap}{\mathbin{\scalebox{1.3}{\ensuremath{\cap}}}}
\newcommand*{\medcup}{\mathbin{\scalebox{1.3}{\ensuremath{\cup}}}}

\begin{document}
 
{\large{
 
%%%%%%%%%%%%%%%%%%%%%%%%%%%%%%%%%%%%%%%%%%%%%%%%%%%%%%%%%%%%%%%%%%%%%%%%%
%%%%%%%%%%%%%%%%%%%%%%%%%%%%%%%%%%%%%%%%%%%%%%%%%%%%%%%%%%%%%%%%%%%%%%%%%
%%%%%%%%%%%%%%%%%%%%%%%%%%%%%%%%%%%%%%%%%%%%%%%%%%%%%%%%%%%%%%%%%%%%%%%%%

\noindent{\ } \hfill  {\sc October 2016}

\vskip 0.30in

\title[An Euler-Poincar{\'{e}} formula for a depth zero Bernstein projector]  
{An Euler-Poincar{\'{e}} formula for a depth zero\\  Bernstein projector}
\markboth{\myauthor}{An Euler-Poincar{\'{e}} formula for a depth zero Bernstein projector}

\author{\myauthor}
%%%%%%%%%%%%%%%%%%%%%%%%%%%%%%%%%%%%%%%%%%%%%%%%%%%%%%%%%%%%%%%%%%%%%%%%%
%%%%%%%%%%%%%%%%%%%%%%%%%%%%%%%%%%%%%%%%%%%%%%%%%%%%%%%%%%%%%%%%%%%%%%%%%
%%%%%%%%%%%%%%%%%%%%%%%%%%%%%%%%%%%%%%%%%%%%%%%%%%%%%%%%%%%%%%%%%%%%%%%%%

\address{Department of Mathematics, Malott Hall, Cornell University, Ithaca, NY 14853--0099, USA \\
Email:\tt barbasch{\char'100}math.cornell.edu}

\address{Mathematical Institute, Andrew Wiles Building, University of Oxford \\
  Oxford, OX2 6GG, UK \\
Email:\tt dan.ciubotaru{\char'100}maths.ox.ac.uk}

\address{Department of Mathematics, The Hong Kong University of Science and Technology \\
Clear Water Bay Road, Hong Kong \\
Email:\tt amoy{\char'100}ust.hk}

%%%%%%%%%%%%%%%%%%%%%%%%%%%%%%%%%%%%%%%%%%%%%%%%%%%%%%%%%%%%%%%%%%%%%%%%%
%%%%%%%%%%%%%%%%%%%%%%%%%%%%%%%%%%%%%%%%%%%%%%%%%%%%%%%%%%%%%%%%%%%%%%%%%
%%%%%%%%%%%%%%%%%%%%%%%%%%%%%%%%%%%%%%%%%%%%%%%%%%%%%%%%%%%%%%%%%%%%%%%%%

\thanks{The first author is partly supported by NSA grant H98230-16-1-0006.}

\thanks{The second author is partly supported by United Kingdom EPSRC grant EP/N033922/1.}

\thanks{The third author is partly supported by Hong Kong Research Grants Council grant CERG {\#}603813.}


%%%%%%%%%%%%%%%%%%%%%%%%%%%%%%%%%%%%%%%%%%%%%%%%%%%%%%%%%%%%%%%%%%%%%%%%%
%%%%%%%%%%%%%%%%%%%%%%%%%%%%%%%%%%%%%%%%%%%%%%%%%%%%%%%%%%%%%%%%%%%%%%%%%
%%%%%%%%%%%%%%%%%%%%%%%%%%%%%%%%%%%%%%%%%%%%%%%%%%%%%%%%%%%%%%%%%%%%%%%%%

\begin{abstract}

Work of Bezrukavnikov--Kazhdan--Varshavsky uses an equivariant system of trivial idempotents of Moy--Prasad groups to obtain an Euler--Poincar{\'{e}} formula for the r--depth Bernstein projector.  We establish an Euler--Poincar{\'{e}} formula for the projector to an individual depth zero Bernstein component in terms of an equivariant system of Peter--Weyl idempotents of parahoric subgroups $\Gk_{F}$ associated to a block of the reductive quotient $\Gk_{F}/\Gk^{+}_{F}$.

\end{abstract}

%%%%%%%%%%%%%%%%%%%%%%%%%%%%%%%%%%%%%%%%%%%%%%%%%%%%%%%%%%%%%%%%%%%%%%%%%
%%%%%%%%%%%%%%%%%%%%%%%%%%%%%%%%%%%%%%%%%%%%%%%%%%%%%%%%%%%%%%%%%%%%%%%%%
%%%%%%%%%%%%%%%%%%%%%%%%%%%%%%%%%%%%%%%%%%%%%%%%%%%%%%%%%%%%%%%%%%%%%%%%%
 
\vskip 0.30in 
\maketitle 
 
%%%%%%%%%%%%%%%%%%%%%%%%%%%%%%%%%%%%%%%%%%%%%%%%%%%%%%%%%%%%%%%%%%%%%%%%%
%%%%%%%%%%%%%%%%%%%%%%%%%%%%%%%%%%%%%%%%%%%%%%%%%%%%%%%%%%%%%%%%%%%%%%%%%
%%%%%%%%%%%%%%%%%%%%%%%%%%%%%%%%%%%%%%%%%%%%%%%%%%%%%%%%%%%%%%%%%%%%%%%%%

\section{Introduction}

\medskip
The results expounded here are the merger of several themes in the representation theory of reductive p-adic groups.  Suppose $\mk$ is a non-archimedean local field, and $\mG$ is a connected reductive linear algebraic group defined over $\mk$.  Let $\Gk = \mG (\mk )$ denote the group of $\mk$-rational points, and let $\ScptB = \ScptB (\Gk )$ denote the Bruhat--Tits bulding of $\Gk$. 

\medskip

Two themes introduced in the 1980s were Euler-Poincar{\'{e}} functions [{\reKota}] and the Bernstein center [{\reBna},{\reBD}].   

\smallskip

\begin{itemize}
  \item[$\bullet$] When $\Gk$ is semisimple, simply connected, Kottwitz selected a set ${\frcSm}$ of representatives for the orbits of $\Gk$ on the facets of $\ScptB$, and defined an Euler-Poincar{\'{e}} function $f_{\text{\rm{EP}}} \in C^{\infty}_{c}(\Gk )$ as $f_{\text{\rm{EP}}} \, = \, {\underset { \sigma \in {\frcSs}} \sum} (-1)^{\dim (\sigma )} \ {\frac{1}{{\text{\rm{Stab}}}(\sigma )}} \ \sgn_{\sigma}$ (see [{\reKota}] for the description of the character $\sgn_{\sigma}$ of ${\text{\rm{Stab}}}(\sigma )$ and other normalizations).  It is obvious that changing the set $\frcSm$ changes $f_{\text{\rm{EP}}}$, but not its orbital integrals.  Kottwitz showed the usefulness of $f_{\text{\rm{EP}}}$ as a function to enter into the trace formula.

\smallskip

  \item[$\bullet$]  The Bernstein center ${\mathcal Z} = {\mathcal Z}(\Gk )$ of $\Gk$ is a commutative algebra (with unity).  Its geometrical realization is as the algebra of $\Gk$-invariant  essentially compact distributions on $\Gk$.  A distribution is essentially compact if $\forall \, f \in C^{\infty}_{0}(\Gk )$, the convolutions $D \star f$ and $f \star D$ are in $C^{\infty}_{0}(\Gk )$.  If $(\pi , V_{\pi})$ is a smooth representation of $\Gk$, one can, by integration, canonically obtain an algebra representation $\pi_{\mathcal Z} \, : \, {\mathcal Z}(\Gk ) \, \longrightarrow \, \End_{\Gk}(V_{\pi})$.  When $\pi$ is irreducible, by Schur's Lemma, each $\pi_{{\mathcal Z}} (D)$ \, ($D \in {\mathcal Z}$) \, is a scalar.

\smallskip

Let $\widehat{\Gk}^{\myadm}$ denote the smooth dual (equivalence
classes of smooth irreducible representations).  We write the map $\{ \pi \} \rightarrow \pi_{{\mathcal Z}} (D)$ as a map \, $\myInf_{D} : \widehat{\Gk}^{\myadm} \longrightarrow \bC$ \, given by $\myInf_{D} (\{ \pi \} ) = \pi_{\mathcal Z}(D)$.  The smooth dual is naturally topologized with the Fell topology, a non-Hausdorff topology.  There is natural Hausdorff topological quotient $\Omega (\Gk )$ of $\widehat{\Gk}^{\myadm}$.  The map $\myInf_{D}$ factors to a map $\Omega (\Gk ) \longrightarrow \bC$ (that we also denote as $\myInf_{D}$).  The points  of $\Omega ( \Gk )$ can be parametrized as equivalence classes of pairs $(\Mk , \sigma )$, where $\Mk$ is a Levi subgroup of $\Gk$, and $\sigma$ is an irreducible cuspidal representation of $\Mk$.  The equivalence relation comes from the adjoint action on the Levi component.  The equivalence class of a pair $(\Mk , \sigma )$ is denoted $[\Mk , \sigma ]$.  For a fixed Levi subgroup $\Mk$, the complex group $X_{\myun}(\Mk )$ of unramified characters acts on the set of elements $[\Mk , \sigma ]$ by twisting the representation component (see [{\reBna}:ChII]), and the orbit (denoted $\Omega ([\Mk , \sigma ])$) is a Bernstein component.  Thus,   $\Omega (\Gk )$ is partitioned into Bernstein components.  The component $\Omega ([\Mk , \sigma ])$ inherits a complex algebraic structure from $X_{\myun}(\Mk )$. The restriction $(\myInf_{D})_{|_{\Omega}}$ of the function $\myInf_{D}$ to a Bernstein component $\Omega = \Omega ([\Mk , \sigma ])$ is a regular function.  Let ${\mathcal R}(\Omega )$ denote the algebra of regular functions on $\Omega$.  In [{\reBD}], it is shown that the map $\myFT_{\Omega} \, : \,  {\mathcal Z}(\Gk ) \rightarrow {\mathcal R}(\Omega )$ defined as $\myFT_{\Omega} (D) = (\myInf_{D})_{|_{\Omega}}$ is a surjective algebra homomorphism, and there is an ideal ${\mathcal I}_{\Omega}$ of ${\mathcal Z}(\Gk )$ so that $\myFT_{\Omega}$ on ${\mathcal I}_{\Omega}$ is an isomorphism, while $\myFT_{\Omega}$ on ${\mathcal I}_{{\Omega}'}$ ($\Omega \neq \Omega'$) is zero.  So, ${\mathcal Z}(\Gk )$ is a product of the ideals  ${\mathcal I}_{\Omega}$.  The unique element $P(\Omega ) \in {\mathcal Z}(\Gk)$ satisfying $(\myInf_{P(\Omega )})_{|_{\Omega'}} = \delta_{\Omega, \Omega'}$ is called the projector of the Bernstein component $\Omega$.   At the time (1980's) extremely little was known explicitly about the  distribution  $P(\Omega )$.  The most illuminating result at that time was a 1976 result of Deligne [{\reDna}].  Suppose $\Gk$ has compact center.  Then, Deligne's result is: \ The support of the character of an irreducible cuspidal representation is in the set of compact elements (those elements which belong to a compact subgroup) of $\Gk$.  This was extended by Dat [{\reDta}] in 2003 to the statement that the projector $P(\Omega )$ of a Bernstein component $\Omega$ has support in the compact elements of $\Gk$.
\end{itemize}

\medskip

In the 1990's, exploitation of the Bruhat--Tits building achieved advances in two directions.

\begin{itemize}
\item[$\bullet$] Moy--Prasad [{\reMPa},{\reMPb}] used points in $\ScptB$ to define subgroups of $\Gk$ and lattices of $\fkg \, = \, {\text{\rm{Lie}}}(\Gk )$ which satisfy descent properties.  In particular, these subgroups and lattices allow one to attach to any irreducible smooth representation $(\pi , V_{\pi} )$ a non-negative rational number $\rho (\pi )$ called the depth.  The application of the parabolic induction functor or the Jacquet functor to an irreducible representation, yields a representation whose constituents all have the same depth as the input.   Thus, all the irreducible representations attached to a Bernstein component $\Omega$ have the same depth, i.e., one can define the depth $\rho (\Omega )$ of a component $\Omega$.   It is clear from their definitions that the Moy--Prasad groups and lattices  are $\Gk$-equivariant objects of $\ScptB$.  
  
\medskip
  
\item[$\bullet$]  Schneider--Stuhler [{\reSSa}] attached to a smooth representation $(\pi , V_{\pi} )$ a $\Gk$-equivariant coefficient system $\gamma_{e}(V_{\pi})$.  To a facet $F$ of the building with parahoric subgroup $U_{F}$, and positive integer $e$, they define a subgroup $U_{F,e}$ (which, if $y$ is a generic point of $F$, is in fact the Moy--Prasad group $\Gk_{y,e}$).  The coefficient system is $\gamma_{e}(V_{\pi})(F) \, := \, V^{U_{F,e}}_{\pi}$.  The space of global sections with compact support in the facets of a given dimension $i$ is a projective smooth representation of $\Gk$.  Schneider--Stuhler used the standard boundary map to get a complex, and under suitable circumstances they proved the important result that this complex is a projective resolution of $V_{\pi}$.   
\end{itemize}

\smallskip

\noindent{No} serious attempt was made in the 1990's to synthesize these two directions together.

\bigskip

An important development made by Meyer--Solleveld [{\reMS}] in 2010 was to replace the coefficient systems of Schneider--Stuhler with idempotent operators $e_{\sigma}$ ($\sigma$ a facet in $\ScptB$) in ${\text{\rm{End}}}_{\bC}(V_{\pi})$.  The situation of Schneider-Stuhler can be recovered from Meyer--Solleveld by taking $e_{\sigma}$ to be the idempotent which projects to the space of $U_{\sigma , e}$-fixed vectors.  A key aspect of their idempotent approach is that the chain complex on $\ScptB$ attached to the idempotents has the property that its restriction to a finite polysimplicial convex subcomplex $\Sigma$ of $\ScptB$ is a resolution of the vector space ${\underset {x \in \Sigma^{\text{\rm{o}}} } \sum  } \ e_{x} (V_{\pi})$, where $\Sigma^{\text{\rm{o}}}$ is the set of vertices in $\Sigma$.  Most importantly, under certain assumptions on the system of idempotents (see [{\reMS}]),  they showed the operator \, ${\underset {\sigma \in \Sigma} \sum } \ (-1)^{\dim (\sigma )} \, e_{\sigma}$ \, is idempotent and projects $V_{\pi}$ to   ${\underset {x \in \Sigma^{\text{\rm{o}}} } \sum } \ e_{x} (V_{\pi})$. 

\bigskip

Work of Bezrukavnikov--Kazhdan--Varshavsky [{\reBKV}] in 2015 linked the Schneider-Stuhler and Meyer-Solleveld theme to Bernstein projectors.  For ease of exposition of their work, we assume $\Gk$ is absolutely quasisimple (see [{\reBKV}] for their more general situation).  They modified the Meyer-Solleveld approach:
\begin{itemize}
\item[(i)]  They replaced the idempotents in End$(V_{\pi})$ with idempotents $e_{F} = {\frac{1}{\meas (\Gk_{F,r^{+}})}} 1_{\Gk_{F,r^{+}}}$ in the Hecke algebra ${\mathcal H}(\Gk )$ \ ($C^{\infty}_{c}(\Gk )$ together with a choice of Haar measure).
\smallskip  
\item[(ii)] They considered an increasing family $\Sigma_{n}$ \, ($n \in {\bN}$) \, of finite convex subcomplexes whose union is the entire building.
\end{itemize}
\noindent{The} resulting idempotents  ${\underset {F \, \in \, \Sigma_{n}} \sum } \ (-1)^{\dim (F )} \, e_{F}$ have limit the depth $r$ Bernstein projector 
$$
P_{r} \ = {\underset {\rho (\Omega ) \le r } \sum } \ P(\Omega ) \ , 
$$
\noindent{and} furthermore, as a distribution, $P_{r}$ has a presentation as an Euler-Poincar{\'{e}} sum $P_{r} = {\underset {F \, \subset \, \ScptB} \sum } \ (-1)^{\dim (F)} e_{F}$.

\bigskip

Here, under the condition that the $\mk$-group $\mG$ is absolutely quasisimple, we further develop the new direction of [{\reBKV}].  We establish an Euler--Poincar{\'{e}} presentation of the projector for an arbitrary Bernstein component of depth zero.   The condition that $\mG$ is absolutely quasisimple has the simplifying convenience that the Bruhat--Tits building $\ScptB (\Gk)$ ($\Gk = \mG (\mk )$) is a simplicial complex. 

\smallskip

When $\Lk$ is a Levi subgroup of $\Gk$, let $\ScptB_{\Gk}(\Lk )$ to be the union of the apartments $\ScptA (\Sk )$ as $\Sk$ runs over the maximal split tori in $\Lk$.  The space $\ScptB_{\Gk}(\Lk )$ is the extended building of $\Lk$.  Suppose $\Omega ([\Mk , \pi ])$ is a depth zero Bernstein component.   It is known from [{\reMPb}] that there exists a pair consisting of:
\begin{itemize}
\item[(i)]  a facet $F$ in $\ScptB_{\Gk}(\Mk )$ satisfying $(\Mk \cap \Gk_{F})/(\Mk \cap \Gk^{+}_{F}) = \Gk_{F}/\Gk^{+}_{F}$, 
\item[(ii)]  an irreducible representation $\sigma$ of $\Mk_{F} := (\Mk \cap \Gk_{F})$ inflated from a cuspidal representation of the finite field group $(\Mk \cap \Gk_{F})/(\Mk \cap \Gk^{+}_{F}) = \Gk_{F}/\Gk^{+}_{F}$,
\end{itemize}
\noindent{so} that $\pi = {\text{\rm{c-Ind}}}^{\Mk}_{\Fk_{F}}(\tau )$, where $\Fk_{F}$ is the normalizer subgroup $N_{\Mk}(\Mk_{F} )$, and $\tau$ is an extension of $\sigma$.  Here, $\Gk^{+}_{F}$ is the maximal normal pro-p-subgroup of $\Gk_{F}$.  If $y$ is a generic point of $F$, so $\Gk_{F} = \Gk_{y,0}$, then $\Gk^{+}_{F} =  \Gk_{y,0^{+}}$.  The relation $(\Mk \cap \Gk_{F})/(\Mk \cap \Gk^{+}_{F}) = \Gk_{F}/\Gk^{+}_{F}$ means $\sigma$ is also canonically a representation of $\Gk_{F}$.  

\smallskip

Let $\Mk \Vk$ be a parabolic subgroup containing $\Mk$, and set $V_{\kappa} = {\text{\rm{Ind}}}^{\Gk}_{\Mk \Vk} (\pi )$.  If $E$ is any facet of $\ScptB$, it follows from [{\reMPb}] that a necessary and sufficient condition for the invariants $V^{\Gk^{+}_{E}}_{\kappa}$ to be nonzero is the existence of a facet $F'$ associate to $F$ \ (a facet $F'$ is associate to $F$ if there exists $g \in \Gk$, so that $(\Gk_{F'} \cap \Gk_{gF})$ surjects onto both $\Gk_{F'}/\Gk^{+}_{F'}$ and $\Gk_{gF}/\Gk^{+}_{gF}$) which contains $E$.  We define idempotents as follows:
\smallskip
$$
e_{E} \ = \ 
\begin{cases}
\quad 0  &{\text{\rm{{\hskip 0.37in}when $V^{\Gk^{+}_{E}}_{\kappa} \, = \, \{ \, 0 \, \}$ ,}}} \\
\ &\ \\
\quad {\frac{1}{\meas ({\Gk^{+}_{E}})}} \, {\underset {\rho} \sum} \, \deg ( \rho ) \, \Theta_{\rho} &
{{\text{\rm{$\begin{array} {l}
{\text{\rm{{\hskip 0.30in}when $V^{\Gk^{+}_{E}}_{\kappa} \, \neq \, \{ \, 0 \, \}$.  The sum is over }}}  \\
{\text{\rm{{\hskip 0.560in}$\rho \in \widehat{{\Gk_{E}}/{\Gk^{+}_{E}}}$ appearing in  $V^{\Gk^{+}_{E}}_{\kappa}$}}}
  \end{array}$}}}}
\end{cases}
$$
 
\noindent{The} $\rho$ which appear belong to a block.  We call $e_{E}$ the Peter--Weyl idempotent.  Clearly this defines a $\Gk$-equivariant system of idempotents on $\ScptB$.   The first and third authors, established in earlier unpublished work that \ $e_{E}  = P (\Omega([\Mk, \pi]) ) \star e_{\Gk^{+}_{E}}$.  Once the $\Gk$-equivariant system of idempotents is in hand, it remains to show (see Theorem \eqref{generaldepthzerotheorem}, Corollary \eqref{generaldepthzerocorollary}, and Theorem \eqref{maindepthzero}):

\begin{thm*} \    Suppose $\mG$ is a connected absolutely quasisimple $\mk$-group.  Let $\Gk = \mG (\mk )$, and let $\ScptB = \ScptB (\Gk )$ be the Bruhat--Tits building.   Suppose $F$ is a facet of $\ScptB$, and $\sigma$ is the inflation to $\Gk_{F}$ of an irreducible cuspidal representation of $\Gk_{F}/\Gk^{+}_{F}$.  Take $\tau \in {\mathcal E}(\sigma)$ as above, and define a $\Gk$-equivariant system of idempotents.  \ Then,

\begin{itemize}
\item[$\bullet$] \ The alternating sum
$$
 P \ = \  {\underset {L \subset \ScptB (\Gk) } \sum} (-1)^{\dim (L)} e_{\tau , L} 
$$
\noindent{over} the facets of $\ScptB (\Gk)$ defines a $\Gk$-invariant essentially compact distribution.
\smallskip
\item[$\bullet$] With Levi subgroup $\Mk$ defined as above, the distribution $P$ is the projector to the Bernstein component of $(\Mk , {\text{\rm{c-Ind}}}^{\Gk}_{{\mathcal F}_{F}} (\tau ))$.
\end{itemize}

\end{thm*}

\smallskip

We note, for the Iwahori component (smooth irreducible representations with nonzero Iwahori-fixed vectors), the (Iwahori) Peter--Weyl idempotent $e_{F}$ of a facet $F$ is the sum of the character idempotents of those irreducible representations of the finite field group $\Gk_{F}/\Gk^{+}_{F}$ which have a nonzero Iwahori fixed vector, i.e., a Borel fixed vector.

\medskip 

We sketch the argument to show the Euler-Poincar{\'{e}} infinite sum defines an essentially compact distribution.  We fix a chamber $C_{0}$, and define the convex ball $\myBall (C_{0},m)$ to be the simplicial subcomplex which is the union of all chambers whose Bruhat length from $C_{0}$ is at most $m$.  The union of these balls obviously exhaust $\ScptB$, and $\myBall (C_{0},(m+1))$ is obtained by adding chambers of Bruhat length $(m+1)$ to $\myBall (C_{0},m)$.  If $D$ has Bruhat length $(m+1)$, let ${\mathcal C}(D)$ be the set of facets of $D$ which are not already in $\myBall (C_{0},m)$.   We note that if $k$ is the number of faces of $D$ in ${\mathcal C}(D)$, then $\# ({\mathcal C}(D)) = 2^{k}$.  If $J$ is any open compact subgroup of $\Gk$, we show the convolution
$$
{\text{\rm{Con}}} = \big( {\underset {E \in {\mathcal C}(D)} \sum } \ (-1)^{\dim (E)} \, e_{E} \ \big) \ \star \ e_{J}
$$
\smallskip
\noindent{vanishes} once $m$ is sufficiently large, say $m \ge N$.  Hence,
$$
{\underset {E \in \myBall (C_{0},n)} \sum } \ (-1)^{\dim (E)} \, e_{E} \ \star \ e_{J} \ = \ {\underset {E \in \myBall (C_{0},N)} \sum } \ (-1)^{\dim (E)} \, e_{E} \ \star \ e_{J} \ , \ \ \ {\text{\rm{for all $n \ge N$.}}}  
$$
\noindent{We} deduce that the infinite Euler-Poincar{\'{e}} sum defines a $\Gk$-invariant essentially compact distribution $P$.  We then establish $P$ is the projector $P(\Omega )$ to the component $\Omega$.  

\medskip

We briefly explain here why the convolution ${\text{\rm{Con}}}$ vanishes when $m$ is sufficiently large.  In the set ${\mathcal C}(D)$, there is a minimal facet $D_{+}$ contained in all the other facets.  This means $\Gk_{D_{+}} \supset \Gk_{E}  \supset \Gk^{+}_{E} \supset \Gk^{+}_{D_{+}}$ holds for all $E \in {\mathcal C}(D)$, and consequently $\Gk_{E}/\Gk^{+}_{D_{+}}$ is a parabolic subgroup in  $\Gk_{D_{+}}/\Gk^{+}_{D_{+}}$.  A convolution vanishing result is established in the finite field group  $\Gk_{D_{+}}/\Gk^{+}_{D_{+}}$, which when the Bruhat length is sufficiently large implies the vanishing of ${\text{\rm{Con}}}$.   

\medskip

We briefly outline the presentation of results.  In section \eqref{apartment}, we introduce notation, and prove preliminary results on facets in the Bruhat--Tits building.   A key result (Proposition \eqref{keypermissible}) on facets is proved in the last subsection.   Section \eqref{harishchandra} is a review of basic results on representations of connected reductive groups over a finite field that follow from Harish-Chandra's philosophy of cusp forms [{\reHCb}].  In sections  \eqref{iwahoribernstein} and \eqref{generaldepthzero} we prove the main result when $\mG$ is split.  In section \eqref{nonsplit}, we indicate the modifications that adapt the proofs of sections \eqref{iwahoribernstein} and \eqref{generaldepthzero} to the nonsplit setting.

\medskip

In the appendix, we show our approach is adaptable to also yield the Euler-Poincar{\'{e}} formula of [{\reBKV}] for the projector $P_{r}$, when $r > 0$ is integral.  

\medskip

The evidence provided by the Euler--Poincar{\'{e}} formula for the depth r projector $P_r$ and individual depth zero projectors, leads to the expectation there should be an Euler--Poincar{\'{e}} formula for suitable combinations of Bernstein projectors.  In the extreme case of a single Bernstein component and positive depth, the equivariant data should involve refinements of the unrefined minimal $\SaK$-types of [{\reMPa}, {\reMPb}].

%%%%%%%%%%%%%%%%%%%%%%%%%%%%%%%%%%%%%%%%%%%%%%%%%%%%%%%%%%%%%%%%%%%%%%%%%%%
%%%%%%%%%%%%%%%%%%%%%%%%%%%%%%%%%%%%%%%%%%%%%%%%%%%%%%%%%%%%%%%%%%%%%%%%%%%
%%%%%%%%%%%%%%%%%%%%%%%%%%%%%%%%%%%%%%%%%%%%%%%%%%%%%%%%%%%%%%%%%%%%%%%%%%%
%%%%%%%%%%%%%%%%%%%%%%%%%%%%%%%%%%%%%%%%%%%%%%%%%%%%%%%%%%%%%%%%%%%%%%%%%%%
%%%%%%%%%%%%%%%%%%%%%%%%%%%%%%%%%%%%%%%%%%%%%%%%%%%%%%%%%%%%%%%%%%%%%%%%%%%

\vskip 0.70in 
 
%%%%%%%%%%%%%%%%%%%%%%%%%%%%%%%%%%%%%%%%%%%%%%%%%%%%%%%%%%%%%%%%%%%%%%%%%
%%%%%%%%%%%%%%%%%%%%%%%%%%%%%%%%%%%%%%%%%%%%%%%%%%%%%%%%%%%%%%%%%%%%%%%%%
%%%%%%%%%%%%%%%%%%%%%%%%%%%%%%%%%%%%%%%%%%%%%%%%%%%%%%%%%%%%%%%%%%%%%%%%%

\section{Notation, review and results on facets in the Bruhat--Tits building}\label{apartment}

\medskip

\subsection{Notation} \quad Suppose $\mk$ is a non-archimedean local field. 
Denote by ${\mathfrak O}_{\mk}$, $\wp_{\mk}$, and $\bF_{q} = {\mathfrak O}_{\mk}/\wp_{\mk}$ respectively, the ring of integers, prime ideal, and residue field of $\mk$.  Let $\mG$ be a connected reductive linear algebraic group defined over $\mk$.   If $\mH$ is a $\mk$-subgroup of $\mG$, we write $\Hk$ for the group of $\mk$-rational points of $\mH$, e.g., $\Gk = \mG (\mk )$.  For convenience, we assume $\mG$ is $\mk$-split and absolutely quasisimple.  Set $\ell = \myrank (\mG)$.  Let $\ScptB = \ScptB (\Gk )$ be the reduced Bruhat-Tits building of $\Gk$.  Let $\mS$ be a maximal $\mk$-split torus of $\mG$, and let $\Ak = \Ak (\Sk)$ be the apartment associated to $\Sk = \mS(\mk)  \subset \Gk$.  Our hypotheses on $\mG$ (split, quasisimple) mean the apartments are simplicial complexes, and hence $\ScptB$ is too.   The group $\Gk$ acts transitively on the chambers ($\ell$-simplices) of $\ScptB$.    The choice of a hyperspecial point $x_{0} \in \ScptA$ corresponds to the choice of a Chevalley basis for the Lie algebra $\fkg$ of $\Gk$.

\medskip

Let $\Phi = \Phi (\Sk)$ be the roots (of $\Gk)$ with respect to $\Sk$, and $\Psi = \Psi (\Ak)$ the system of affine roots on $\Ak$.  If $\alpha$ (resp.~$\psi$) is a root (resp.~affine root), set $\Uk_{\alpha}$ (resp.~$\Xk_{\psi}$) to be the associated root (resp.~affine root) group.  If $\Phi^{+}$ is any set of positive roots of $\Phi$, let $\Delta$ denote the simple roots subset of $\Phi^{+}$.   

\medskip

\begin{itemize} 
\item[$\bullet$] Fix a Borel subgroup $\mB \supset \mS$ of $\mG$, and let $\Phi_{\Bk}^{+} = \Phi (\Sk , \Bk )$ denote the set of positive roots with respect to $\Bk$, and $\Delta_{\Bk} \, = \, \{ \, \alpha_1 \, , \, \dots \, , \, \alpha_{\ell} \, \}$ the simple roots. 

  \medskip
  
\begin{itemize}
\item[(i)]  Let $\psi_{i}$ \ ($ \, 1 \le i \le \ell$) be the affine roots so that $\mygrad (\psi_{i} ) = \alpha_{i}$, and $\psi (x_{0}) = 0$.

  \smallskip
  
  \item[(ii)] Let $\psi_{0}$ be the affine root so that $\mygrad ( \psi_{0} )$ is the negative of the highest root, and $\psi_{0}(x_{0}) = 1$ (we have normalized the value group [{\reTa}:{\S}0.2] to be $\bZ$). 
\end{itemize}

\item[$\bullet$]  Set 
$$
\aligned
{\mathcal S} \ :&= \ \{ \, x \in \ScptA \ | \ \psi_{i} (x ) > 0 \ , \quad \forall \ \ 1 \le i \le \ell \, \}  \\
&\qquad  {\text{\rm{the positive Weyl chamber in $\ScptA$ with respect to $x_{0}$ and $\Phi^{+}$, }}}  \\
{\mathcal S}_{0} \ :&= \ \{ \, x \in \ScptA \ | \ \psi_{i} (x ) > 0 \ , \quad \forall \ 0 \le i \le \ell \, \} \ . \\
\endaligned
$$

\noindent{Note}  ${\mathcal S}_{0} \subset {\mathcal S}$, and its closure $\overline{{\mathcal S}_{0}}$ is a chamber (affine Weyl Chamber) in $\ScptA$, and with respect to   ${\mathcal S}_{0}$, the sets
$$
\Psi^{+} \ = \ \{ \ \psi \in \Psi ( \ScptA ) \ | \ \psi (x) > 0 \ \ \forall \ x \in {\mathcal S}_{0} \ \} \quad {\text{\rm{, \ and}}}  \quad \Delta_{0} \ = \ \{ \ \psi_{0} , \, \psi_{1} , \, \dots \, , \, \psi_{\ell} \ \} 
$$
\noindent{are} the positive affine roots, and the simple affine roots respectively.  The affine roots $\Psi (\ScptA )$ are integer combinations \ $\psi = \sum \, n_{i} \psi_{i}$ \ satisfying $\mygrad (\psi ) \in \Phi$. 

\smallskip

\item[$\bullet$]  For $x \in \ScptB$, and $r \ge 0$, let $\Gk_{x,r}$ be the Moy-Prasad subgroup associated to $x$ and $r$.   

\smallskip

\item[$\bullet$] We fix a Haar measure on $\Gk$.  If $J$ is an open compact subgroup of $\Gk$, we define $e_{J}$ to be the idempotent:
$$
e_{J}(x) \ := \ {\frac{1}{\meas{(J)}}} 
\begin{cases} 
\ \ 1 &{\text{\rm{if $x \, \in \, J$}}} \\
\ \ 0 &{\text{\rm{otherwise}}} \, . 
\end{cases}
$$

\end{itemize}

\medskip

\subsection{A simplex Lemma} \quad We recall and designate some nomenclature.  An $\ell$-dimensional simplex $D$ is the convex closure of a set ${\text{\rm{Vert}}} = \{ \, v_0, \, v_1, \, \dots \, , \, v_{\ell} \, \}$  of $(\ell + 1)$ points in an affine space so that $v_1 -v_0, \, v_2 -v_0, \, \dots \, , \, v_{\ell} -v_0$ are linearly independent.  For a non-empty subset $K \subset {\text{\rm{Vert}}}$ with $(k+1)$ elements, we designate: 

\begin{itemize}
\item[$\bullet$]   ${\text{\rm{facet}}} (K) \ := \ {\text{\rm{convex closure of $K$}}}$.  It is a {\it $k$-facet} of $D$.
\smallskip  
\item[$\bullet$] The convex set
    $$
    \mypfacet (K) \ := \ {\text{\rm{facet}}} (K) \quad \backslash \quad {\underset {L \subsetneq K} \bigcup}  \ {\text{\rm{facet}}} (L) \ .
    $$
\noindent{The} recess is the  interior of ${\text{\rm{facet}}} (K)$ when $k \ge 1$ and equal to ${\text{\rm{facet}}} (K)$ for $k=0$.  A useful feature of \mypfacets \ is that the simplex $D$ is partitioned by them, and there is a one-to-one correspondence from \mypfacets \ to facets, namely the process of taking the closure.  For convenience, when $E$ is a facet of $D$, we write $\mypfacet (E)$ for the \mypfacet \ whose closure is $E$.
 
\smallskip
\item[$\bullet$]  It is elementary that:
\smallskip
\begin{itemize}
\item[(i)] the number of $k$-facets contained in a $j$-facet \ ($j \ge k$) is $\binom{j+1}{k+1}$,
\smallskip
\item[(ii)]  the total number of facets \ is $2^{\ell +1}-1$.
\end{itemize}

\smallskip

\noindent{A} {\it face} of $D$ is, by definition, the convex closure of $\ell$ points of $V$, i.e., a maximal proper facet of $D$. 

\medskip

\item[$\bullet$] If $F$ is a facet in $\ScptB (\Gk)$, and $y \in \mypfacet (F)$, let $\Gk_{F}$ denote the parahoric subgroup  $\Gk_{y,0}$, and let $\Gk^{+}_{F} = \Gk_{y,0^{+}}$.   

\end{itemize}

\medskip

\begin{lemma}\label{simplexlemma}  Suppose $D$ is an $\ell$-dimensional simplex, and ${\mathcal F}$ is a non-empty collection of faces of $D$.  Set $m = \# ({\mathcal F})$.  Then  

\begin{itemize}

\item[(i)] The union \ ${\mathcal P} \ = \ {\underset {F \, \in \, {\mathcal F}} \bigcup } F$ \ is a simplicial complex inside $D$.  The number of $k$-facets of $D$ in ${\mathcal P}$ is 
{\small{
$$
\sum^{m}_{r=1} \ \ (-1)^{(r-1)} \, \binom{\ell + 1 - r}{k+1}  \, \binom{m}{r} \ .
$$
}}
\medskip

\item[(ii)] Let ${\mathcal C}$ be the facets of $D$ occurring in the complement of ${\mathcal P}$.  Then:

\smallskip

\begin{itemize}
  
\item[(ii.1)] A facet in ${\mathcal C}$ has codimension at most $( \ell + 1 - m )$.  

\smallskip

\item[(ii.2)] The number of facets of codimension $j$ is $\binom{\ell + 1 - m}{j}$. 

\smallskip

\item[(ii.3)] The total number of facets in ${\mathcal C}$ is $2^{\ell + 1 - m}$.  

\end{itemize}

\end{itemize}

\end{lemma}

\medskip

\noindent{Note.} \ (i) \, when ${\mathcal F}$  is all the $\ell + 1$ faces, then  ${\mathcal C}$ consists of $D$,  \ (ii) \, when ${\mathcal F}$  is all but one of the faces, then ${\mathcal C}$ consists two elements  (the remaining face, and $D$ itself), \ (iii) \, when ${\mathcal F} \, = \, \{ \, F \, \}$ is a single face $F$, then ${\mathcal C}$ consists of all the facets not contained in $F$.

\medskip

\begin{proof}   The proof of (i) is based on inclusion and exclusion.  Suppose ${\Sigma}$ is the closure of a $j$-facet in the union ${\mathcal P}$.  For $k \le j$, the number of $k$-facets in ${\Sigma}$ is $\binom{j+1}{k+1}$.    The intersection of $r$ distinct face closures is the closure of a unique $(\ell - r)$-facet, e.g., a single face is a $(\ell - 1)$-facet.  This $(\ell - r)$-facet has  $\binom{\ell - r +1}{k+1}$ $k$-facets in its closure.  By the principle of inclusion and exclusion, the number of $k$-facets in the union ${\mathcal P}$ is the stated

{\small{
$$
\sum^{m}_{r=1} \ \ (-1)^{(r-1)} \binom{\ell + 1 - r}{k+1}  \, \binom{m}{r} \ .
$$
}}

\noindent{To} prove statement (ii), we consider the sum obtained by extending the index $r$ to $r=0$, i.e., the sum
{\small{
$$
\sum^{m}_{r=0} \ \ (-1)^{(r-1)} \binom{\ell + 1 - r}{k+1}  \, \binom{m}{r} \ .
$$ 
}}
\bigskip

\noindent{This} is $(-1)^{(\ell-k+1)}$ times the coefficient of $x^{(\ell - k)}$ in the power series expansion of $\frac{1}{(1+x)^{(k+2)}}(1+x)^m$.

\medskip

When $m \ge (k+2)$,  the  power series is the polynomial $(1+x)^{((m-2)-k)}$, and the coefficient of $x^{(\ell - k )}$ is zero; so,

{\small{
$$
\binom{\ell +1}{k+1} \ = \ \sum^{m}_{r=1} \ \ (-1)^{(r-1)} \binom{\ell + 1 - r}{k+1}  \, \binom{m}{r} \ ;
$$  
}}

\noindent{In particular}, all $k$-facets of $D$ are in the union ${\mathcal P}$; so none are in ${\mathcal C}$. 

\bigskip

When $m < (k+2)$, the coefficient of $x^{(\ell - k)}$ in the power series expansion of $\frac{1}{(1+x)^{((k+2)-m)}}$
is $(-1)^{(\ell - k )} \, \binom{\ell +1 -m}{k+1-m}$.  Thus, 
{\small{
$$
\sum^{m}_{r=1} \ \ (-1)^{(r-1)} \binom{\ell + 1 - r}{k+1}  \, \binom{m}{r} \ = \ \binom{\ell +1 }{k+1} \ - \ 
\binom{\ell +1 -m}{k+1-m} 
$$
}}
\noindent{Thus,} the number of $k$-facets (of $D$) in ${\mathcal C}$ is  $\binom{\ell +1 -m}{k+1-m} = \binom{\ell +1 -m}{\ell - k }$.  The integer $j = {\ell - k }$ is the codimension. 

\end{proof}

\medskip

Suppose $D$ is a chamber of $\ScptB$, and ${\mathcal F}$ is a non-empty collection of $m$ faces of $D$.   Let  $W \, = \, \{ \, E_{1}, \, E_{2}, \, \dots \, , \, E_{(\ell + 1 - m)} \}$ be the faces of $D$ complementary to the faces in ${\mathcal F}$.  Then, the facets in the complement ${\mathcal C}$ of codimension $j$ can be described as follows: \ Given a subset $Y \subset W$ of $j$ faces, let 
$$
F(Y) \ = \ \big( {\underset {E \in Y} \bigcap} E \, \big) \quad {\text{\rm{a $(\ell - j)$-facet}}} \, . 
$$  
\noindent{When} $Y = \emptyset$, we use the convention $F( \emptyset ) = D$; so, $\Gk_{F(\emptyset )} = \Gk_{D}$.      Clearly, all the $2^{(\ell + 1 - m)}$ facets in ${\mathcal C}$ are obtained in this fashion.  If $( \,\emptyset \, \subset ) \ Y_1 \subset Y_2 \subset W$, then $D \supset F(Y_{1}) \supset F(Y_{2}) \supset F(W)$; so, $\Gk_{D} \subset \Gk_{F(Y_{1})} \subset \Gk_{F(Y_{2})} \subset \Gk_{F(W)}$, and 
$\Gk^{+}_{D} \supset \Gk^{+}_{F(Y_{1})} \supset \Gk^{+}_{F(Y_{2})} \supset \Gk^{+}_{F(W)}$. The quotient $\BFq = \Gk_{D}/\Gk^{+}_{F(W)}$ is a Borel subgroup of the finite field group $\GFq = \Gk_{F(W)}/\Gk^{+}_{F(W)}$.  The parahoric subgroups which fix the  $2^{(\ell + 1 - m)}$ facets of \eqref{simplexlemma} part (ii.3) corresponds to the standard parabolic subgroups of $\GFq$ which contain the Borel subgroup $\BFq$.

\vskip 0.3in

%%%% \noindent{Any}  chamber $D$ of $\ScptA$ can be obtained from $C_{0}$ by a sequence of reflections $r_{\psi} \, ( \, \in W^{\text{\rm{a}}} \, ) $ through the zero sets of affine roots.

%%%%%%%%%%%%%%%%%%%%%%%%%%%%%%%%%%%%%%%%%%%%%%%%%%%%%%%%%%%%%%%%%%%%%%%%%%%%%
%%%%%%%%%%%%%%%%%%%%%%%%%%%%%%%%%%%%%%%%%%%%%%%%%%%%%%%%%%%%%%%%%%%%%%%%%%%%%
%%%%%%%%%%%%%%%%%%%%%%%%%%%%%%%%%%%%%%%%%%%%%%%%%%%%%%%%%%%%%%%%%%%%%%%%%%%%%
%%%%%%%%%%%%%%%%%%%%%%%%%%%%%%%%%%%%%%%%%%%%%%%%%%%%%%%%%%%%%%%%%%%%%%%%%%%%%

\subsection{Bruhat height} \quad  We fix an apartment $\ScptA = \ScptA (\Sk )$ of the building, and a chamber $C_{0}$ in $\ScptA$.  Let $\Sk_{\text{\rm{c}}}$ denote the maximal bounded (compact) subgroup of $\Sk$.  We recall that the normalizer $\Nk = N_{\Gk}( \Sk )$ of $\Sk$ acts on $\ScptA$, with action kernel equal to $\Sk_{\text{\rm{c}}}$, the maximal bounded (compact) subgroup of $\Sk$, i.e., the action factors through the extended affine Weyl group
$$
W^{\text{\rm{a}}} := N_{\Gk}(\Sk )/{\Sk_{\text{\rm{c}}}} \ .
$$
\noindent{For} $n \in W^{\text{\rm{a}}}$, 
let $\ell_{\myBruhat}(n)$ denote the Bruhat length of $n$.  If $D = n.C_{0}$, we define the {\it Bruhat height of $D$ with respect to $C_{0}$} as:
$$
\myht_{C_{0}}(D) \ := \ \ell_{\myBruhat}(n) \ .
$$

\medskip
If $\psi$ is a affine root, we set the associated affine hyperplane as{\,}:
\begin{equation}\label{defnaffinehyperplane}
H_{\psi}\ := \ {\text{\rm{the zero locus (an affine hyperplane) of $\psi$}}} \, .
\end{equation}
\noindent{We} also use the notation $H_{\pm \psi }$ for this affine hyperplane.

\medskip

For any facet $F$ (not necessarily a face) of $D$, we set

\vskip -0.10in

\begin{equation}\label{facetzeroaffineroots}
{\Psi} (F) \ := \ {\text{\rm{set of affine roots $\psi$ which vanishes on $F$}}} \ .
\end{equation}

\smallskip

\noindent{A} face $F$  (facet of dimension $(\ell - 1)$) of $D$ is in the zero hyperplane set of a unique pair of affine roots $\pm \psi$, i.e., ${\Psi} (F) = \{ \pm \psi \}$.    When $F$ is a face of $D$, set

\vskip -0.00in

\begin{equation}
  \aligned
  r_{F} \, ( \, \in \, W^{\text{\rm{a}}} \, ) \ &= \ {\text{\rm{the affine reflection across the hyperplane $H_{\pm \psi}$}}} \, ,  \\
  \myopp_{F} (D) \ &= \ {\text{\rm{the chamber obtained from $D$ by reflection across $F$.}}}
  \endaligned
\end{equation}

\vskip 0.10in

\noindent{Any}  chamber $D$ of $\ScptA$ can be obtained from $C_{0}$ by a composition of reflections $r_{F} \, ( \, \in W^{\text{\rm{a}}} \, ) $.  \ We find it useful to define for a pair of roots $\{ \pm \gamma \} \subset \Phi$, the $\{ \pm \gamma \}$-height of a chamber $D$ with respect to $C_{0}$ as:

\begin{equation}\label{bruhatheightpieces}
\myht^{\pm \gamma}_{C_{0}}(D) \ := \ \begin{cases}
\begin{tabular}{p{3.5in}}
The number of affine hyperplanes $H_{\psi}$ satisfying: \\ \quad (i) \ $\mygrad (\psi) = \pm \gamma$, and \\ \quad (ii) $H_{\psi}$ separates $C_{0}$ and $D$.  \\
\end{tabular}
\end{cases}
\end{equation}

\noindent{Then}, 

\begin{lemma}\label{bruhatheight} Fix a base chamber $C_{0}$ in $\ScptA$.  If $D$ is a chamber in $\ScptA$, then the minimum number of affine reflections needed to take $C_{0}$ to $D$ is the sum 

$$
\myht_{C_{0}} \, (D) \ = \ {\underset { \{ \pm \gamma \} } {\sum}} \ \myht^{\pm \gamma}_{C_{0}}(D) 
$$

\noindent{over} all pairs of roots in $\Phi$.  
\end{lemma}

\begin{proof} \quad   In $W^{\text{\rm{a}}}$,  let $X = {\Sk}/{\Sk_{\text{\rm{c}}}}$, and $W = N_{\Gk}(\Sk )/{\Sk}$, the finite Weyl group of $\Gk$.  If we take a hyperspecial point
  $x_{0}$ in $\ScptA$, then every element $w \in W$, has a representative $n_{w} \in (\Gk_{x_{0},0} \cap N_{\Gk}(\Sk ))$, which is unique modulo $\Sk_{\text{\rm{c}}}$.  Let $W_{x_{0}} \, = \, \{ \, n_w \ | \ w \in W \, \}$ be a set of such representatives of $W$.   Then, any $n \in N_{\Gk}(\Sk )$ can be written as $n = x n_{w}$ with $x \in {\Sk}$ and $n_{w} \in W_{x_{0}}$. 

  \medskip

  Suppose $D$ is a chamber of $\ScptA$.  Take $n = x n_{w} \, \in \, N_{\Gk}(\Sk )$ so that $D = (x n_{w}) \, C_{0}$.  Fix a positive system of roots $\Phi^{+} \subset \Phi$.  Denote the negative roots as $\Phi^{-}$.  By Proposition 1.23 of [{\reIM}], the Bruhat length $\myBruhat (n)$ of $n$ is
  \begin{equation}
    \aligned
    \myBruhat (n) \ = \ \myBruhat (\, x \, n_{w} \, ) \ &= \ {\underset { {\text{\rm{\tiny{$\begin{array} {c}
    {\alpha \in \Phi^{+}} \\
    {w^{-1}(\alpha ) \in {\Phi^{+}} } 
  \end{array}$}}}}} 
      \sum } | \, \langle x , \alpha \rangle \, | \ \ + \ \ 
    {\underset { {\text{\rm{\tiny{$\begin{array} {c}
    {\alpha \in \Phi^{+}} \\
    {w^{-1}(\alpha ) \in {\Phi^{-}} } 
  \end{array}$}}}}}
      \sum } | \, \langle x , \alpha \rangle \ - \ 1 \, | 
    \endaligned
  \end{equation}

  \noindent{For} $\alpha \in \Phi^{+}$, the geometric meaning of the function
  \smallskip
  $$
  \aligned
  {\text{\rm{Ht}}}^{\alpha} ( \, xn_{w} \, ) \ :=  \begin{cases}
    \ \ | \, \langle x , \alpha \rangle \, | &{\hskip 0.40in} {\text{\rm{when  ${w^{-1}(\alpha ) \in {\Phi^{+}} }$ }}} \\
    \ \\
    \ \ | \, \langle x , \alpha \rangle \ - \ 1 \, | &{\hskip 0.40in} {\text{\rm{when  ${w^{-1}(\alpha ) \in {\Phi^{-}} }$ }}} \\
    \end{cases}
\endaligned
$$

\medskip

\noindent{is} precisely the function $\myht^{\pm \alpha}_{C_{0}}$.  The Lemma follows.

\end{proof}

We note that if $\ScptA'$ is another apartment of $\ScptB$ containing $C_{0}$,  then the two height functions agree on chambers in the intersection  $\ScptA \cap \ScptA'$.  Thus, there is a unique extension of the height function $\myht_{C_{0}}$ to all the chambers of $\ScptB$.  

\medskip

%%%%%%%%%%%%%%%%%%%%%%%%%%%%%%%%%%%%%%%%%%%%%%%%%%%%%%%%%%%%%%%%%%%%%%%%%%%%%
%%%%% OLD POSITION FOR HEIGHT OF CHAMBERS FIRGURE FOR C2
%%%%%  Figure ONE 
%%%%%%%%%%%%%%%%%%%%%%%%%%%%%%%%%%%%%%%%%%%%%%%%%%%%%%%%%%%%%%%%%%%%%%%%%%%%%
%%%%%%%%%%%%%%%%%%%%%%%%%%%%%%%%%%%%%%%%%%%%%%%%%%%%%%%%%%%%%%%%%%%%%%%%%%%%%
\begin{comment}
\end{comment}
%%%%%%%%%%%%%%%%%%%%%%%%%%%%%%%%%%%%%%%%%%%%%%%%%%%%%%%%%%%%%%%%%%%%%%%%%%%%%
%%%%%%%%%%%%%%%%%%%%%%%%%%%%%%%%%%%%%%%%%%%%%%%%%%%%%%%%%%%%%%%%%%%%%%%%%%%%%

We observe that if $F$ is a face of a chamber $D \subset \ScptA$, and ${\Psi} (F) = \{ \pm \psi \}$, then for $z \in \mypfacet (F)$ and small positive $\epsilon$, the point $z + \epsilon \, \mygrad (\psi )$ is either in $D$ or $\myopp_{F} (D)$.  With respect to $D$, we define:

\begin{equation}\label{outward-inward}  
\aligned
{\text{\rm{(i)}}} \ \ &{\text{\rm{$\psi$ is {\it outwards oriented} \ if $z \, + \, \epsilon \, \mygrad (\psi )$ is in $\myopp_{F} (D)$ for small positive $\epsilon$}}} \\ 
{\text{\rm{(ii)}}} \ \ &{\text{\rm{$\psi$ is {\it inwards oriented} \ if $z \, + \, \epsilon \, \mygrad (\psi )$ is in $D$ for small positive $\epsilon$}}} \, .
\endaligned
\end{equation}

\medskip

Suppose $D$ is a chamber in the apartment $\ScptA$.  Set
\begin{equation}\label{chamberinout}
\aligned
c(D) \ &= \ \{ {\text{\rm{ $F$ a face of $D$ \ | \ $\myht_{C_{0}}(   \myopp_{F}(D) ) \, = \, \myht_{C_{0}} ( D) \, + \, 1$ \ \} }}} \\
p(D) \ &= \ \{ {\text{\rm{ $F$ a face of $D$ \ | \ $\myht_{C_{0}}(   \myopp_{F}(D) ) \, = \, \myht_{C_{0}} ( D) \, - \, 1$ \ \} . }}} \\
\endaligned 
\end{equation}

\noindent{Clearly} any face of $D$ belongs to either $c(D)$ or $p(D)$.  Mnemonically, the set $p(D)$ (resp.~$c(D)$)  is the set of `parent' or `inward' (resp.~`child' or `outward')  faces of the chamber $D$.

\medskip

\begin{prop}\label{childprop}  Suppose $D$ is a chamber of an apartment $\ScptA$, and $F$ is a face of $D$, and $\myht_{C_{0}} ( \myopp_{F}( D ) ) \, = \, \myht_{C_{0}} ( D ) \, + \, 1$.    Write ${\Psi} (F) \, = \, \{ \pm \psi \}$ (notation \eqref{facetzeroaffineroots}), and choose $\psi$ to be outward oriented for $D$ (notation \eqref{outward-inward}).  Set $\alpha = \mygrad (\psi )$.    Then for any $y \in \mypfacet (D)$ and $x \in \mypfacet (F)$, 

$$
\Gk_{y,0^{+}} \cap \Uk_{\alpha} \ = \ \Gk_{x,0^{+}} \cap \Uk_{\alpha} \ .
$$   

\end{prop}

\begin{proof} \quad One verifies both subgroups equal the affine root subgroup $\Xk_{\psi+1}$.
\end{proof}

\smallskip

%%%%%%%%%%%%%%%%%%%%%%%%%%%%%%%%%%%%%%%%%%%%%%%%%%%%%%%%%%%%%%%%%%%%%%%%%%%%%
%%%%% NEW POSITION FOR HEIGHT OF CHAMBERS FIRGURE FOR C2
%%%%%  Figure ONE 
%%%%%%%%%%%%%%%%%%%%%%%%%%%%%%%%%%%%%%%%%%%%%%%%%%%%%%%%%%%%%%%%%%%%%%%%%%%%%
%%%%%%%%%%%%%%%%%%%%%%%%%%%%%%%%%%%%%%%%%%%%%%%%%%%%%%%%%%%%%%%%%%%%%%%%%%%%%
%\begin{comment}
%\vfill \eject
\begin{figure}[ht]
\includegraphics[width=10cm]{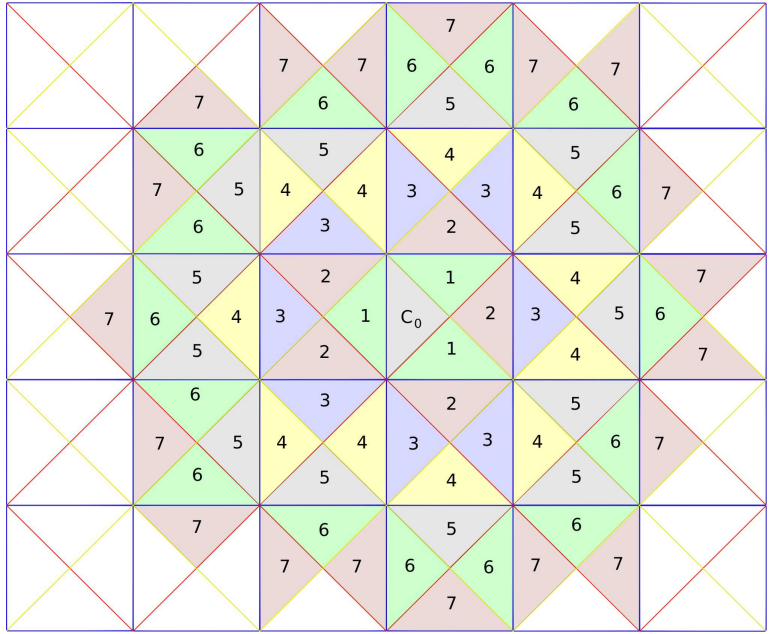}
\caption{Illustration of the height of chambers in an apartment when $\Gk$ is of type ${\text{\rm{C}}}_2$.}
    \label{fig:figure1}
\end{figure}
%\end{comment}
%%%%%%%%%%%%%%%%%%%%%%%%%%%%%%%%%%%%%%%%%%%%%%%%%%%%%%%%%%%%%%%%%%%%%%%%%%%%%
%%%%%%%%%%%%%%%%%%%%%%%%%%%%%%%%%%%%%%%%%%%%%%%%%%%%%%%%%%%%%%%%%%%%%%%%%%%%%

\vskip 0.3in

\subsection{Chamber based sectors in an apartment} \quad  We fix a maximal torus $\Sk$ and its associated apartment $\ScptA = \ScptA (\Sk )$.  If  $\psi$ is an affine root, set 
\smallskip
\begin{equation}
\aligned
H_{\psi > 0} \ :&= \ \{ \ x \in \ScptA \ | \ \psi (x) > 0 \ \} \quad {\text{\rm{and}}} \quad H_{\psi \ge 0} \ :&= \ \{ \ x \in \ScptA \ | \ \psi (x) \ge 0 \ \} \ .
\endaligned
\end{equation}
\medskip
\noindent{Suppose} $\Phi^{+}$ is a choice of positive roots in $\Phi = \Phi (\Sk )$, and $C_{0}$ is a chamber.  Set:

\begin{equation}\label{csector}
\aligned
S(C_{0},\Phi^{+}) \ :&= {\underset
{{\text{\rm{$\begin{array} {c}
  {\mygrad (\psi ) \in \Phi^{+}} \\ 
  {C_{0} \subset H_{\psi \ge 0}}
  \end{array}$}}}}  \bigcap }  H_{\psi \ge 0}  \, .
\endaligned
\end{equation}

\noindent{We} call such a set the $C_{0}$-chamber based sector with respect to the positive roots $\Phi^{+}$.  

\bigskip

%\vfill \eject

%%%%%%%%%%%%%%%%%%%%%%%%%%%%%%%%%%%%%%%%%%%%%%%%%%%%%%%%%%%%%%%%%%%%%%%%%%%%%
%%%%%  Figure TWO 
%%%%%%%%%%%%%%%%%%%%%%%%%%%%%%%%%%%%%%%%%%%%%%%%%%%%%%%%%%%%%%%%%%%%%%%%%%%%%
%%%%%%%%%%%%%%%%%%%%%%%%%%%%%%%%%%%%%%%%%%%%%%%%%%%%%%%%%%%%%%%%%%%%%%%%%%%%%
%%%%%%%%%%%%%%%%%%%%%%%%%%%%%%%%%%%%%%%%%%%%%%%%%%%%%%%%%%%%%%%%%%%%%%%%%%%%%
%%%%% OLD POSITION FOR HEIGHT OF CHAMBERS FIRGURE FOR C2
%%%%%  Figure TWO 
%%%%%%%%%%%%%%%%%%%%%%%%%%%%%%%%%%%%%%%%%%%%%%%%%%%%%%%%%%%%%%%%%%%%%%%%%%%%%
%%%%%%%%%%%%%%%%%%%%%%%%%%%%%%%%%%%%%%%%%%%%%%%%%%%%%%%%%%%%%%%%%%%%%%%%%%%%%
\begin{comment}
\end{comment}
%%%%%%%%%%%%%%%%%%%%%%%%%%%%%%%%%%%%%%%%%%%%%%%%%%%%%%%%%%%%%%%%%%%%%%%%%%%%%
%%%%%%%%%%%%%%%%%%%%%%%%%%%%%%%%%%%%%%%%%%%%%%%%%%%%%%%%%%%%%%%%%%%%%%%%%%%%%

\begin{prop} \quad The chamber based sector $S(C_{0}, \Phi^{+})$ is the set of chambers which can be obtained from $C_{0}$ by repeated application of affine reflections $s_{\psi}$ with $\mygrad (\psi) \, \in \Phi^{+}$.  

\end{prop}

%%%%%%%%%%%%%%%%%%%%%%%%%%%%%%%%%%%%%%%%%%%%%%%%%%%%%%%%%%%%%%%%%%%%%%%%%%%%%
%%%%%  Figure TWO 
%%%%%%%%%%%%%%%%%%%%%%%%%%%%%%%%%%%%%%%%%%%%%%%%%%%%%%%%%%%%%%%%%%%%%%%%%%%%%
%%%%%%%%%%%%%%%%%%%%%%%%%%%%%%%%%%%%%%%%%%%%%%%%%%%%%%%%%%%%%%%%%%%%%%%%%%%%%
%%%%%%%%%%%%%%%%%%%%%%%%%%%%%%%%%%%%%%%%%%%%%%%%%%%%%%%%%%%%%%%%%%%%%%%%%%%%%
%%%%% NEW POSITION FOR HEIGHT OF CHAMBERS FIRGURE FOR C2
%%%%%  Figure TWO 
%%%%%%%%%%%%%%%%%%%%%%%%%%%%%%%%%%%%%%%%%%%%%%%%%%%%%%%%%%%%%%%%%%%%%%%%%%%%%
%%%%%%%%%%%%%%%%%%%%%%%%%%%%%%%%%%%%%%%%%%%%%%%%%%%%%%%%%%%%%%%%%%%%%%%%%%%%%
%\begin{comment}
%\vfill \eject
\begin{figure}[ht]
\includegraphics[width=10cm]{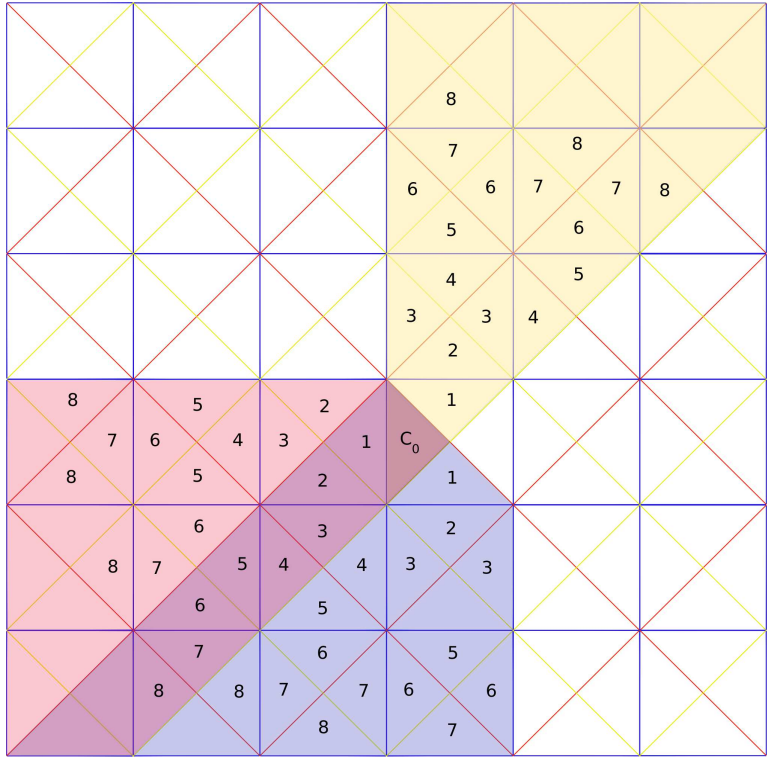}
\caption{Illustration of three chamber based sectors when $\Gk$ is of type ${\text{\rm{C}}}_2$.}
    \label{fig:figure2}
\end{figure}
%\end{comment}
%%%%%%%%%%%%%%%%%%%%%%%%%%%%%%%%%%%%%%%%%%%%%%%%%%%%%%%%%%%%%%%%%%%%%%%%%%%%%
%%%%%%%%%%%%%%%%%%%%%%%%%%%%%%%%%%%%%%%%%%%%%%%%%%%%%%%%%%%%%%%%%%%%%%%%%%%%%

\noindent{We} note:

\begin{itemize}
\item[$\bullet$] The union of the sets $S(C_{0},\Phi^{+})$, as $\Phi^{+}$ runs over positive roots subsets of $\Phi$, is $\ScptA$.
\smallskip
\item[$\bullet$] A chamber $D$ can belong to more than one $S(C_{0},\Phi^{+})$.  This happens when $D$ is near a (Weyl chamber) wall.
\end{itemize}

\medskip

If $D$ is a chamber of $\ScptA$, set  

$$
  \aligned
  {\mathcal R}(C_{0},D) \ :&= \ \{ \ {\text{\rm{$\Phi^{+}$ a set of positive roots}}} \ | \  D \subset S(C_{0},\Phi^{+}) \  \} \\
B(C_{0},D) \ :&= \ {\underset {\Phi^{+} \in {\mathcal R}(C_{0},D)} \bigcap } S(C_{0},\Phi^{+}) \ . \\
\endaligned
$$

\noindent{Alternatively}, we define a subset $\Phi (C_{0} , D)$ of the roots $\Phi$ as follows:

\begin{itemize}

\item[(i)] Suppose $\alpha$ is a root and there is an affine root $\psi$ with $\mygrad (\psi ) = \alpha$, and the zero-hyperplane $H_{\psi} \, (=H_{-\psi})$ separates $C_{0}$ and $D$.  We define $D$ to be $\mygrad (\psi )$ oriented with respect to $C_{0}$, if for $y \in H_{\psi}$ and $x \in \mypfacet (D)$ we have:
\begin{equation}
\langle \, ( \, x \, - \, y \, ) \, , \, \mygrad (\psi ) \, \rangle \ > \ 0 \ .
\end{equation}

\noindent{Obviously}, $D$ is either $\mygrad ( \psi)$  or $\mygrad ( -\psi)$ oriented with respect to $C_{0}$.   Also, $D$ is $\mygrad (\psi )$ oriented if, for any positive $\epsilon$, the point $y \, + \, \epsilon \, \mygrad (\psi )$ and points of $\mypfacet (D)$ are on the same side of the zero hyperplane $H_{\psi}$.  Set $\alpha = \mygrad (\psi )$.   Clearly, there is an affine root $\alpha - k$ with smallest possible $k$ so that $H_{\alpha - (k+1)}$ separates $C_{0}$ and $D$, but $H_{\alpha - k}$ does not, i.e., 
\smallskip
$$
\aligned
H_{(\alpha - k),>0} \ :&= \ \{ \ x \in \ScptA \ | \ (\alpha - k) (x) > 0 \ \} \ 
\endaligned
$$
\medskip
\noindent{contains} $\mypfacet (C_{0})$ and $\mypfacet (D)$, and is the smallest of the affine root half spaces $H_{\psi,>0}$ with $\mygrad (\psi ) = \alpha$ which does so.
  
\smallskip

\item[(ii)] Suppose $\alpha$ is a root so that for any affine root $\psi$ with $\mygrad (\psi ) = \pm \alpha$, the zero-hyperplane $H_{\psi}$ of $\psi$ does not separate $C_{0}$ and $D$.  Here, by replacing $\alpha$ by $-\alpha$ if necessary, there is an affine root $\psi = \alpha - k$ so that $\mypfacet (C_{0})$ and $\mypfacet (D)$ lie in the band
\smallskip
$$
B_{\alpha} \ := \ \{ \ x \in \ScptA \ | \ k \ < \ \alpha (x) \ < (k+1) \ \} \ .
$$    

\end{itemize}

\medskip

\noindent{We} define a canonical set $\Phi( C_{0},D)$ as:  
\smallskip
$$
\aligned
\Phi (C_{0},D) \ &= \ {\text{\rm{ the union of the roots $\mygrad (\psi )$ in (i) }}}  \, . 
\endaligned
$$
\smallskip
\noindent{In} the situation, when there are no roots of type (ii), then the roots $\Phi( C_{0},D)$ are a positive system of roots, and $B(C_{0},D) = S(C_{0},\Phi( C_{0},D))$.

\medskip
We explain the significance of the roots of type (i) and (ii).  Suppose $x \in C_{0}$ and $y \in D$.    Then, as we move from $x$ to $y$ along the line segment $(1-t)x+ty${\,}:

\begin{itemize}

\item[$\bullet$]  For $\alpha \in \Phi( C_{0},S)$, 
\begin{equation}
\aligned
\Gk_{((1-t)x+ty),0^{+}} \cap \Uk_{\alpha} \ &{\text{\rm{ increases from \ $\Gk_{x,0^{+}} \cap \Uk_{\alpha}$ \ to $\Gk_{y,0^{+}} \cap \Uk_{\alpha}$}}} \ , \\
\Gk_{((1-t)x+ty),0^{+}} \cap \Uk_{-\alpha} \ &{\text{\rm{ decreases from \ $\Gk_{x,0^{+}} \cap \Uk_{-\alpha}$ \ to $\Gk_{y,0^{+}} \cap \Uk_{-\alpha}$}}} \, .
\endaligned
\end{equation}

\item[$\bullet$] If $\pm \alpha$ is a root of type (ii), then 
\begin{equation}
\aligned
\Gk_{((1-t)x+ty),0^{+}} \cap \Uk_{\alpha} \ &{\text{\rm{ is constant }}} \ . \\
\endaligned
\end{equation}

\end{itemize}
\smallskip

\noindent{The} roots of type (ii) are the roots of a Levi subgroup $\Lk$ containing $\Sk$.  The roots of type (i), are those of a unipotent radical $\Vk$ of a parabolic subgroup $\Pk \, = \, \Lk \Vk$.  Let $\overline{\Vk}$ denote the opposite radical $\Vk$.  For $x \in \ScptA (\Sk)$, set  
\smallskip
$$
\aligned
\Lk_{x,0^{+}} \ := \ \Gk_{x,0^{+}} \, \cap \, \Lk \quad , \quad \Vk_{x,0^{+}} \ := \ \Gk_{x,0^{+}} \, \cap \, \Vk \quad , \quad {\overline{\Vk}}_{x,0^{+}} \ := \ \Gk_{x,0^{+}} \, \cap \, {\overline {\Vk}} \ \ .
\endaligned
$$
\medskip
\noindent{Then}
$$
\Gk_{x,0^{+}} \  =  \ \Lk_{x,0^{+}} \, \Vk_{x,0^{+}} \, {\overline {\Vk}}_{x,0^{+}}  \qquad {\text{\rm{ in any order}}} \, .
$$

\begin{prop} Suppose $D$ is a chamber of $\ScptA$.  Define the Levi $\Lk$ and $\Vk$ as above.  Let $c(D)$ be the set of child faces of $D$, and suppose $E, F \in {\mathcal C}(c(D))$, so that $F$ is a face of $E$.  Then, for $y \in \mypfacet{(E)}$, and $x \in \mypfacet{(F)}$:
$$
\Vk_{y,0^{+}} \ = \ \Vk_{x,0^{+}} \quad , {\text{\rm{and}}} \quad \Lk_{y,0^{+}} \ = \ \Lk_{x,0^{+}} \ .
$$
\end{prop}

\begin{proof} \quad Apply Proposition \eqref{childprop}.
\end{proof}

\vskip 0.40in

%%%%%%%%%%%%%%%%%%%%%%%%%%%%%%%%%%%%%%%%%%%%%%%%%%%%%%%%%%%%%%%%%
%%%%%%%%%%%%%%%%%%%%%%%%%%%%%%%%%%%%%%%%%%%%%%%%%%%%%%%%%%%%%%%%%
%%%%%%%%%%%%%%%%%%%%%%%%%%%%%%%%%%%%%%%%%%%%%%%%%%%%%%%%%%%%%%%%%
\subsection{A finiteness condition on chambers}  \quad  If $D$ is a chamber, set

\begin{equation}
  \Xi (D) \ := \ {\text{\rm{set of affine root pairs $\pm \psi$ so that $H_{\pm \psi} \cap D$ is a face of $D$}}} \, .
\end{equation}

\medskip

\noindent{Define} a nonempty set of affine root pairs ${\mathcal X} = \{ \, \pm \psi_{1}, \, \dots \, , \pm \psi_{k} \, \}$ to be {\it permissible} if{\,}:

\smallskip

\begin{itemize}
\item[$\bullet$] $\mygrad (\psi_{1}), \, \dots \, , \, \mygrad (\psi_{k})$ are linearly independent.  Here, we have selected one affine root from each pair.  Let{\,}:
\smallskip
\begin{equation}\label{permissiblzero}
  \ScptA \ScptS ({\mathcal X}) \ := \ H_{\pm \psi_{1}}  \, \, \cap \, \cdots  \, \cap \, \ H_{\pm \psi_{k}} \  . 
\end{equation}
\noindent{a} $(\ell - k)$-dimensional affine subspace of $\ScptA$.
\vskip 0.05in
\item[$\bullet$] There exists a chamber $D$ so that each intersection $H_{\psi_{i}} \cap D$ is a face of $D$.  We name this relationship as the chamber  $D$ being {\it incident} with ${\mathcal X}$.  In this situation, the intersection $D \, \cap \, \ScptA \ScptS ({\mathcal X})$ is a $( \ell - k)$-dimensional facet of $D$.  

\end{itemize}

\smallskip

\noindent{We} observe{\,}:
\begin{equation}\label{permissibleone}
\bullet \ \begin{tabular}{p{4.5in}}
If $D$ is a chamber, then any nonempty proper subset  ${\mathcal X} = \{ \, \pm \psi_{i_1} \, , \,\dots \, , \, \pm \psi_{i_k} \, \}$ of $\Xi (D)$  is permissible.  \\
\end{tabular} \hskip 0.22in
\end{equation}

\smallskip

\noindent{When} $D$ is incident with ${\mathcal X}$, let
$$
{\mathcal X}_{D}' \ = \ \{ \, \pm \phi_{k+1}, \, \dots \, , \, \pm \phi_{\ell+1} \, \} \ = \ \Xi (D) \ \backslash \ {\mathcal X}
$$
\noindent{be} the affine root pairs so that the remaining $(\ell + 1 - k)$  faces of $D$ lie in the affine hyperplanes  $H_{\pm \phi_{k+1}}, \, \dots \, , \, H_{\pm \phi_{\ell+1}}$.  The affine subspace $\ScptA \ScptS ({\mathcal X}_{D}')$ has dimension $(k-1)$, and the two subsimplices $( \, D \, \cap \, \ScptA \ScptS ({\mathcal X}) \, )$ and $( \, D \, \cap \, \ScptA \ScptS ({\mathcal X}_{D}') \, )$ of the simplex ${D}$ are opposite subsimplices, i.e., any vertex of ${D}$ lies in precisely one of the two subsimplices.   Define vector spaces{\,}
\medskip
$$
\aligned
\myVect ( \, \ScptA \ScptS ({\mathcal X}) \, ) \ :&= \ {\text{\rm{translations of $\ScptA$ which leave $\ScptA \ScptS ({\mathcal X})$ invariant}}} \\
\myVect ( \, \ScptA \ScptS ({\mathcal X}_{D}') \, ) \ :&= \ {\text{\rm{translations of $\ScptA$ which leave $\ScptA \ScptS ({\mathcal X}_{D}')$ invariant}}} \\
\endaligned
$$

\medskip

\noindent{A} consequence of the oppositeness of the two subsimplices $( \, {D} \, \cap \, \ScptA \ScptS ({\mathcal X}) \, )$ and  $( \, D \, \cap \, \ScptA \ScptS ({\mathcal X}_{D}') \, )$ of $D$ is  that $\myVect ( \, \ScptA \ScptS ({\mathcal X}) \, )$ and $\myVect ( \, \ScptA \ScptS ({\mathcal X}_{D}') \, )$ are linearly independent.   They span a subspace of dimension $(\ell - 1)$.

\medskip
  
We note if ${\mathcal X}$ is permissible and $\dim (\ScptA \ScptS ({\mathcal X}) ) = (\ell - k) > 0$, then there are infinitely many distinct chambers $D$ incident with ${\mathcal X}$, and $\ScptA \ScptS ({\mathcal X})$ is tiled by the intersections ${D} \, \cap \, \ScptA \ScptS ({\mathcal X})$ as $D$ runs over all chambers incident with ${\mathcal X}$.   Indeed, let $\Sk_{c}$ be the maximal bounded (hence compact) subgroup of $\Sk$.   We recall the discrete (free commutative of rank $\ell$) group  $\Gamma ( \ScptA ) := \Sk / \Sk_{c}$ acts as translations on $\ScptA$.   Set
\smallskip
$$
\aligned
\Gamma (\ScptA \ScptS ({\mathcal X}) ) \ :&= \ {\text{\rm{subgroup of $\Gamma ( \ScptA ) $ preserving $\ScptA \ScptS ({\mathcal X})$}}} \ , \\
\Gamma (\ScptA \ScptS ({\mathcal X}'_{D}) ) \ :&= \ {\text{\rm{subgroup of $\Gamma ( \ScptA ) $ preserving $\ScptA \ScptS ({\mathcal X}'_{D})$}}} \ .
\endaligned
$$

\noindent{These} subgroups are of ranks $(\ell -k)$ and $(k-1)$ respectively, and 
\smallskip
$$
\myVect(\ScptA \ScptS ({\mathcal X})) \ = \ \Gamma (\ScptA \ScptS ({\mathcal X}) ) \otimes_{\bZ} \bR \qquad , \qquad \myVect(\ScptA \ScptS ({\mathcal X}'_{D})) \ = \ \Gamma (\ScptA \ScptS ({\mathcal X}'_{D}) ) \otimes_{\bZ} \bR \ .
$$

\smallskip

\noindent{If} $D$ is an incident chamber to ${\mathcal X}$, and $x \in \Gamma (\ScptA \ScptS ({\mathcal X}) )$, then the translated chamber $D + x$ is also incident.  We deduce there are finitely many distinct ${\mathcal X}$ incident chambers $D_1 \, , \, \dots \, , \, D_{M}$  so that the union of facets 
\smallskip
\begin{equation}\label{tilingone}
{\text{\rm{UF}}} \ = \ ( \, D_1 \, \cap \, \ScptA \ScptS ({\mathcal X}) \, ) \ \cup \ \dots \ \cup \ ( \, D_{M} \, \cap \, \ScptA \ScptS ({\mathcal X}) \, )
\end{equation} 

\medskip

\noindent{is} a fundamental domain for the translation action of $\Gamma (\ScptA \ScptS ({\mathcal X}) )$ on $\ScptA \ScptS ({\mathcal X})$.

\medskip

\begin{prop}\label{keypermissible} Fix a chamber $C_{0}$, and suppose ${\mathcal X} = \{ \, \pm \psi_{1}, \, \dots \, , \pm \psi_{k} \, \}$ is a permissible set of affine roots.  Then, the number of chambers $D$ which are incident with ${\mathcal X}$, and satisfy

$$
c(D) \ = \ \{ \ H_{\pm \psi_{i}} \ \cap \ D \ | \ 1 \le i \le k \ \} 
$$

\noindent{(notation \eqref{chamberinout})}  is finite.
\end{prop}

\medskip

We remark about the extreme cases:
\medskip
\begin{itemize}
\item[$\bullet$]  If $k = \ell$, and ${\mathcal X}$ is permissible, then there is a unique chamber incident with ${\mathcal X}$, so the assertion is obvious.
\smallskip
\item[$\bullet$] If $k=1$, the singleton pair set ${\mathcal X} = \{ \pm \psi \}$ is automatically permissible,  $\ScptA \ScptS ({\mathcal X})$ is $H_{\pm \psi }$, and for any ${\mathcal X}$-incident chamber $D$, the intersection $E = (D \cap \ScptA \ScptS ({\mathcal X}))$ is a face of $D$, and $(D \cap \ScptA \ScptS ({\mathcal X}_{D}'))$ is the vertex opposite to $E$.
\end{itemize}

\bigskip

\begin{proof} \quad  Suppose  $D$ is a chamber incident with ${\mathcal X}$, and ${\mathcal X}_D' =   \Xi (D) \backslash {\mathcal X}$.    For each affine root pair $\pm \psi_{i}$ in ${\mathcal X}$, or $\pm \phi_{j}$ in ${\mathcal X}'_{D}$ it is convenient for us to select a preferred affine root.  We do this by designating $\psi_{i}$ to be the affine root so that the chamber $D$ lies in the half-space $H_{\psi_{i} \le 0}$ for $1 \le i \le k$.  Similarly, for a pair $\pm \phi_{j}$ in ${\mathcal X}'_{D}$, we designate $\phi_{j}$ to be the affine root so that $D$ is in the half-space $H_{\phi_{i} \ge 0}$.  

\medskip

The Proposition's hypothesis that $c(D)$ is the set $\{ \ H_{\psi_{i}} \ \cap \ D \ | \ 1 \le i \le k \ \}$ means

\begin{equation}\label{cone-one}
    C_{0} \quad \subset \quad {\text{\rm{Cone}}}({\mathcal X},D) \ := \  \big( \ {\underset {1 \le i \le k} \bigcap } H_{\psi_{i} \le 0 } \big)  \quad \bigcap \quad \big( {\underset {(k + 1) \le i \le (\ell + 1)} \bigcap } H_{\phi_{i} \le 0 } \ \big) \ .
\end{equation}

\noindent{We} define and observe: 
$$
\aligned
{\text{\rm{Cone}}}_{\mathcal X} \ :&= \ \big( \ {\underset {1 \le i \le k} \bigcap } H_{\psi_{i} \le 0 } \big) \quad \begin{tabular}{p{2.7in}}
depends only on the permissible set ${\mathcal X}$, and not on the chamber $D$, \\
\end{tabular} \\
{\text{\rm{Cone}}}_{D} \ :&= \ \big( {\underset {(k + 1) \le i \le (\ell + 1)} \bigcap } H_{\phi_{i} \le 0 } \ \big) \quad {\text{\rm{is dependent on the incident chamber $D$.}}} \\
\endaligned
$$  

\noindent{Obviously}, the cones ${\text{\rm{Cone}}}_{\mathcal X}$ and ${\text{\rm{Cone}}}_{D}$ are $\myVect (\ScptA \ScptS ({\mathcal X}))$--invariant and  $\myVect (\ScptA \ScptS ({\mathcal X}'_{D}))$--invariant respectively.  If we translate the chamber $D$ by $y \in \Gamma  ( \, \ScptA \ScptS ({\mathcal X}) \, )$, to obtain another chamber $D+y$, we have{\,}:
$$
{\text{\rm{Cone}}}_{D+y} \ = \ ( \, {\text{\rm{Cone}}}_{D} \, ) + y \ \ , \ \ {\text{\rm{and}}} \quad 
{\text{\rm{Cone}}}({\mathcal X},D+y) \ = \ {\text{\rm{Cone}}}({\mathcal X},D) + y \ .
$$

\medskip

We consider the $(\myVect (\ScptA \ScptS ({\mathcal X})) \, + \, \myVect (\ScptA \ScptS ({\mathcal X}'_{D}))$-invariant band
$$
\aligned
  {\mathcal Bn}(C_{0}) \ :&= \ C_{0} \, + \, \myVect (\ScptA \ScptS ({\mathcal X})) \, + \, \myVect (\ScptA \ScptS ({\mathcal X}'_{D})) \\
&= \ \{ \, x + y + z \ | \ x \in C_{0} \ , \ y \in \myVect (\ScptA \ScptS ({\mathcal X})) \ , \ z \in \myVect (\ScptA \ScptS ({\mathcal X}'_{D})) \,  \} \, .
\endaligned
$$
\noindent{if} $C_{0}$ lies in both ${\text{\rm{Cone}}}({\mathcal X},D)$ and ${\text{\rm{Cone}}}({\mathcal X},D+y)$, we must have 
\begin{equation}\label{c0interesction}
C_{0} \ \subset \ \big( \, {\text{\rm{Cone}}}({\mathcal X},D) \cap {\mathcal Bn}(C_{0}) \, \big) \ \cap \ \big( \, \big( \, {\text{\rm{Cone}}}({\mathcal X},D) \cap {\mathcal Bn}(C_{0}) \, \big) \, + \, y \, \big) \ .
\end{equation}
Since the set $({\text{\rm{Cone}}}({\mathcal X},D) \cap {\mathcal Bn}(C_{0}))$ is compact, and $\Gamma  ( \, \ScptA \ScptS ({\mathcal X}) \, )$ are (discrete) translations, the set of $y \in  \Gamma  ( \, \ScptA \ScptS ({\mathcal X}) \, )$ satisfying \eqref {c0interesction} is finite.  

\medskip

To summarize: \ If $D$ is a chamber incident to (a permissible set) ${\mathcal X}$ and $C_{0} \subset  {\text{\rm{Cone}}}({\mathcal X},D)$, then there are only finitely many $y \in  \Gamma  ( \, \ScptA \ScptS ({\mathcal X}) \, )$ satisfying $C_{0} \subset  {\text{\rm{Cone}}}({\mathcal X},D+y)$ as well.   As mentioned above, we can find finitely many disjoint incident chambers so that the union \eqref{tilingone} of their closures is a fundamental domain for  $\Gamma  ( \, \ScptA \ScptS ({\mathcal X}) \, )$.  The statement of the Proposition follows.
\end{proof}

\bigskip

\subsection{A criterion for a distribution to be essentially compact} \

\medskip

A much exploited criterion (see [{\reBD}:{\S}1.7]) for a distribution $D$ to be essentially compact is the following.

\smallskip

\begin{lemma} \  Suppose $D$ is a $\Gk$-invariant distribution on $\Gk$.  Then, $D$ is essentially compact, and therefore in the Bernstein center ${\mathcal Z}(\Gk )$,  if and only if for any open compact subgroup $J$,  the function 
$D \star 1_{J}$  is compactly supported. 
\end{lemma}

%%%%%%%%%%%%%%%%%%%%%%%%%%%%%%%%%%%%%%%%%%%%%%%%%%%%%%%%%%%%%%%%%%%%%%%%%%%
%%%%%%%%%%%%%%%%%%%%%%%%%%%%%%%%%%%%%%%%%%%%%%%%%%%%%%%%%%%%%%%%%%%%%%%%%%%
%%%%%%%%%%%%%%%%%%%%%%%%%%%%%%%%%%%%%%%%%%%%%%%%%%%%%%%%%%%%%%%%%%%%%%%%%%%
%%%%%%%%%%%%%%%%%%%%%%%%%%%%%%%%%%%%%%%%%%%%%%%%%%%%%%%%%%%%%%%%%%%%%%%%%%%
%%%%%%%%%%%%%%%%%%%%%%%%%%%%%%%%%%%%%%%%%%%%%%%%%%%%%%%%%%%%%%%%%%%%%%%%%%%

\vskip 0.70in 
 
%%%%%%%%%%%%%%%%%%%%%%%%%%%%%%%%%%%%%%%%%%%%%%%%%%%%%%%%%%%%%%%%%%%%%%%%%
%%%%%%%%%%%%%%%%%%%%%%%%%%%%%%%%%%%%%%%%%%%%%%%%%%%%%%%%%%%%%%%%%%%%%%%%%
%%%%%%%%%%%%%%%%%%%%%%%%%%%%%%%%%%%%%%%%%%%%%%%%%%%%%%%%%%%%%%%%%%%%%%%%%

%%%%%%%%%%%%%%%%%%%%%%%%%%%%%%%%%%%%%%%%%%%%%%%%%%%%%%%%%%%%%%%%%%%%%%%%%%%%%%%
%%%%%  Local property / finite field group property
%%%%%%%%%%%%%%%%%%%%%%%%%%%%%%%%%%%%%%%%%%%%%%%%%%%%%%%%%%%%%%%%%%%%%%%%%%%%%%%
%%%%%%%%%%%%%%%%%%%%%%%%%%%%%%%%%%%%%%%%%%%%%%%%%%%%%%%%%%%%%%%%%%%%%%%%%%%%%%%

\section{Harish-Chandra cuspidal classes and idempotents}\label{harishchandra}

\medskip

Throughout this section, $\SamG$ is a connected reductive linear algebraic group defined over a finite field $\bF_q$.   When $\SamP$ is a parabolic subgroup of $\SamG$, we denote its unipotent radical by  ${\text{\rm{rad}}}(\SamP)$.  Given a fixed Borel subgroup $\SamB \subset \SamG$, a $\SamB$-standard parabolic subgroup, is one which contains $\SamB$.  Let  $\GFq=\SamG (\bF_q)$, $\BFq = {\SamB} (\bF_q)$, $\PFq = {\SamP}(\bF_q)$, etc, denote the groups of $\bF_q$-rational points.  We take the Haar measure on $\GFq$ and its subgroups to be the point mass. 

\medskip

\subsection{A convolution property of idempotents of ${\text{\rm{rad}}}(\PFq)$} \quad  For $\HFq$ any subgroup of $\GFq$, and a representation $\kappa$  of $\HFq$, let $\Theta_{\kappa}$ signify its character.  The function 

\begin{equation}
e_{\tau} \ := \ {\frac{1}{\# (\HFq)}} \ \ {\text{\rm{deg}}} (\kappa ) \ \Theta_{\kappa} \ = {\frac{1}{\# (\HFq)}} \ \ {\text{\rm{deg}}} (\kappa ) \ {\underset {g \in \HFq} \sum} \ \ \Theta_{\kappa} (g) \, \, \delta_{g} \ 
\end{equation}

\noindent{is} a convolution idempotent (in $C(\GFq )$).

\medskip

Take $\TFq$ to be a maximal torus in $\BFq$, let $\Phi^{+} = \Phi^{+}(\TFq)$ be the positive roots determined by $\BFq$, and let $\Delta$ be the simple roots of $\Phi^{+}$.  \ The $\BFq$-standard parabolic subgroups are indexed by subsets of $\Delta$.  A subset $\frcA \subset \Delta$ corresponds to the parabolic subgroup $\PFq_{ \frcAs }$ generated by the Borel subgroup $\BFq$ and the root groups $\UFq_{-\alpha}$, $\alpha \in \frcA$.  In particular, $\PFq_{\emptyset} = \BFq$, and $\PFq_{\Delta} = \GFq$.

\medskip

\begin{lemma}  Suppose $\RFq  \supsetneq \QFq$ are two $\BFq$-standard parabolic subgroups and $\WFq = {\text{\rm{rad}}}( \RFq )$, $\VFq = {\text{\rm{rad}}}(\QFq)$ are their unipotent radicals.  Then, 
$$
\Big( \ 
{\underset 
{\RFq \supset \PFq \supset \BFq}
\sum } 
(-1)^{{\text{\rm{rank}}}(\PFq)} \ e_{{\text{\rm{rad}}}(\PFq)} \ \Big) \ \star \ e_{\VFq} \ = \ {\text{\rm{zero function}}} \, .
$$
\end{lemma}

\begin{proof}  \quad   As mentioned above, we index a $\BFq$-standard parabolic subgroup by subset $\frcA \subset \Delta$.  For convenience, we use the notation $e_{\frcAs}$ to denote the idempotent $e_{ {\text{\rm{rad}}} (P_{\frcAs}) }$. Then, \ $e_{\frcAs} \star e_{\frcBs} \, = \, e_{\frcAs \cap \frcBs}$.  Denote by $\frcQ$ and $\frcR$ the subsets of $\Delta$ satisfying $\QFq = \PFq_{\frcQs}$, and $\RFq = \PFq_{\frcRs}$. 

\smallskip

The alternating sum of the idempotents is a sum over $2^{\#( \frcRs )}$ terms.  Since $\QFq \subsetneq \RFq$, we have $\frcQ \subsetneq  \frcR$, i.e., the complement $\frcQ' := \frcR \backslash \frcQ$ is nonempty.  For each subset $q \subset \frcQ$, we consider the subsets of $\frcR$ obtained from $q$ by adding a subset $q'$ of $\frcQ'$.  The convolution of $e_{\frcQs}$ with the sum over these $2^{\#( \frcQs')}$ subsets is, up to $\pm 1$:
\smallskip
$$
\aligned
\Big( \ 
{\underset 
{q' \subset \frcQs'}
\sum } 
(-1)^{\#(q \cup q')} \ e_{q \cup q'} \ \Big) \ \star \ e_{\frcQs} \ &= \ (-1)^{\#(q)} \ 
{\underset 
{q' \subset \frcQs'}
\sum } 
(-1)^{\#(q')} \ e_{q \cup q'} \ \star \ e_{\frcQs} \\
&= \ (-1)^{\#(q)} \ 
{\underset 
{q' \subset \frcQs'}
\sum } 
(-1)^{\#(q')} \ e_{q } \\
\ &= \ {\text{\rm{zero function}}} \, .
\endaligned
$$

\noindent{The} statement of the Lemma follows.
  
\end{proof}

\medskip

\subsection{Harish--Chandra cuspidal classes} \quad Define a {\it Harish-Chandra cuspidal class} to be an equivalence class ${\frcL}$ of pairs $(\LFq, \sigma )$, modulo conjugation by $\GFq$, consisting of a Levi subgroup $\LFq \subset \GFq$ and an irreducible cuspidal representation $\sigma$ of $\LFq$.  So, $(\LFq, \sigma ) \sim (\LFq ', \sigma' )$, if there exists $g \in G$ satisfying $\LFq \ = \ g^{-1} \LFq' g$, and $\sigma (x) = \sigma' (gxg^{-1})$ for all $ x \in  \LFq$.  We recall two parabolic subgroups $\PFq$ and $\PFq'$ are associative if they have Levi factors $\LFq$ and $\LFq'$ which are $\GFq$ conjugate. 

\medskip

A Harish-Chandra cuspidal class ${\frcL}$ defines up to conjugation by $\GFq$ a Levi subgroup $\LFq$, and thus an associativity class of parabolic subgroups, and possibly several irreducible cuspidal representations of $\LFq$.   A Harish-Chandra cuspidal class ${\frcL}$  determines a subcategory ${\mathcal R}_{\frcLs}$ of the category of representations ${\mathcal R}$ of $\GFq$, namely 

\begin{equation}
{\mathcal R}_{\frcLs} \ = \ \begin{cases}
\begin{tabular}{p{4.2in}}
Representations of $\GFq$ whose irreducible subrepresentations are equivalent to representations in the induced representation ${\text{\rm{Ind}}}^{\GFq}_{\LFq \UFq} (\sigma )$, where $(\LFq, \sigma ) \in {\frcL}$, and $\LFq \UFq$ is a parabolic subgroup containing the Levi subgroup $\LFq$.    \\
\end{tabular}
\end{cases}
\end{equation}

\medskip

 If ${\frcL}$ is a Harish-Chandra cuspidal class, take $(\LFq, \sigma ) \in {\frcL}$, and set   
\smallskip
\begin{equation}
\aligned
e_{\LFq ,{\frcLs}} \ :&= \ {\frac{1}{\# (\LFq )}} \
{\underset
  {{\text{\rm{$\begin{array} {c}
    {\kappa \in \widehat{\LFq}} \\
    {(\LFq, \kappa) \in {\frcLm}}
  \end{array}$}}}}
  \sum } \deg (\kappa ) \ \Theta_{\kappa} \ , 
\endaligned
\end{equation}
\smallskip
\noindent{the} sum over those $\kappa \in \widehat{\LFq}$ satisfying $(\LFq, \kappa) \in {\frcLm}$. For a parabolic subgroup $\PFq = \LFq \UFq$, set
\smallskip
\begin{equation}
\aligned
e_{\PFq ,{\frcLs}} \ :&= \ {\frac{1}{\# (\UFq)}} \ {\text{\rm{inflation to $\PFq$ of $e_{\LFq,{\frcLs}}$}}} \ ,
\endaligned
\end{equation}
\smallskip
\noindent{We} extend $e_{\LFq,{\frcLs}}$, $e_{\PFq,{\frcLs}}$, and $e_{\UFq}$  to functions on $\GFq$ by setting their values outside $\LFq$, $\PFq$, and $\UFq$ respectively to be zero.  In $\GFq$, it is obvious 
\smallskip
$$
e_{P,{\frcLs}} \ = \ e_{L,{\frcLs}} \star e_{\UFq} \ .
$$ 

\smallskip

To any parabolic subgroup $\QFq$, we wish to attach  an idempotent $e_{Q,{\frcLs}}$.  We do this as follows:  Fix ${(\LFq, \kappa) \in {\frcLm}}$. Let $\QFq = \MFq \VFq$ be the Levi decomposition of $\QFq$.    We consider whether or not there is a conjugate of $\LFq$ contained in $\MFq$.

\begin{itemize}
\item[$\bullet$] When a conjugate $\LFq' = g\LFq g^{-1}$ of $\LFq$ is contained in $\MFq$, we take $\PFq' = \LFq' \UFq'$ to be a parabolic subgroup (of $\GFq$) which is contained in $\QFq$.  Then $\LFq' \, ( \MFq \cap \UFq')$ is a parabolic subgroup of $\MFq$.  Set  

\begin{equation}
e_{M,{\frcLs}} \ := \ 
{\frac{1}{\# (\MFq)}} \quad 
{\underset 
{\tau \in \widehat{\text{\rm{$\MFq$}}} } 
\sum'}
\deg (\tau ) \ \Theta_{\tau} \ ,  
\end{equation}

\noindent{where} the sum \ $\sum'$ \ is over the irreducible representations $\tau$ of $\MFq$ for which the invariants $\tau^{\MFq \cap \UFq'}$ contains an irreducible representation $\sigma'$ of $\LFq'$ with $(\LFq',{\sigma}') \in {\frcL}$. \ We also set  
\begin{equation}
\aligned
e_{\QFq,{\frcLs}} \ :&= \ \ {\frac{1}{\# (\VFq)}} \ \ {\text{\rm{inflation of $e_{\MFq,{\frcLs}}$ to $\QFq$ }}} \ = \ e_{\MFq,{\frcLs}} \star e_{\VFq} \ .
\endaligned
\end{equation}

\bigskip

Let

\begin{equation}
{\mathcal R}_{\MFq , \frcLs} \ = \ \begin{cases}
\begin{tabular}{p{4.3in}}
 the subcategory of representations $\pi$ of $\MFq$ so that if $\tau$ is a irreducible subrepresentation of $\pi$, and  $\sigma'$ is an irreducible $\LFq'$-subrepresentation of the invariants $\tau^{\MFq \cap \UFq'}$, then $(\LFq',{\sigma}') \in {\frcL}$. \\
\end{tabular}
\end{cases}
\end{equation}

\noindent{In particular}, ${\mathcal R}_{\GFq, \frcLs} = {\mathcal R}_{\frcLs}$.  The idempotent $e_{M, {\frcLs}}$ is characterized as

\smallskip

\begin{equation}
  \pi ( e_{\MFq, {\frcLs}} ) \ = \
  \begin{cases}
    \ \, \myId_{V_{\pi}} &{\text{\rm{for any representation $(\pi , V_{\pi})$ of }}} \\
    \ &{\text{\rm{$\MFq$ in the subcategory ${\mathcal R}_{\MFq, \frcLs}$,}}} \\
\ &\ \\
\ \, 0_{V_{\pi}} &{\text{\rm{for any irreducible representation }}} \\
    \ &{\text{\rm{$(\pi , V_{\pi})$ of $\MFq$  not in ${\mathcal R}_{\MFq, \frcLs}$.}}}
  \end{cases}
\end{equation}

\bigskip

\item[$\bullet$] When the Levi $\MFq$ does not contain any $\GFq$-conjugate of $\LFq$, we set 
\begin{equation}
e_{Q,{\frcLs}} \ := \ {\text{\rm{zero function}}} \ .
\end{equation}

\end{itemize}

\medskip

\begin{prop} Suppose ${\frcL}$ is a fixed Harish-Chandra cuspidal class.  Let $\QFq$ be a parabolic subgroup, and $\QFq = \MFq \VFq$ a Levi decomposition.  Then  
\begin{equation}\label{G-equality}
e_{\GFq,{\frcLs}} \star e_{\VFq} \ = \ e_{\QFq,{\frcLs}} \ .
\end{equation}
\end{prop}

\begin{proof}  \quad Suppose $\rho$ is any irreducible representation of $\GFq$.  We claim:

\begin{equation}\label{G-operator}
\rho (e_{\GFq,{\frcLs}} \star e_{\VFq} ) \ = \ \rho (e_{\QFq,{\frcLs}} ) \ .
\end{equation}

\medskip

\noindent{Assuming} the validity of the claim, we have for all $g \in \GFq$ that  $\rho ( \delta_{g^{-1}} \star e_{\GFq,{\frcLs}} \star e_{\VFq} ) \ = \ \rho ( \delta_{g^{-1}} \star e_{\QFq,{\frcLs}} )$, so 

\begin{equation}\label{G-trace}
\trace( \rho ( \delta_{g^{-1}} \star e_{\GFq,{\frcLs}} \star e_{\VFq} ) ) \ = \ \trace( \rho ( \delta_{g^{-1}}  \star e_{\QFq,{\frcLs}} )  ) \, . 
\end{equation}

\medskip

\noindent{Since} we can recover any function $f$ on $\GFq$  as 
\begin{equation}
\aligned
f(x) \ &= \ (\delta_{x^{-1}} \star f) (1) \ = \ {\frac{1}{\# (\GFq)}} \ {\underset { \rho \in \widehat{\GFq} } \sum } \ \deg (\rho ) \, \trace ( \, \rho (\delta_{x^{-1}} \star f) \, ) \ ,
\endaligned
\end{equation}

\noindent{we} see \eqref{G-trace}, hence \eqref{G-operator} implies the conclusion \eqref{G-equality}.

\medskip

To establish the claim, we note that since $e_{\GFq,{\frcLs}}$ is central in the group algebra ${\bC}(\GFq )$, it commutes (under convolution) with $e_{\VFq }$, and so the  operators $\rho (e_{\GFq ,{\frcLs}})$ and $\rho (e_{\VFq })$ commute.  Also 
$$
\rho (e_{\GFq ,{\frcLs}} ) \ = \ 
\begin{cases}
\ \, {\text{\rm{Id}}}_{V_{\rho}} &{\text{\rm{when $\rho \in {\mathcal R}_{\frcLs}$}}} \\
\ \\
\ \, 0_{\VFq_{\rho}} &{\text{\rm{otherwise}}} \, .
\end{cases}
$$  

\smallskip

\noindent{The} operator $\rho (e_{\VFq})$ projects to the $\VFq$-invariants (which we note is a representation of $\MFq$).   Therefore $\rho ( e_{\VFq} \star e_{G,{\frcLs}} )$ is the projection to the $\VFq$-invariants when $\rho \in {\mathcal R}_{\frcLs}$, and is the zero operator otherwise.  Since $e_{\QFq ,{\frcLs}} = e_{\MFq ,{\frcLs}} \star e_{\VFq }$, we see the operator $\rho (e_{\QFq ,{\frcLs}})$ projects to ${\mathcal R}_{\MFq ,{\frcLs}}$ when $\rho \in {\mathcal R}_{\frcLs}$ and is zero otherwise.   Thus, the claim follows.

\end{proof}

\bigskip

We fix a Borel subgroup $\BFq \subset \GFq$, and consider the following alternating sum function over $\BFq$-standard parabolic subgroups{\,}:

\begin{equation}
{\underset 
{\PFq \supset \BFq}
\sum }
(-1)^{{\text{\rm{rank}}}(P)} e_{\PFq ,{\frcLs}} \, .
\end{equation}

\smallskip

\noindent{The} next Corollary is the finite field result we shall use to show the Bruhat--Tits building version of the above alternating sum belongs to the Bernstein center.

\begin{cor}\label{finitefieldcora}  Suppose $\QFq \supset \BFq$ is proper standard parabolic subgroup, i.e, $\VFq = {\text{\rm{rad}}}(\QFq ) \neq \{ 1 \}$.  Then, 
$$
\Big( \ 
{\underset 
{\PFq \supset \BFq}
\sum }
(-1)^{{\text{\rm{rank}}}(\PFq)} e_{\PFq,{\frcLs}} \ \Big) \star e_{\VFq} \ = \ {\text{\rm{zero function}}} \, .
$$
\end{cor}

\begin{proof}  For a parabolic $\PFq \supset \BFq$, let ${\text{\rm{rad}}}(\PFq )$ denote its radical.   By the Proposition, $e_{\PFq ,{\frcLs}} = e_{\GFq ,{\frcLs}} \star e_{{\text{\rm{rad}}}(\PFq )}$.  Then,
\smallskip
$$
\aligned  
    \Big( \ 
{\underset 
{\PFq \supset \BFq}
\sum }
\ (-1)^{{\text{\rm{rank}}}(\PFq)} e_{\PFq,{\frcLs}} \ \Big) &\star e_{\VFq} \ = \ 
    \Big( \ 
{\underset 
{\PFq \supset \BFq}
\sum }
\ (-1)^{{\text{\rm{rank}}}(\PFq )} e_{\GFq ,{\frcLs}} \star e_{{\text{\rm{rad}}}(\PFq)} \ \Big) \star e_{\VFq } \\
&= \ 
   e_{\GFq,{\frcLs}} \ \star \ \Big( \ 
{\underset 
{\PFq \supset \BFq}
\sum } 
(-1)^{{\text{\rm{rank}}}(\PFq )} \ e_{{\text{\rm{rad}}}(\PFq )} \star e_{\VFq} \ \Big) \\
&= \ 
   e_{\GFq,{\frcLs}} \ \star \ 0 \ = \ 0 \, .\\
\endaligned
$$

\end{proof}

\noindent{More} generally, we have

\begin{cor}\label{finitefieldcorb}  Suppose $\RFq \supsetneq \QFq$ are two $\BFq$-standard parabolic subgroups, and $\VFq = {\text{\rm{rad}}}(\QFq)$.  Then, 
$$
\Big( \ 
{\underset 
{\RFq \supset \PFq \supset \BFq}
\sum }
(-1)^{{\text{\rm{rank}}}(\PFq )} e_{\PFq,{\frcLs}} \ \Big) \star e_{\VFq} \ = \ {\text{\rm{zero function}}} \, .
$$
\end{cor}

\begin{proof}  The proof is similar to the proof of Corollary \eqref{finitefieldcora}.  We use the property 
$$
\Big( \ 
{\underset 
{\RFq \supset \PFq \supset \BFq}
\sum } 
(-1)^{{\text{\rm{rank}}}(\PFq)} \ e_{{\text{\rm{rad}}}(\PFq)} \ \Big) \ \star \ e_{\VFq} \ = \ {\text{\rm{zero function}}} \, .
$$ 
\end{proof}

%%%%%%%%%%%%%%%%%%%%%%%%%%%%%%%%%%%%%%%%%%%%%%%%%%%%%%%%%%%%%%%%%%%%%%%%%%%%%%%
%%%%%%%%%%%%%%%%%%%%%%%%%%%%%%%%%%%%%%%%%%%%%%%%%%%%%%%%%%%%%%%%%%%%%%%%%%%%%%%
%%%%%%%%%%%%%%%%%%%%%%%%%%%%%%%%%%%%%%%%%%%%%%%%%%%%%%%%%%%%%%%%%%%%%%%%%%%%%%%
%%%%%%%%%%%%%%%%%%%%%%%%%%%%%%%%%%%%%%%%%%%%%%%%%%%%%%%%%%%%%%%%%%%%%%%%%%%%%%%
%%%%%%%%%%%%%%%%%%%%%%%%%%%%%%%%%%%%%%%%%%%%%%%%%%%%%%%%%%%%%%%%%%%%%%%%%%%%%%%
 
\vskip 0.70in 
 
%%%%%%%%%%%%%%%%%%%%%%%%%%%%%%%%%%%%%%%%%%%%%%%%%%%%%%%%%%%%%%%%%%%%%%%%%
%%%%%%%%%%%%%%%%%%%%%%%%%%%%%%%%%%%%%%%%%%%%%%%%%%%%%%%%%%%%%%%%%%%%%%%%%
%%%%%%%%%%%%%%%%%%%%%%%%%%%%%%%%%%%%%%%%%%%%%%%%%%%%%%%%%%%%%%%%%%%%%%%%%

\section{The projector for the Iwahori--Bernstein component}\label{iwahoribernstein}

\medskip

We recall our notation:  $\mk$ is a non--archimedean local field, $\Gk = \mG ( \mk )$ is the group of $\mk$-rational points of a split connected quasisimple group $\mG$ defined over $\mk$.   Let $\ScptB = \ScptB (\Gk )$ denote the Bruhat--Tits building of $\Gk$.  We fix a chamber $C_{0} \subset \ScptB$, and let $\myht_{C_{0}}$ be the Bruhat height function on chambers.   For $m \in \bN$, we define
a `ball and shell of radius m' as:

\begin{equation}\label{ballshellradiusm}
\aligned
\myBall (C_{0},m) \ :&= \ \{ \ {\text{\rm{chamber $D$}}} \ | \ \myht_{C_{0}}(D) \, \le \, m \ \} \ , \ \\ 
{\text{\rm{Shell}}}(C_{0},m) \ :&= \ \big\{ {\text{\rm{chambers}}} \ D \ | \ {\text{\rm{ht}}}_{C_0}(D) \, = \, m \ \big\} \ .
\endaligned
\end{equation}

\smallskip
\noindent{Clearly}, $\myBall (C_{0},m)$ is a convex simplicial subcomplex of $\ScptB$.

\medskip

\subsection{An equivariant system of idempotents} \quad We define a system of idempotents indexed by the facets of $\ScptB = \ScptB (\Gk )$ as follows:

\smallskip

\begin{itemize}
\item[$\bullet$]  For a chamber $F \subset \ScptB$, we take $e_{F}$ to be the idempotent of the trivial representation of the Iwahori subgroup $\Gk_{F}$.  

\smallskip

\item[$\bullet$]  For a facet $E \subset \ScptB$, let $F$ be a chamber whose closure contains $E$, so $\Gk_{E} \supset \Gk_{F}$,  $\Gk_{E}^{+} \subset \Gk_{F}^{+}$, and $\BFq = \Gk_{F}/\Gk^{+}_{E}$ is canonically a Borel subgroup of $\GFq = \Gk_{E}/\Gk_{E}^{+}$.  The pair $\frcL = (\BFq , 1_{\BFq})$ consisting of the trivial representation of the Borel subgroup $\BFq$ of $\GFq$ is a Harish-Chandra cuspidal class for $\GFq$.  Let
\begin{equation}\label{iwahoriidempotent}
    {\text{\rm{ $e_{E} \ = \ $inflation of the idempotent \ $e_{\BFq,{\frcLs}}$ \ of  $\GFq$ to $\Gk_{E}$.}}}  
\end{equation}
  
\end{itemize}

\smallskip

\noindent{The} system of idempotents $e_{F}$ for the Iwahori subgroups, i.e., chambers $F$, is clearly $\Gk$-equivariant, and therefore the collection of canonically attached idempotents $e_{E}$ for arbitrary facets is also $\Gk$-equivariant.

\bigskip

%%%%\noindent{\sc{Note:}} {\sc For general depth zero we need to configure the system of local idempotents into a coherent global system of idempotents.}

%%%%\bigskip

\subsection{Euler--Poincar{\'{e}} presentation of a distribution} \quad The key result we show is{\,}:

\begin{thm}\label{mainiwahori} For a facet $F$ of $\ScptB (\Gk)$, let $e_{F}$ be the idempotent defined in \eqref{iwahoriidempotent}.  The $\Gk$-invariant distribution defined as the infinite alternating sum
  
$$
\Big( \ 
{\underset 
{E  \,  \subset  \, \, \ScptB (G)}
\sum } 
(-1)^{{\text{\rm{dim}}}(E)} \ e_{E} \ \Big) \ 
$$
\noindent{over} the facets is an essentially compact distribution on $\Gk$, i.e., in the Bernstein center.  

\end{thm}

To prove Theorem \eqref{mainiwahori}, it is enough to show for any open compact subgroup $J$, and $e_{J} = {\frac{1}{\meas (J)}} 1_{J}$, the convolution

$$
{\underset 
{E  \,  \subset  \, \, \ScptB (G)}
\sum } 
(-1)^{{\text{\rm{dim}}}(E)} \ e_{E} \ \star e_{J}
$$

\noindent{has} compact support.  Fix a chamber $C_{0} \subset \ScptB$, and let 
$\myht_{C_{0}}$ be the height function, with respect to $C_{0}$, on the chambers of $\ScptB$.  When $D$ is a chamber, we use $\myht_{C_{0}}$ to partition the faces of $D$ into the two subsets of parent $p(D)$ faces and child faces $c(D)$ as in
 \eqref{chamberinout}.  We note $c(D)$ is always non-empty, and when $D \neq C_{0}$, i.e., $\myht_{C_{0}} (D) > 0$, then $p(D)$ is non-empty too.  In this latter situation ($\myht_{C_{0}} (D) > 0$), the intersection 

\begin{equation}\label{min-facet}
D_{+} \ := \ {\underset {F \in c(D)} \bigcap } \ F \qquad 
\end{equation}

\noindent{is} a facet of $D$. \ Set
\begin{equation}\label{min-facet-set}
  {\mathcal F}_{+}(D) \ := \ {\text{\rm{the set of facets $E$ of $D$ which contain $D_{+}$}}} \, .
\end{equation}

\medskip

\begin{thm}\label{dplus}  Fix a chamber $C_{0} \subset \ScptB$.  Suppose $J$ is an open compact subgroup of $G$.  If $D$ is a chamber with $\myht_{C_{0}}(D)$ sufficiently large (depending on $J$), then the convolution
  
\begin{equation}\label{dplusconvolution}
{\underset 
{E  \,  \in  \, {\mathcal F}_{+}(D)}
\sum } 
(-1)^{{\text{\rm{dim}}}(E)} \ e_{E} \ \star e_{J} \ = \ {\text{\rm{ zero function}}} \, . 
\end{equation}

\end{thm}

\medskip

The proofs are in section {\eqref{proofsectioniwahori}{\,}}.
\bigskip

\subsection{Some preliminary results}

\

\medskip

We need several preliminary results to prove Theorem \eqref{dplus}. \ Suppose $\Sk$ is a maximal split torus, and $\ScptA = \ScptA (\Sk)$  its associated apartment.  Given a choice of positive roots $\Phi^{+} \subset \Phi = \Phi (\Sk )$, we write $\Delta$ for the corresponding set of simple roots, and we write the vectors of the dual basis of $\Delta$ as $\Delta^{\star} = \{ \, \lambda_{\alpha} \, | \, \alpha \in \Delta  \, \}$.  Fix a chamber $C_{0} \subset \ScptA$.  For any choice of a positive root system $\Phi^{+} \subset \Phi$, we defined the $C_{0}$-based sector $S(C_{0},\Phi^{+})$ in $\ScptA$ in \eqref{csector}.   The union ${\underset {\Phi^{+}} \bigcup } \ S(C_{0},\Phi^{+})$ over all positive sets of roots is the apartment $\ScptA$.

\medskip

\begin{lemma}\label{simplerootestimatea}  Suppose $\Phi^{+} \subset \Phi (\Sk)$ is a set of positive roots and  $\Delta \subset \Phi^{+}$ are the simple roots.  Fix a simple root $\alpha \in \Delta$, and suppose $\gamma \in \Phi^{+}$ satisfies $\lambda_{\alpha} (\gamma ) > 0$, i.e., the expression of $\gamma$ as a sum of simple roots has positive $\alpha$ coefficient.  If  $D \subset S(C_{0},\Phi^{+})$ is a chamber which is separated from $C_{0}$ by $L \ge 2$ distinct affine root hyperplanes $H_{\psi}$ with $\mygrad (\psi ) = \alpha$, then $D$ is separated from $C_{0}$ by at least $(L - 1)$ distinct affine root hyperplanes $H_{\psi}$ with $\mygrad (\psi ) = \gamma$.
\end{lemma}

\begin{proof} \quad Let $\gamma = {\underset {\beta \in \Delta} \sum } n_{\beta} \beta$ be the linear expansion of $\gamma$ in terms of simple roots.   The hypothesis that there are $L$ affine hyperplanes perpendicular to $\alpha$ means $\alpha (y) - \alpha (x) > (L-1)$ for any $x \in C_{0}$, and $y \in D$.  Suppose $\beta \in \Delta \, \backslash \, \{ \, \alpha \, \}$.  \ \ If \ (i) $D$ and $C_{0}$ are separated by an affine hyperplane perpendicular to $\beta$, then $\beta (y) \, - \, \beta (x) \, > \, 0$.  \ \ Else, \ (ii) $D$ and $C_{0}$ are not separated by any affine hyperplane perpendicular to $\beta$.  Let $Q$ be the set of these simple roots.  These simple roots are linearly independent (as is any nonempty subset of $\Delta$), and therefore we can replace $y$ by some $y' \in \mypfacet (D)$, so that $\beta (y') = \beta (x)$ for all $\beta \in S$.   \ Therefore, $\gamma (y') - \gamma (x) > n_{\alpha} (L-1) \ge (L-1)$.   This means $C_{0}$ and $D$ are separated by at least $(L-1)$  affine hyperplanes perpendicular to $\gamma$. 
\end{proof}

\smallskip

\begin{cor}\label{separation} \quad Suppose $D$ is a chamber in the $C_{0}$-based sector $S(C_{0},\Phi^{+})$, and for all simple roots $\alpha \in \Delta$, $D$ is separated from $C_{0}$ by $L \ge 2$ affine hyperplanes perpendicular to $\alpha$, then for any $\gamma \in \Phi^{+}$, $D$ is separated from $C_{0}$ by at least $(L-1)$  affine hyperplanes perpendicular to $\gamma$.
\end{cor} 

\begin{proof} Clear.
\end{proof}

\medskip

\bigskip

\subsection{Proof of Theorems \eqref{dplus}  and \eqref{mainiwahori}}\label{proofsectioniwahori}

\

\medskip

We prove Theorem \eqref{dplus}.

\begin{proof}  \quad We can and do replace the open compact subgroup $J$ by a group $G_{x_{0},\rho}$, where $x_{0} \in C_{0}$, and $\rho$ is a sufficiently large integer to insure $G_{x_{0},\rho} \subset J$.  The fact that $\rho$ is an integer means $G_{x,\rho} = G_{x_{0},\rho}$ for any $x \in \mypfacet{(C_{0})}$.

\medskip

Take $\Sk$ to be a maximal split torus of $\Gk$ so that the apartment $\ScptA = \ScptA (\Sk )$ contains $C_{0}$ and $D$.  We recall $\ScptA$ is the union the $C_{0}$-based sectors $S(C_{0}, \Phi^{+})$ as $\Phi^{+}$ runs over the sets of possible positive roots.

\medskip

Take $\Phi^{+} \subset \Phi = \Phi (\Sk )$ to be a choice of positive roots so that the chamber based sector $S(C_{0}, \Phi^{+})$ contains $D$.  The choice of the particular set of positive roots may not be unique.  For a choice $\Phi^{+}_{D}$, let $\mB_{D}^{+}$ and $\mB_{D}^{-}$ be the Borel subgroups of $\mG$ associated to the sets of positive roots $\Phi^{+}$ and $-\Phi^{+}$.     Let $\Delta$ denote the set of simple roots of $\Phi^{+}$.  We have a decomposition of the subgroup $J = \Gk_{x_{0},\rho}$ into subgroups:

\begin{equation}
  \Gk_{x_{0},\rho} \ = \ \Gk_{x_{0},\rho}^{-} \,\Sk_{\rho} \ \Gk_{x_{0},\rho}^{+} \ ,
\end{equation}

\noindent{where}

$$
\Gk_{x_{0},\rho}^{+}  := \Gk_{x_{0},\rho} \cap {\mB_{D}^{+}}(\mk) \quad {\text{\rm{is a product}}} {\underset {{\text{\rm{\tiny{$\begin{array} {c} \psi \ , \ \psi (x_{0}) \geq \rho \\ \mygrad ({\psi} ) \in  \Phi^{+} \end{array}$ }}}}} \prod} \ \Xk_{\psi}
$$

\noindent{of} affine roots groups $\Xk_{\psi}$, with $\psi (x_{0}) \geq \rho$, and $\mygrad ({\psi} ) \in  \Phi^{+}$.  Similarly,

$$
\Gk_{x_{0},\rho}^{-}  := \Gk_{x_{0},\rho} \cap {\mB_{D}^{-}}(\mk) \quad = {\underset {{\text{\rm{\tiny{$\begin{array} {c} \psi \ , \ \psi (x_{0}) \geq \rho \\ \mygrad ({\psi} ) \in  -\Phi^{+} \end{array}$ }}}}} \prod} \ \Xk_{\psi} \, .
$$

\noindent{Suppose} $\alpha \in \Phi^{+}$.  The intersection $\Gk_{x_{0},\rho}^{-} \cap \Uk_{-\alpha}$ is an affine root group $\Xk_{\psi}$ (necessarily $\mygrad (\psi ) = -\alpha$), and for all $x \in C_{0}$, we have \ $\rho + 1 > \psi (x) \ge \rho$.   \ If the chambers $C_{0}$ and $D$ are separated by at least $\rho+1$ affine hyperplanes perpendicular to $\gamma$, then  $\psi (y) < 0$ for all $y \in \mypfacet (D)$, and thus $\psi (y) \le 0$ for all $y \in D$.  This means the affine root group $\Xk_{\psi}$ contains $\Gk_{y,0} \cap \Uk_{-\alpha}$ for all $y \in D$.  We note that since $\psi (x) \ge \rho$ for all $x \in C_{0}$ (so for $x_0$), it is the case that $\Xk_{\psi}$ is contained in $J = \Gk_{x_{0},\rho}$.  

\medskip

To summarize:  If $\gamma \in \Phi^{+}$, and $D$ is a chamber of  $S(C_{0}, \Phi^{+})$ separated from $C_{0}$ by sufficiently many (at least  ($\rho + 1$)) affine hyperplanes perpendicular to $\gamma$, then $\Gk_{x_{0},\rho}^{-} \cap \Uk_{-\gamma}$, hence $\Gk_{x_{0},\rho}$, contains $\Gk_{y,0} \cap \Uk_{-\gamma}$ for any $y \in D$.  

\medskip

For $k \in \bN$, set

\begin{equation}\label{recessgroupone}
{\frcR}_{k} \ = \ {\text{\rm{the set }}} \Big\{ \ 
\begin{tabular}{p{3.5in}}
chambers $D$ of $S(C_{0},\Phi^{+})$ satisfying{\,}: for each $\alpha \in \Delta$, the chamber $D$ is separated from $C_{0}$ by at least $k$ affine hyperplanes $H_{\psi}$ perpendicular to $\alpha$. \\
\end{tabular}
\ \Big\}
\end{equation}

\medskip

\noindent{Suppose} $D \in {\frcR}_{(\rho+2)}$.  By Corollary \eqref{separation}, for any $\gamma \in \Phi^{+}$, the chamber $D$ is separated from $C_{0}$ by at least $(\rho + 1)$ affine hyperplanes perpendicular to $\gamma$; hence, $\Gk_{x_{0},\rho }$ contains $\Gk_{y,0} \cap \Uk_{-\gamma}$ for all $y \in D$.

\medskip

\noindent{When} $E$ is a facet in $\ScptA$, define{\,}:
\smallskip
\begin{equation}\label{facetaffineroots}
\aligned
\Psi (E) \ :&= \ {\text{\rm{set of affine roots $\psi$ which vanish on $E$}}}{\,} \, , \\
\Psi (E,\Phi^{+}) \ :&= \ \{ \, \psi \in \Psi (E) \ | \ \mygrad (\psi ) \in \Phi^{+} \, \} \, .
\endaligned
\end{equation} 
\medskip
Recall $D_{+}$ is defined to be the facet of $D$ which is the intersection of all the outward oriented faces of $D$.   We pick $y$ to be a point in $\mypfacet (D_{+})$.    The finite field group $\Gk_{y,0}/\Gk_{y,0^{+}}$ has root system
$$
\Phi (\Gk_{y,0}/\Gk_{y,0^{+}} ) \ = \ \{ \ \mygrad (\psi ) \ | \ {\text{\rm{$\psi$ affine root so that $\psi_{|_{D_{+}}} \equiv 0$}}} \ \} \ = \ \mygrad ( \Psi (E) ) \, .
$$

\medskip
\noindent{Suppose} $\psi \in \Psi (E,\Phi^{+})$.  The hypothesis $C_{0}$ and $D$ are separated by $(\rho + 1)$ affine hyperplanes perpendicular to $\gamma$ means $(\psi (y) - \psi (x_0)) \, > \, \rho$, i.e., $-\psi (x_{0}) > \rho$.  Thus, $\Gk_{x_{0},\rho}$ contains $\Xk_{-\psi}$. This latter subgroup is $\Gk_{y,0} \cap \Uk_{-\gamma}$.  It follows   

\begin{equation}\label{recessgrouptwo}
V_{y} \ :=  \ {\underset { \psi \, \in \, \Psi (E,\Phi^{+}) } \prod } \ \Xk_{-\psi}
\end{equation}

\noindent{is} contained in $\Gk_{x_{0}, \rho}$ and that $V_{y} \Gk^{+}_{D_{+}} = \Gk^{+}_{D}$ (the subgroup $V_{y}$ adds in the affine root groups $\Xk_{-\psi}$), so $V_{y} \Gk^{+}_{D_{+}}/\Gk^{+}_{D_{+}} \subset \Gk_{D_{+}}/\Gk^{+}_{D_{+}}$ is the unipotent radical of the Borel subgroup $\Gk_{D}/\Gk^{+}_{D_{+}}$. 

\medskip

This means, by Corollary \eqref{finitefieldcora},  the convolution 
\smallskip
\begin{equation}\label{recessgroupthree}
\aligned
{\underset 
{E  \,  \in  \, {\mathcal F}_{+}(D)}
\sum } 
(-1)^{{\text{\rm{dim}}}(E)} \ e_{E} &\ \star \ e_{V_{y}} \ =  {\underset 
{E  \,  \in  \, {\mathcal F}_{+}(D)}
\sum } \ (-1)^{{\text{\rm{dim}}}(E)} \ ( \ e_{E} \ \star \ e_{\Gk_{y,0^{+}}} ) \ \star \ e_{V_{y}} \\
&\ = {\underset 
{E  \,  \in  \, {\mathcal F}_{+}(D)}
\sum } \ (-1)^{{\text{\rm{dim}}}(E)} \ e_{E} \ \star \ ( \ e_{\Gk_{y,0^{+}}} \ \star \ e_{V_{y}} \ ) \\
&\ = {\underset
{E  \,  \in  \, {\mathcal F}_{+}(D)}
\sum } \ (-1)^{{\text{\rm{dim}}}(E)} \ e_{E} \ \star \ ( \, e_{(\Gk_{y,0^{+}} \, V_{y} )} \, ) \\
&\ =  \ {\text{\rm{ zero function}}} \, . \\ 
\endaligned
\end{equation}

\noindent{So}, under the assumption $D \subset \frcR_{(\rho+2)}$, we see

\begin{equation}\label{recessgroupfour}
\aligned
{\underset 
{E  \,  \in  \, {\mathcal F}_{+}(D)}
\sum } 
(-1)^{{\text{\rm{dim}}}(E)} \ e_{E} &\ \star e_{\Gk_{x_{0},\rho}} \ =  {\underset 
{E  \,  \in  \, {\mathcal F}_{+}(D)}
\sum } \ (-1)^{{\text{\rm{dim}}}(E)} \ e_{E} \ \star \ ( \ e_{V_{y}} \ \star \ e_{\Gk_{x_{0},\rho}} \ ) \\
&\ = \big( {\underset 
{E  \,  \in  \, {\mathcal F}_{+}(D)}
\sum } \ (-1)^{{\text{\rm{dim}}}(E)} \ e_{E} \ \star \ e_{V_{y}} \ \big) \ \star \ e_{\Gk_{x_{0},\rho}} \\
&\ =  \ 0 \ \star \ e_{\Gk_{x_{0},\rho}} \ = \ {\text{\rm{ zero function}}} \, . \\ 
\endaligned
\end{equation}

\bigskip

We turn to the situation when the chamber $D$ is in $S(C_{0},\Phi^{+}) \, \backslash \, \frcR_{(\rho + 2)}$.

\bigskip

For a subset $I \subset \Delta$, and an integer $k$ set 

\begin{equation}\label{recessgroupsix}
\aligned
{\frcR}_{\{ I, k \}} \ := \ \{ \, D \subset S(C_{0},\Phi^{+}) \ | \ &\myht^{\pm \alpha}_{C_{0}}(D) \, \ge  \, k \ \ \forall \ \alpha \in I \ {\text{\rm{, and }}} \\ 
&\myht^{\pm \alpha}_{C_{0}}(D) \, < \, k \ \ \forall \ \alpha \in  ( \Delta \, \backslash \, I) \, \} \, .
\endaligned
\end{equation}

\smallskip

\noindent{We note}
\begin{itemize}
\item[(i)] ${\frcR}_{\{ \Delta, k \}}$ is  the set ${\frcR}_{k}$ in \eqref{recessgroupone}{\,}.
\smallskip
\item[(ii)] The set ${\frcR}_{\{ I, k \}}$ is finite precisely when $I \, = \, \emptyset${\,}.
\smallskip
\item[(iii)] For a fixed $k$, the sets ${\frcR}_{\{ I, k \}}$ partition $S(C_{0},\Phi^{+})${\,}.
\end{itemize}

\medskip

To complete the proof of Theorem \eqref{dplus}, we need to show, when $I$ is nonempty, the convolution \eqref{dplusconvolution} vanishes for all $D \in {\frcR}_{\{ I, (\rho + 2) \}}$ provided $\myht_{C_{0}}(D)$ is sufficiently large.   The case $I \, = \, \Delta$ has already been treated above.  Set 
\smallskip
$$
I^{+} \ := \ \{ \ \gamma \in \Phi^{+} \ | \ {\text{\rm{there exists $\alpha \in I$, such that $\lambda_{\alpha}(\gamma) > 0$}}} \} \, .
$$
\bigskip
We reuse the initial argument when $I \, = \, \Delta$ to see
$$
\forall \ D \, \in \, {\frcR}_{\{ I, (\rho + 2) \}} \ , \ {\text{\rm{and}}} \ y \in D \ \ : \quad \Gk_{x_{0}, \rho} \, \supset \, (\Gk_{y,0} \cap \Uk_{-\gamma}) \quad \forall \ \gamma \, \in \, I^{+} \, . 
$$

\noindent{Thus,}

\begin{equation}\label{recessgroupeight}
\Gk_{x_{0},\rho} \, \supset \, V_{y} \ := \ {\underset {\gamma \in I^{+}} \prod } (\Gk_{y,0} \cap \Uk_{-\gamma}) \ .
\end{equation}

\noindent{Take} $y \in D_{+}$.  We recall the $\Sk_{c}$-roots $\Phi (\Gk_{D_{+}}/\Gk^{+}_{D_{+}},\Sk_{c})$  of the finite field group $\Gk_{D_{+}}/\Gk^{+}_{D_{+}}$ are the gradients of the affine roots in the set $\Psi (D_{+})$ (see \eqref{facetaffineroots}).  Suppose $\psi$ is such an affine root, and $\gamma := \mygrad (\psi) \in \Phi^{+}$.  If  there is an $\alpha \in I$ such that $\lambda_{\alpha} (\gamma ) > 0$, then necessarily $C_{0}$ and $D$ are separated by at least $\rho + 1$ affine hyperplanes perpendicular to $\gamma$.  Since $\psi (y) = 0$, we get $-\psi (x_{0}) \, = \, (\psi (y) - \psi (x_{0}) ) \, >  \, \rho$.  This means $\Xk_{-\psi} \, \subset \, \Gk_{x_{0}, \rho}$.   The image of $\Xk_{-\psi}$ in the finite field group $\Gk_{y,0}/\Gk^{+}_{y,0}$ (equal to $\Gk_{D_{+}}/\Gk^{+}_{D_{+}}$) is the (nontrivial) root group attached to $-\gamma$. 

\smallskip

To continue the proof, we observe the group $V_{y} \Gk^{+}_{D_{+}}/\Gk^{+}_{D_{+}}$ is the unipotent radical of the $\Gk_{D}/\Gk^{+}_{D_{+}}$-standard parabolic subgroup $\Gk_{F}/\Gk^{+}_{D_{+}}$  of $\Gk_{D_{+}}/\Gk^{+}_{D_{+}}$ generated by the affine root groups $\Xk_{-\psi}$ with $\psi \in \Psi (D_{+})$ satisfying $\mygrad (\psi ) \in I^{+}$.   If we define the coroot
$$
\lambda_{I} \ := \ {\underset {\alpha \in I} \sum } \ \lambda_{\alpha} \ ,
$$
then: {\ }(i){\,} the Levi subgroup of $\Gk_{F}/\Gk^{+}_{D_{+}}$ (containing $\Sk_{c}\Gk^{+}_{D_{+}}$ is generated by the root groups $\Xk_{\psi}$ ($\psi \in \Psi (D_{+})$) satisfying $\lambda_{I} (\mygrad (\psi )) = 0$, and {\,}(ii){\,} the unipotent radical of $\Gk_{F}/\Gk^{+}_{D_{+}}$ is the product of the root groups $\Xk_{-\psi}$ with $\lambda_{I} (\mygrad (\psi )) > 0$.  Provided $V_{y} \Gk^{+}_{D_{+}}/\Gk^{+}_{D_{+}}$ is not the trivial unipotent subgroup, the ending argument for the case $\frcR_{(\rho + 2)}$ can be applied to deduce the convolution  \eqref{dplusconvolution} vanishes.  By what we have argued above, this happens if there exists $\psi \in \Psi (D_{+} , \Phi^{+})$ and an $\alpha \in I$ with $\lambda_{\alpha}( \mygrad (\psi) ) \ne 0$.  For these chambers the convolution \eqref{dplusconvolution} vanishes.  

\smallskip

Given a subset $K \subset \Delta$, set
\begin{equation}\label{subsetofdelta}
\Phi (K) \ := \ \{ \, \alpha \in \Phi \ | \ \alpha {\text{\rm{ is a linear combination of (simple) roots in $K$}}} \, \} \ .
\end{equation}

We are reduced to investigating $D \in \frcR_{I,(\rho + 2)}$ so that every $\psi \in \Psi ( D_{+} , \Phi^{+} )$ satisfies $\mygrad (\psi ) \in \Phi (\Delta \backslash I)$.  Denote this set by $\frcR^{\text{\rm{last}}}_{I,(\rho + 2)}$.

$$
\frcR^{\text{\rm{last}}}_{I,(\rho + 2)} \ := \ \{ \ D \in \frcR_{I,(\rho + 2)} \ | \ {\text{\rm{$\psi \in \Psi ( D_{+} , \Phi^{+})$ satisfies $\mygrad (\psi ) \in \Phi (\Delta \backslash I)$}}}  \ \} \, .
$$

The set of outward oriented faces of a chamber $D \in  \frcR^{\text{\rm{last}}}_{I,(\rho + 2)}$ must have gradients in $\Phi ( \Delta \backslash I)$, but in principle, it could be a proper subset.  We partition $\frcR^{\text{\rm{last}}}_{I,(\rho + 2)}$ as follows: \ \ To a (nonempty) subset $K$ of $\Delta \backslash I$, we set

$$
\aligned
\frcR^{\text{\rm{K}}}_{I,(\rho + 2)} \ := \ \{ \ D \in \frcR_{I,(\rho + 2)} \ | \ &{\text{\rm{$\psi \in \Psi ( D_{+} , \Phi^{+} )$ satisfies}}} \\
&{\text{\rm{(i) $\mygrad (\psi ) \in \Phi (\Delta \backslash I)$}}}   \\ 
&{\text{\rm{(ii) each simple root $\beta \in K$ occurs as a $\mygrad (\psi)$}}}  \ \} \ . 
\endaligned
$$ 

\noindent{The} sets $\frcR^{\text{\rm{K}}}_{I,(\rho + 2)}$ are a partitioning of $\frcR^{\text{\rm{last}}}_{I,(\rho + 2)}$ into $2^{\# (\Delta \backslash I)}-1$ subsets.  Furthermore, a chamber in $\frcR^{\text{\rm{K}}}_{I,(\rho + 2)}$ is incident with the permissible set $\tilde{K} \,:= \, \{ \, \psi \in \Psi (D_{+},\Phi^{+}) \ | \ \mygrad (\psi ) \in K \, \}$ (see \eqref{permissibleone}).  We apply Proposition \eqref{keypermissible} to say the number of chambers in $\frcR^{\text{\rm{K}}}_{I,(\rho + 2)}$ is finite.   So, $\frcR^{\text{\rm{last}}}_{I,(\rho + 2)}$ is finite, and we deduce the Theorem when $D$ is a chamber in $\ScptA$.  

\medskip

The (compact) Iwahori subgroup $\Gk_{C_{0}}$ acts transitively on the set $\ScptA (C_{0})$ of apartments containing $C_{0}$.  Fix an apartment $\ScptA'$ containing $C_{0}$.  Consider an apartment $g \ScptA'$ ($g \in \Gk_{C_{0}}$).  The above argument applied to the apartment $g \ScptA'$ shows there is an compact open subgroup $K_{g}$ and an integer $M_{g} > 0$ so that the convolution \eqref{dplusconvolution} vanishes for all $D \in hg\ScptA'$ ($h \in K_{g}$) provided $\myht_{C_{0}}(D) \ge M_{g}$.  The collection of sets $\{ \, K_{g}g \ | \ g \in \Gk_{C_{0}} \, \}$ is an open cover of $\Gk_{C_{0}}$, and so has a finite subcover $\{ \, g_{i}K_{g_{i}} \ | \ i = 1,\dots , n \, \}$.  Take $M = {\text{\rm{max}}} (M_{g_{1}}, \dots , M_{g_{n}})$.   The convolution \eqref{dplusconvolution} vanishes for any chamber $D$ of $\ScptB$ satisfying $\myht_{C_{0}}(D) \ge M$.

\end{proof}

\vskip 0.50in

We turn to the proof of Theorem \eqref{mainiwahori}.

\begin{proof} \quad We recall that we can replace the open compact subgroup by a subgroup $\Gk_{x_{0},\rho } \subset J$ with $x_{0} \in C_{0}$ and $\rho$ integral.  We fix a chamber $C_{0}$ and for a positive integer $m$, consider the ball $\myBall (C_{0},m)$ of \eqref{ballshellradiusm}.    We consider the convolutions

\begin{equation}\label{ballofradiusm}
{\underset 
{E  \,  \subset \, {\text{\rm{Ball}}}(C_{0},m) }
\sum } 
(-1)^{{\text{\rm{dim}}}(E)} \ e_{E} \ \star e_{\Gk_{x_{0},\rho }} \ .
\end{equation}

\smallskip

\noindent{The} sum is over the facets in ${\text{\rm{Ball}}}(C_{0},m)$.  It is clear the convolution over ${\text{\rm{Ball}}}(C_{0},(m+1))$ is obtained from the convolution over ${\text{\rm{Ball}}}(C_{0},m)$ by adding convolution terms of the form

$$
{\underset 
{E  \,  \in  \, {\mathcal F}_{+}(D)} 
\sum }
(-1)^{{\text{\rm{dim}}}(E)} \ e_{E} \ \star e_{\Gk_{x_{0},\rho}} \, ,
$$

\noindent{where} $D$ runs over the chambers satisfying $\myht_{C_{0}}(D) = (m+1)$, i.e., in Shell$(C_{0},(m+1))$.  \ By Theorem \eqref{dplus}, these convolution terms vanish provided $m$ is sufficiently large.  Therefore, the convolution over ${\text{\rm{Ball}}}(C_{0},m)$ and ${\text{\rm{Ball}}}(C_{0},(m+1))$ are the same when $m$ is sufficiently large.  \ This establishes Theorem \eqref{mainiwahori}.

\end{proof}

\bigskip

\subsection{The Iwahori--Bernstein component} \quad The next Proposition and Corollary show the essentially compact distribution of Theorem \eqref{mainiwahori} is the projector to the Bernstein component of representations with a nonzero Iwahori fixed vector.

\medskip

\begin{prop} For any facet $E \subset \ScptB = \ScptB (G)$, define the idempotent $e_{E}$ as in \eqref{iwahoriidempotent}. \ Fix a chamber $C_{0}$ in $\ScptB = \ScptB (G)$. \ Then:
\begin{itemize}
\item[(i)] 
$$
{\underset 
{E  \,  \subset  \, C_{0}}
\sum } 
(-1)^{{\text{\rm{dim}}}(E)} \ e_{E} \ \star \ e_{C_{0}} \ = \ e_{C_{0}}
$$
\item[(ii)]  For any chamber $D \ \neq \ C_{0}$:
$$
{\underset 
{E  \,  \in  \, {\mathcal F}_{+}(D)}
\sum } 
(-1)^{{\text{\rm{dim}}}(E)} \ e_{E} \ \star \ e_{C_{0}} \ = \ {\text{\rm{zero function}}} \, .
$$
\item[(iii)] 
$$
\big( {\underset 
{E  \,  \subset \, \ScptB (G)}
\sum } 
(-1)^{{\text{\rm{dim}}}(E)} \ e_{E} \big) \ \star \ e_{C_{0}} \ = \ e_{C_{0}}
$$
\end{itemize}
\end{prop}

\begin{proof} \quad Statement (i) follows from the fact that $e_{E} \star e_{C_{0}} \, = \, e_{C_{0}}$ for any facet $E \subset C_{0}$.

\medskip

Statement (ii) is seen by modifying the proof of Theorem \eqref{dplus}. \ Let $\Phi^{+} \subset \Phi$ be a positive root system so that $S(C_{0},\Phi^{+})$ contains $D$.  The difference between the two Iwahori subgroups $\Gk_{D}$ and $\Gk_{C_{0}}$ is the following: \ for $\alpha \in \Phi^{+}$ it is the case $(\Gk_{C_{0}} \cap \Uk_{\alpha} ) \subset ( \Gk_{D} \cap \Uk_{\alpha})$, while $( \Gk_{C_{0}} \cap \Uk_{-\alpha} )  \supset ( \Gk_{D} \cap \Uk_{-\alpha} )$.  Thus, $(\Gk_{C_{0},0^{+}} \cap \Gk_{D_{+},0})\Gk_{D_{+},0^{+}} = \Gk_{D}$.  So

$$
\aligned
{\underset 
{E  \,  \in  \, {\mathcal F}_{+}(D)}
\sum } 
(-1)^{{\text{\rm{dim}}}(E)} \ &e_{E} \ \star \ e_{C_{0}} \ = \ \big( \, {\underset 
{E  \,  \in  \, {\mathcal F}_{+}(D)}
\sum } 
(-1)^{{\text{\rm{dim}}}(E)} \ e_{E} \ \star \ e_{\Gk_{D_{+},0^{+}}} \, \big) \star ( \, e_{(\Gk_{C_{0},0^{+}} \cap \Gk_{D_{+},0})}  \star e_{C_{0}} ) \\
&= \ \big( \, {\underset 
{E  \,  \in  \, {\mathcal F}_{+}(D)}
\sum } 
(-1)^{{\text{\rm{dim}}}(E)} \ e_{E} \ \star \ ( \, e_{\Gk_{D_{+},0^{+}}}  \, \star  \, e_{(\Gk_{C_{0},0^{+}} \cap \Gk_{D_{+},0})} \, ) \  \big) \star e_{C_{0}} ) \\
&= \ \big( \, {\underset 
{E  \,  \in  \, {\mathcal F}_{+}(D)}
\sum } 
(-1)^{{\text{\rm{dim}}}(E)} \ e_{E} \ \star \ e_{\Gk_{D}} \, \big) \star  e_{C_{0}}  \\
&= \ {\text{\rm{zero function}}} \, .
\endaligned
$$

\medskip

Statement (iii) is an obvious consequence of statements (i) and (ii).

\end{proof}

\medskip

\begin{cor}\label{iwahoriprojectorcor}  The distribution $P  := \big( {\underset {E  \,  \subset \, \ScptB (G)} \sum } (-1)^{{\text{\rm{dim}}}(E)} \ e_{E} \big)$  is the projector to the Bernstein component with nonzero Iwahori fixed vectors.
\end{cor}

\medskip

\begin{proof}  \quad Suppose $(\pi , V_{\pi})$ is an irreducible smooth representation of $G$.  The operator $\pi (P)$ is a scalar operator.

\medskip

For any facet $E \subset \ScptB$, the operator $\pi (e_{E})$ projects to the subspace $V_{\pi}^{\Gk_{E,0^{+}}}$.  Furthermore, from the definition of $e_{E}$, any nonzero irreducible representation of $\Gk_{E,0}/\Gk_{E,0^{+}}$ in $V_{\pi}^{\Gk_{E,0^{+}}}$ must have a nonzero Iwahori fixed vector for any Iwahori subgroup $\Gk_{D}$ contained in $\Gk_{E}$.  It follows that if $V_{\pi}^{\Gk_{D}} = \{ 0 \} $ for any chamber $D$, then the scalar $\pi (P)$ is zero, i.e., a necessary condition for $\pi (P)$ to be nonzero is that  $\pi$ has an nonzero Iwahori fixed vector. 

\medskip

On the other hand, if there is a chamber $D$, so that $V_{\pi}^{\Gk_{D}} \neq \{ 0 \} $, then  $V_{\pi}^{\Gk_{C_{0}}} \neq \{ 0 \} $ too.  Since $P \star e_{C_{0}} = e_{C_{0}}$, we conclude $\pi (P) = \myId_{V_{\pi}}$, and thus $P$ is the Bernstein projector for the component with nonzero Iwahori fixed vectors.

\end{proof}

%%%%%%%%%%%%%%%%%%%%%%%%%%%%%%%%%%%%%%%%%%%%%%%%%%%%%%%%%%%%
%%%%%%%%%%%%%%%%%%%%%%%%%%%%%%%%%%%%%%%%%%%%%%%%%%%%%%%%%%%%
%%%%%%%%%%%%%%%%%%%%%%%%%%%%%%%%%%%%%%%%%%%%%%%%%%%%%%%%%%%%
%%%%%%%%%%%%%%%%%%%%%%%%%%%%%%%%%%%%%%%%%%%%%%%%%%%%%%%%%%%%
%%%%%%%%%%%%%%%%%%%%%%%%%%%%%%%%%%%%%%%%%%%%%%%%%%%%%%%%%%%%
 
\vskip 0.70in 
 
%%%%%%%%%%%%%%%%%%%%%%%%%%%%%%%%%%%%%%%%%%%%%%%%%%%%%%%%%%%%%%%%%%%%%%%%%
%%%%%%%%%%%%%%%%%%%%%%%%%%%%%%%%%%%%%%%%%%%%%%%%%%%%%%%%%%%%%%%%%%%%%%%%%
%%%%%%%%%%%%%%%%%%%%%%%%%%%%%%%%%%%%%%%%%%%%%%%%%%%%%%%%%%%%%%%%%%%%%%%%%

\section{General depth zero}\label{generaldepthzero}

\medskip

\subsection{Preliminaries} \quad   Suppose $F$ is a facet of $\ScptB$.  Let $\Gk_{F}$ be the parahoric subgroup attached to $F$.  The quotient $\Gk_{F}/\Gk_{F}^{+}$ is the group of $\bF_{q}$-rational points of a reductive linear connected group.  We want to take a cuspidal representation $\sigma$ of  $\Gk_{F}/\Gk_{F}^{+}$, and use the corresponding idempotent $e_{\sigma}$ of $\Gk_{F}$ to define a Bernstein component of $\Gk$.  In order to do this we recall some preliminaries.

\medskip

To a parahoric subgroup $\Gk_{F}$, we can attach a Levi subgroup $\Mk \subset \Gk$.  We recall some results from section 6.2 in [{\reMPb}].    We take a maximal split $\mk$-torus $\Sk$ so that the apartment $\ScptA (\Sk )$ contains the facet $F$.  Then $\Sk$ gives rise to a maximal split $\bF_{q}$-torus in $\Gk_{F}/\Gk^{+}_{F}$.  We take the unique $\mk$-subtorus $\Ck$ of $\Sk$ so that $(\Ck \cap \Gk_{F}) / (\Ck \cap \Gk^{+}_{F})$ is the center of $\Gk_{F}/\Gk^{+}_{F}$.  The centralizer
\begin{equation}\label{attachedlevi}
\Mk = Z_{\Gk}(\Ck) \ \ \begin{cases}
\begin{tabular}{p{3.9in}}
is a Levi subgroup with center $\Ck$, and $\Mk_{F} = \Mk \cap \Gk_{F}$ is a maximal parahoric of $\Mk$, and  $\Gk_{F}/\Gk^{+}_{F} = \Mk_{F}/\Mk^{+}_{F}$.\\
\end{tabular}
\end{cases}
\end{equation}

\smallskip

\noindent{The} Levi subgroup $\Mk$ is defined up to a conjugation by $\Gk_{F}$.  Let $\Fk_{F} = N_{\Mk}(\Mk_{F}) \ (\supset \Ck)$ denote the normalizer of $\Mk_{F}$ in $\Mk$.  The quotient group $\Fk_{F}/\Ck$ is compact.  

\bigskip

We view an irreducible cuspidal representation $\sigma$ of the finite field group $\Gk_{F}/\Gk^{+}_{F}$ as also one of $\Mk_{F}/\Mk^{+}_{F}$, and we continue to write $\sigma$ for its inflation to an irreducible representation of $\Mk_{F}$.  Let
\smallskip
\begin{equation}\label{sigmaextension}
{\mathcal E}(\sigma ) \ = \ \begin{cases}
\begin{tabular}{p{3.9in}}
the collection of those irreducible representations of $\Fk_{F}$, up to equivalence, whose restrictions to $\Mk_{F}$ contains $\sigma$. \\
\end{tabular}
\end{cases}
\end{equation}
\medskip
\noindent{Define} $\tau_{1}, \, \tau_{2} \, \in {\mathcal E}(\sigma )$ to be equivalent if there is an unramified character $\chi$ of $\Mk$ so that $\tau_{2} = \tau_{1} \otimes \chi_{|_{\Fk_{F}}}$.  This equivalence relation partitions the collection ${\mathcal E}(\sigma )$ into finitely many equivalence classes.  We recall from [{\reMPb}]:

\bigskip

\begin{prop}\label{compactinductiona} \ (Proposition 6.6 in [{\reMPb}]) \quad \ Suppose $F \subset \ScptB (\Gk )$ is a facet and $\Mk$ is a Levi subgroup attached to $F$ as in \eqref{attachedlevi}, so $\Gk_{F}/\Gk^{+}_{F} = \Mk_{F}/\Mk^{+}_{F}$, and suppose $\sigma$ is the inflation to $\Mk_{F}$ of an irreducible cuspidal representation of $\Mk_{F}/\Mk^{+}_{F}$.
\smallskip
\begin{itemize}
\item[$\bullet$] Given $\tau \in {\mathcal E}(\sigma )$, the representation ${{\text{\rm{c-Ind}}}}^{\Mk}_{\Fk_{F}}(\tau )$ is an cuspidal representation of $\Mk$. 
\smallskip
\item[$\bullet$] Suppose $(\pi, V_{\pi})$ is an irreducible smooth representation of $\Mk$ which contains $\sigma$ upon restriction to $\Mk_{F}$.  Then $\pi$ is equivalent to ${{\text{\rm{c-Ind}}}}^{\Mk}_{\Fk_{F}}(\tau )$ for some $\tau \in {\mathcal E}(\sigma )$.
\end{itemize}
\end{prop}

\medskip

\begin{prop}\label{compactinductionb} \ (Proposition 5.3 in [{\reMPb}]) \quad Suppose:
  \begin{itemize}
  \item[$\bullet$] $F$ and $E$ are two facets of $\ScptB (\Gk)$, and $\tau$ and $\kappa$ are irreducible representations of $\Gk_{F}$ and $\Gk_{E}$ inflated from cuspidal representations of $\Gk_{F}/\Gk^{+}_{F}$ and $\Gk_{E}/\Gk^{+}_{E}$ respectively.
  \item[$\bullet$] $(\pi ,V_{\pi})$ is a smooth irreducible representation of $\Gk$ so that $\tau$ and $\kappa$ appear in the restriction of $\pi$ to  $\Gk_{F}$ and $\Gk_{E}$ respectively.
  \end{itemize}
Then, there exists $g \in G$ so that $\Gk_{F} \cap \Gk_{gE}$ surjects onto both  $\Gk_{F} / \Gk^{+}_{F}$ and $\Gk_{gE} / \Gk^{+}_{gE}$, and $\kappa = \tau \circ {\Ad (g)}$.
\end{prop}

\medskip 

\begin{prop}\label{compactinductionc}  \ (Proposition 6.2 in [{\reMPb}]) \quad Suppose $F, \, E \, \subset \ScptB (\Gk )$ are facets so that $\Gk_{F} \cap \Gk_{E}$ surjects onto $\Gk_{F}/\Gk^{+}_{F}$ and $\Gk_{E}/\Gk^{+}_{E}$ respectively, and $\sigma$ is an irreducible cuspidal representation of $\Gk_{F}/\Gk^{+}_{F} = \Gk_{E}/\Gk^{+}_{E}$. \  Let $\sigma_{F}$ and $\sigma_{E}$ denote respectively the inflation of $\tau$ to  $\Gk_{F}$ and $\Gk_{E}$.  \ If $(\pi, V_{\pi})$ is an irreducible smooth representation of $\Gk$ which contains $\sigma_{F}$ upon restriction to $\Gk_{F}$, then $\pi$ also contains $\sigma_{E}$ upon restriction to $\Gk_{E}$. 
\end{prop}

Recall from section 3.4 in [{\reMPb}], two parahoric subgroups $\Gk_{F}$ and $\Gk_{E}$ are associate if there exists $g \in \Gk$ so that $( \, \Gk_{F} \cap \Gk_{gE} \, )$ surjects onto both  $\Gk_{F} / \Gk^{+}_{F}$ and $\Gk_{gE} / \Gk^{+}_{gE}$.  In this situation, we get  $\Gk_{F} / \Gk^{+}_{F} = \Gk_{gE} / \Gk^{+}_{gE}$.   If $E$ is a facet in an apartment $\ScptA$, let $\ScptA \ScptS({\ScptA}, E)$ denote the minimal affine subspace of $\ScptA$ which contains $E$.  It is equal to the intersection of all the affine hyperplanes $H_{\pm \psi}$ which contain $E$.  We observe that if $E$ and $F$ belong to $\ScptA$, then  $\Gk_{F} \cap \Gk_{E}$ surjects onto both  $\Gk_{F} / \Gk^{+}_{F}$ and $\Gk_{E} / \Gk^{+}_{E}$ precisely when $\ScptA \ScptS({\ScptA}, E) \, = \, \ScptA \ScptS({\ScptA}, F)$.

\begin{lemma} \

\begin{itemize}
\item[$\bullet$] Associativity of parahoric subgroups in $\Gk$ is an equivalence relation.
\item[$\bullet$] If $F, \, E \, \subset \ScptB (\Gk )$ are associate facets, then the Levi subgroups attached to them by the above procedure are conjugate in $\Gk$.
\end{itemize}
\end{lemma}

\begin{proof}  \quad To prove the first statement, it suffices to prove transitivity.  Suppose facets $F_a$, and $F_b$ are associate.  Take $g \in \Gk$ so that  $(\Gk_{F_a} \cap \Gk_{gF_b})$ surjects onto $\Gk_{F_a} / \Gk^{+}_{F_a}$ and $\Gk_{gF_b} / \Gk^{+}_{gF_b}$.  This is equivalent to the equality of $\ScptA \ScptS({\ScptA}, F_{a} )$ and $\ScptA \ScptS({\ScptA}, F_{gb} )$ for any apartment $\ScptA$ containing both facets.  Similary, suppose $F_{b}$ and $F_{c}$ are associate.  This means there is an $h \in \Gk$ so that $( \, \Gk_{F_{b}} \cap \Gk_{hF_{c}} \, )$ surjects onto $\Gk_{F_{b}} / \Gk^{+}_{F_{b}}$ and $\Gk_{hF_{c}} / \Gk^{+}_{hF_{c}}$.  Hence $( \, \Gk_{gF_{b}} \cap \Gk_{ghF_{c}} \, )$ surjects onto $\Gk_{gF_{b}} / \Gk^{+}_{gF_{b}}$ and $\Gk_{ghF_{c}} / \Gk^{+}_{ghF_{c}}$.  Choose a chamber $C$ in $\ScptA$ which contains the facet $gF_{b}$.  The Iwahori subgroup $\Gk_{C}$ acts transitively on the apartments containing $C$, so there is a $k \in \Gk_{C}$ satisfying $k (gh F_{c}) \, \subset \ScptA$.  In $\ScptA$, we have,

$$
  \ScptA \ScptS({\ScptA}, F_{a}) \, = \, \ScptA \ScptS({\ScptA}, (gF_{b})) \, = \, \ScptA \ScptS({\ScptA}, (kgF_{b})) \, = \, \ScptA \ScptS({\ScptA}, (kghF_{c})) \ ,
$$

\noindent{which} means $F_a$ and $F_c$ are associate.
  
  \smallskip

  The second assertion follows from the first.  This is because the Levi subgroup attached to $F$ is the centralizer $C_{\Gk}(\Zk )$ of a lift $\Zk \, ( \, \subset \Sk \, )$ of the central torus $\ZFq$ of the finite field group $\Gk_{F}/\Gk^{+}_{F}$.  If $E$ and $F$ are associate, we can assume   $( \, \Gk_{F} \cap \Gk_{E} \, )$ surjects onto $\Gk_{F} / \Gk^{+}_{F}$ and $\Gk_{E} / \Gk^{+}_{E}$, and thereby canonically identify the two, and therefore their central torus, and hence lift.  The assertion follows.
\end{proof}

\bigskip

We recall the equivalence relation in the data used to define a Bernstein component:   Suppose $\Mk$ is a parabolic $\mk$-subgroup of $\Gk$, and $\pi_{a}$ and $\pi_{b}$ are  two irreducible cuspidal representations of $\Mk$.  Define
\begin{equation}
  \aligned
  \ \qquad \pi_{a} \sim \pi_{b} \quad &
\begin{cases}
\begin{tabular}{p{3.9in}}
  when there is a $g \in N_{\Gk}(\Mk )$ so that the representation $\pi^{g}_{a} := \pi_{a} \circ \Ad(g)$ is isomorphic to the representation  $\pi_{b}$.  \\
\end{tabular}
\end{cases}
    \endaligned
\end{equation}

\begin{lemma} \quad Suppose $\Mk$ is a parabolic subgroup of $\Gk$, and $F, \, E \subset \ScptB (\Gk )$ are facets contained in $\ScptB (\Mk )$ so that $\Mk \cap \Gk_{F}$ and $\Mk \cap \Gk_{E}$ are maximal parahoric subgroups of $\Mk$ and 
$$
\pi_{a} = {\text{\rm{c-Ind}}}^{\Mk}_{\Fk_{F}} (\tau ) \qquad , \qquad \pi_{b} = {\text{\rm{c-Ind}}}^{\Mk}_{\Fk_{E}} (\kappa )
$$
\noindent{are} equivalent irreducible cuspidal representations ($\pi_{a} \sim \pi_{b}$).  Then, the facets $F, \, E $ are associate.   
\end{lemma}

\begin{proof} \quad We note that  $\pi_{a} \circ \Ad(g) = {\text{\rm{c-Ind}}}^{\Mk}_{g^{-1}\Fk_{F}g} (\tau \circ \Ad(g)) = {\text{\rm{c-Ind}}}^{\Mk}_{\Fk_{g^{-1}F}} (\tau \circ \Ad(g))$.  Thus, the hypothesis that $\pi^{g}_{a}$ and $\pi_{b}$ are isomorphic representations of $\Mk$ means there is $h \in \Mk$ so that $\Fk_{g^{-1}F} = h \Fk_{E}h^{-1}$ and $\kappa = \tau \circ \Ad(g) \circ \Ad (h)$.  In particular, this means the two facets $F , \, E \, \subset \, \ScptB (\Gk )$ \ are associate.
\end{proof}

\bigskip

\subsection{Bernstein components}  \quad  Suppose $F$ is a facet in $\ScptB (\Gk )$, and $\sigma$ is the inflation to $\Gk_{F}$ of an irreducible cuspidal representation of $\Gk_{F}/\Gk^{+}_{F}$.  Let $\Mk$ be a Levi subgroup as in \eqref{attachedlevi}.  As mentioned there, the group  $\Mk_{F} = (\Mk \cap \Gk_{F})$ is a maximal parahoric subgroup of $\Mk$.  We take $\tau \in {\mathcal E}(\sigma )$, and consider the irreducible cuspidal representation $\pi (\tau) = {\text{\rm{c-Ind}}}^{\Mk}_{\Fk_{F}}(\tau )$.  For any facet $L \subset \ScptB (\Gk )$, let

\begin{equation}
\aligned
  \Theta_{{\pi},L} \ :&= \ {\text{\rm{$\Gk_{L}$-character of the restriction of $\pi (\tau)$ to $\Gk_{L}$}}} \, , \\
          f_{\tau , L} \ :&= \  \Theta_{{\pi},L} \star e_{\Gk^{+}_{L}} \, .
\endaligned
\end{equation}

\smallskip

\noindent{The} function $f_{\tau , L}$ has an expansion in terms of characters of $\Gk_{L}/\Gk^{+}_{L}$:

\begin{equation}
  f_{\tau , L} \ = {\underset {\kappa \in \big( \Gk_{L}/\Gk^{+}_{L} \big)^{\widehat{\ }}} \sum} m(\kappa ) \ \Theta_{\kappa} \, .
\end{equation}

\smallskip

\noindent{We} note:

\smallskip

\begin{itemize}
\item[$\bullet$] Any two $\tau, \, \tau' \, \in \, {\mathcal E}(\sigma)$ (see \eqref{sigmaextension}) produce representations $\pi (\tau)$ and $\pi (\tau ')$ with the same $\Gk_{L}$ spectrum, so $f_{\tau , L}$ does not depend on which $\tau$ is used.
\smallskip
\item[$\bullet$] The process is clearly canonical and so produces a system of functions $f_{\tau , L}$ on the collection of parahoric subgroups of $\Gk$ which is $\Gk$-equivariant.  Set
\medskip  
  \begin{equation}\label{parahoricblock}
    {Bk}(L) \ := \ {\text{\rm{characters of $\Gk_{L}$  which appear in $f_{\tau , L}$}}} \, .
\end{equation}
\smallskip
\item[$\bullet$] Propositions \eqref{compactinductiona}, \eqref{compactinductionb} and \eqref{compactinductionc} say a necessary and sufficient  condition for the function $f_{\tau , L}$ to be nonzero is that $\Gk_{L}$ contains a parahoric subgroup $\Gk_{E}$ with $E$ associate to $F$, i.e., $L$ is contained in a facet $E$ associate to $F$.  In this situation, we define the idempotent $e_{\tau , L}$ as:
\begin{equation}\label{parahoricidemponent}
      e_{\tau , L} \ := \ {\frac{1}{\# (\Gk_{L}/\Gk^{+}_{L})}} \ {\underset {\Theta_{\kappa} \in {Bk}(L)} \sum} \deg (\kappa ) \, \Theta_{\kappa} \ \star \ e_{\Gk^{+}_{L}} \ .
\end{equation}
\smallskip    
\noindent{We} call this idemponent a Peter--Weyl idempotent.  This canonical construction clearly yields a $\Gk$-equivariant system of idempotents.

\smallskip

\item[$\bullet$]  Suppose $E$ is a facet associate to $F$, and facets $K$, $L$ satisfy $K \subset L \subset E$, i.e., $\Gk_{E}/\Gk^{+}_{K} \subset \Gk_{L}/\Gk^{+}_{K}$ are parabolic subgroups of $\Gk_{K}/\Gk^{+}_{K}$.   The finite field results say:

\smallskip
  
\begin{itemize}
  
\item[(i)]  $e_{\tau , K} \ \star \ e_{\Gk^{+}_{L}} \ = \ e_{\tau , L}$.

\item[(ii)]  Under the hypothesis $K \subsetneq J \subset E$, i.e., $\Gk_{J}/\Gk^{+}_{K}$ is a proper parabolic of $\Gk_{K}/\Gk^{+}_{K}$, with unipotent radical $\Gk^{+}_{J}/\Gk^{+}_{K} \neq \{ 1 \}$, then 

  \smallskip

\begin{equation}
\big( \ {\underset    {\text{\rm{$K \subset L \subset E$}}}  \sum}   
(-1)^{\dim (L)} e_{\tau , L} \ \big) \ \star \ e_{\Gk^{+}_{J}} \ = \ {\text{\rm{zero function}}} \, . 
\end{equation}

\end{itemize}
\end{itemize}

\medskip

\begin{thm}\label{generaldepthzerotheorem}   Suppose $F$ is a facet and $\sigma$ is the inflation to $\Gk_{F}$ of an irreducible cuspidal representation of $\Gk_{F}/\Gk^{+}_{F}$.  Define idempotents as in \eqref{parahoricidemponent}.  \ Then, the alternating sum

$$
  {\underset {L \subset \ScptB (\Gk) } \sum} (-1)^{\dim (L)} e_{\tau , L} 
$$

\noindent{over} the facets of $\ScptB (\Gk)$ defines a $\Gk$-invariant essentially compact distribution.    
\end{thm}

\medskip

\begin{proof} \quad The proof is an extension of the proof in the Iwahori setting.  Fix a chamber $C_{0}$ and suppose $D$ is an arbitrary chamber.  We take $\Sk$ to be a maximal split torus so that $\ScptA (\Sk )$ contains both chambers.  Take $\Phi^{+} \subset \Phi (\Sk )$ to be a set of positive roots so that $D \subset S(C_{0},\Phi^{+})$.  As before, let $c(D)$, as in \eqref{chamberinout},  denote the set of outward oriented faces of $D$ and define $D_{+}$ and ${\mathcal F}_{+}(D)$ as in \eqref{min-facet} and \eqref{min-facet-set}, i.e., $D_{+} = {\underset {F \in c(D)} \medcap } \, F$ and ${\mathcal F}_{+}(D)$ is the set of facets which contain $D_{+}$.  Suppose $J$ is an open compact subgroup.  The key to adapting the Iwahori setting proof to the present one is to show

  $$
  {\underset {E \in {\mathcal F}_{+}(D)} \sum } \ (-1)^{\dim (E)} \, e_{\tau , E} \ \star e_{J} \ = \ {\text{\rm{zero function}}}
  $$
  when $\myht_{C_{0}}(D)$ is sufficiently large (dependent on $J$).  We may and do assume the open compact subgroup $J$ has the form $J = \Gk_{x_{0}, \rho}$ for $x_{0} \in \mypfacet (C_{0})$, and $\rho$ is a positive integer.

\bigskip

  \begin{itemize}
  \item[$\bullet$] Define ${\frcR}_{(\rho + 2)}$ as in \eqref{recessgroupone}.  \ For $y \in D \subset {\frcR}_{(\rho + 2)}$, we again have $\Gk_{x_{0},\rho }$ contains the group $V_{y}$ of \eqref{recessgrouptwo} and $V_{y}\Gk^{+}_{D_{+}} = \Gk^{+}_{D}$, so  $V_{y}\Gk^{+}_{D_{+}}/\Gk^{+}_{D_{+}}$ is the unipotent radical of the Borel subgroup $\Gk_{D}/\Gk^{+}_{D_{+}}$ of $\Gk_{D_{+}}/\Gk^{+}_{D_{+}}$. \ The analogue of the computations \eqref{recessgroupthree} and \eqref{recessgroupfour} are

  $$
    {\underset 
{E  \,  \in  \, {\mathcal F}_{+}(D)}
\sum } 
  (-1)^{{\text{\rm{dim}}}(E)} \ e_{\tau , E} \ \star \ e_{V_{y}} =  \ {\text{\rm{ zero function}}} \,
  $$

  \noindent{and}

  $$
    {\underset 
{E  \,  \in  \, {\mathcal F}_{+}(D)}
\sum }
  (-1)^{{\text{\rm{dim}}}(E)} \ e_{\tau , E} \ \star \ e_{\Gk_{x_{0},\rho}} =  \ {\text{\rm{ zero function}}} \, .
  $$

    \medskip

    A difference between the Iwahori setting and the general depth zero setting is the following: \ In the Iwahori setting, the individual convolution terms $e_{E} \, \star \, e_{V}$ are all nonzero, but their alternating sum is zero.  In the general setting, some of the individual convolutions $e_{\tau , E} \, \star \, e_{V}$ are zero due to the cuspidal assumption on $\tau$.

    \medskip
    
  \item[$\bullet$]  We again use the sets in \eqref{recessgroupsix} to partition ${\frcR}_{(\rho + 2)}$.  Fix $I \subset \Delta$.  With $V_{y}$ ($y \in I$) defined as in \eqref{recessgroupeight}, we have $V_{y}\Gk^{+}_{D_{+}}/\Gk^{+}_{D_{+}}$ is the unipotent radical of a parabolic subgroup of  $\Gk_{D_{+}}/\Gk^{+}_{D_{+}}$, i.e., there is a $K \in {\mathcal F}_{+}(D)$ so that $V_{y}\Gk^{+}_{D_{+}} = \Gk^{+}_{K}$.    With finitely many exceptions,  the unipotent radical is not $\{ \, 1 \, \}$, and so Corollary \eqref{finitefieldcorb} applies to give 
    ${\underset {E  \,  \in  \, {\mathcal F}_{+}(D)} \sum } (-1)^{{\text{\rm{dim}}}(E)} \ e_{\tau , E} \, \star \, e_{V_{y}}$ is the zero function, and therefore ${\underset {E  \,  \in  \, {\mathcal F}_{+}(D)} \sum } (-1)^{{\text{\rm{dim}}}(E)} \ e_{\tau , E} \, \star \, e_{\Gk_{x_{0},\rho}}$ vanishes too.  We handle the finite number of exceptions by replacing $(\rho + 2)$ by a larger value to exclude these finitely many exceptions.  Again, some of the convolutions $e_{\tau , E} \, \star \, e_{\Gk_{x_{0},\rho}}$ vanish due to the cuspidality of $\tau$. The Theorem follows.

  \end{itemize}
\end{proof}

\begin{cor}\label{generaldepthzerocorollary} The above distribution is the projector to the Bernstein component of $(\Mk , {\text{\rm{c-Ind}}}^{\Gk}_{{\mathcal F}_{F}} (\tau ))$.
\end{cor}

\begin{proof}

We replace the facet $F$ by an associate one in the `base' chamber $C_{0}$.  Then, as extensions of the Iwahori situation we have: 

\medskip

\begin{itemize}
\item[(i)] \quad   ${\underset {E  \, \subset  \, C_{0}} \sum } (-1)^{{\text{\rm{dim}}}(E)} \ e_{\tau , E} \, \star \, e_{\Gk^{+}_{F}} \ = \ e_{\tau , F}$

\smallskip

\item[(ii)] \quad  For any chamber $D \, \neq \, C_{0}$:
$$
{\underset {E  \, \subset  \, {\mathcal F}_{+}(D)} \sum } (-1)^{{\text{\rm{dim}}}(E)} \ e_{\tau , E} \, \star \, e_{\Gk^{+}_{F}} \ = \ {\text{\rm{zero function}}} \, .
$$

\smallskip

\item[(iii)] \quad  
$$
\big( \, {\underset {E  \, \subset  \, \ScptB (\Gk)} \sum } (-1)^{{\text{\rm{dim}}}(E)} \ e_{\tau , E} \, \big) \ \star \ e_{\Gk^{+}_{F}} \ = \ e_{\tau , F}
$$

\end{itemize}
%\vskip 2.0in

\bigskip

Set $P \, = \, \big( \, {\underset {E  \, \subset  \, \ScptB (\Gk)} \sum } (-1)^{{\text{\rm{dim}}}(E)} \ e_{\tau , E} \, \big)$.  If $(\pi , V_{\pi})$ is an irreducible smooth representation of $\Gk$, by (iii), we have

$$
\aligned
\pi ( \, e_{\tau , F} \, ) \ &= \ \pi (\, P \ \star \ e_{\Gk^{+}_{F}} ) \ = \ 
\pi ( \, P \, ) \ \pi ( \, e_{\Gk^{+}_{F}} \, ) \ .
\endaligned
$$

\noindent{When} $\pi$ belongs to the Bernstein component $(\Mk , {\text{\rm{c-Ind}}}^{\Gk}_{{\mathcal F}_{F}} (\tau ))$ the left size is nonzero and we conclude $\pi (P)$ is the identity.  Additionally, the fact that the individual terms of $P$ have the form $e_{\tau , E}$ means $\pi (P)$ is zero unless there exists an facet $E$ so that $\pi ( e_{\tau , E} )$ is nonzero.  Hence, $P$ is the projector to the Bernstein component.

\end{proof}

\bigskip

\subsection{The depth zero projector} \quad When we sum over all the Bernstein components of depth zero we obtain{\,}:
\medskip
\begin{itemize}
\item[$\bullet$]  \ $P_{0} \ = \, {\underset {\rho (\Omega ) \, = \, 0} \sum} P(\Omega)$ \ , which is, by definition, the depth zero projector.
  \smallskip
\item[$\bullet$] \ For any facet $F$, the sum of the Peter--Weyl idempotents is ${\frac{1}{\meas (\Gk_{F} )}}$ times the character of the regular representation of $\Gk_{F}/\Gk^{+}_{F}$.  This is the idempotent $e_{\Gk^{+}_{F}} \, = \, {\frac{1}{\meas (\Gk^{+}_{F} )}} \, 1_{\Gk^{+}_{F}}${\,}.
\end{itemize}

\smallskip

\noindent{The} resulting Euler-Poincar{\'{e}} formula for $P_{0}$ is exactly the one in [{\reBKV}].

%%%%%%%%%%%%%%%%%%%%%%%%%%%%%%%%%%%%%%%%%%%%%%%%%%%%%%%%%%%%%%%%%%%%%%%%%%%%%%%
%%%%%%%%%%%%%%%%%%%%%%%%%%%%%%%%%%%%%%%%%%%%%%%%%%%%%%%%%%%%%%%%%%%%%%%%%%%%%%%
%%%%%%%%%%%%%%%%%%%%%%%%%%%%%%%%%%%%%%%%%%%%%%%%%%%%%%%%%%%%%%%%%%%%%%%%%%%%%%%
%%%%%%%%%%%%%%%%%%%%%%%%%%%%%%%%%%%%%%%%%%%%%%%%%%%%%%%%%%%%%%%%%%%%%%%%%%%%%%%
%%%%%%%%%%%%%%%%%%%%%%%%%%%%%%%%%%%%%%%%%%%%%%%%%%%%%%%%%%%%%%%%%%%%%%%%%%%%%%%
 
\vskip 0.70in 
 
%%%%%%%%%%%%%%%%%%%%%%%%%%%%%%%%%%%%%%%%%%%%%%%%%%%%%%%%%%%%%%%%%%%%%%%%%
%%%%%%%%%%%%%%%%%%%%%%%%%%%%%%%%%%%%%%%%%%%%%%%%%%%%%%%%%%%%%%%%%%%%%%%%%
%%%%%%%%%%%%%%%%%%%%%%%%%%%%%%%%%%%%%%%%%%%%%%%%%%%%%%%%%%%%%%%%%%%%%%%%%

\section{Nonsplit groups}\label{nonsplit}

\medskip

In this section, we explain the modifications needed in the proofs of sections \eqref{iwahoribernstein} and  \eqref{generaldepthzero} so that they apply when the $\mk$-defined group $\mG$ is nonsplit.  We assume $\mG$ is connected, absolutely quasisimple.  Set $\Gk = \mG (\mk)$.  Let $\mK$ be the maximal unramified extension of $\mk$, and let $\myGal (\mK/\mk)$ denote the Galois group.  

\medskip

\subsection{{\,}} \quad We recall (see [{\reTa}:\S1.10]) there exists a torus $\SamS$ defined over $\mk$ satisfying{\,}: \ (i) {\,}$\SamS$ is a maximal split $\mK$-torus, \ and \ (ii) {\,}  $\mS \, := \, \SamS^{\myGal (\mK / \mk )}$ is a maximal split $\mk$-torus.  We also recall the result of Steinberg that $\mG$ is quasi-split over $\mK$, and therefore, the centralizer $\SamZ \, = \, C_{\mG}(\SamS )$ is a maximal $\mk$-torus.  Furthermore{\,}: 
\smallskip
\begin{itemize}

\item[$\bullet$] \quad Since $\mG$ is assumed to be absolutely quasisimple, the Bruhat--Tits building $\ScptB (\mG (\mK ))$ is a simplicial complex, and both $\mG (\mK)$ and $\myGal (\mK /\mk )$  act by simplicial automorphisms.   The building $\ScptB (\Gk )$ is the $\myGal (\mK/\mk)$-fixed points of $\ScptB (\mG (\mK ))$. 
\medskip
\item[$\bullet$] \quad Let $\ScptA (\mS (\mK ))$ be the apartment of $\mS (\mK )$, and $\Psi (\mS (\mK ))$, the corresponding system of affine roots.   The fixed points $\ScptA (\mS (\mK ))^{\myGal (\mK /\mk )}$ and $\ScptB (\mG (\mK ))^{\myGal (\mK /\mk )}$ are identified (defined) as the building $\ScptB (\Gk )$ and apartment $\ScptA (\Sk)$ {\,}($\Sk = \mS (\mk ) = \mS (\mK)^{\myGal (\mK / \mk )}$).  
\smallskip
Let $\mM_{\text{\rm{o}}} = C_{\mG}(\mS )$, a minimal Levi $\mk$-subgroup.  Given a choice of positive roots $\Phi^{+}$, let $\mP_{\text{\rm{o}}} = \mM_{\text{\rm{o}}} \mU_{\text{\rm{o}}}$ be the corresponding minimal parabolic $\mk$-subgroup.  We set $\Mk_{\text{\rm{o}}} = \mM_{\text{\rm{o}}} (\mk)$, $\Uk_{\text{\rm{o}}} = \mU_{\text{\rm{o}}} (\mk)$, and $\Pk_{\text{\rm{o}}} = \Mk_{\text{\rm{o}}} \Uk_{\text{\rm{o}}} = \mP_{\text{\rm{o}}} (\mk )$.
\medskip
\item[$\bullet$]  \quad  The affine root system $\Psi = \Psi (\Sk )$ on $\ScptA$ consists of all nonconstant restrictions to $\ScptA$ of affine roots in $\Psi( \SamS (\mK ))$ (see [{\reTa}:{\S}1.10.1]).   If $E$ is a facet of $\ScptB$, and $x,y \, \in \, \mypfacet{(E)}$, then
  $$
\Gk_{x,0} \ = \ \Gk_{y,0} \qquad {\text{\rm{and}}} \qquad \Gk_{x,0^{+}} \ = \ \Gk_{y,0^{+}} \ . 
  $$
Because of these equalities, we denote the common subgroups as $\Gk_{E}$ and $\Gk^{+}_{E}$.  We note that if $F$ is a subfacet of $E$, then
$$
\Gk_{F} \ \supset \ \Gk_{E} \ \supset \ \Gk^{+}_{E} \ \supset \  \Gk^{+}_{F}  \ ,
$$
\noindent{and} $\Gk_{E}/\Gk^{+}_{F}$ is a parabolic subgroup of $\Gk_{F}/\Gk^{+}_{F}$. 
\medskip
\item[$\bullet$]  \quad  Let $\Zk = \SamZ (\mk)$, and let $\Zk_{c}$ be the maximal bounded (compact) subgroup of $\Zk$.  The group $N_{\Gk}(\Zk )$ acts as orthogonal affine maps on $\ScptA$ with $\Zk_{c}$ acting trivially.   We call $N_{\Gk}(\Zk ) / \Zk_{c}$ the extended affine Weyl group of $\Gk$.  The Coxeter group $\myCoxeter_{\Psi}$ of $\Psi (\ScptA )$, i.e, the symmetry group of $\ScptA$ generated by reflections in the hyperplanes $H_{\psi}$ {\,}($\psi \in \Psi$) is a finite index subgroup of   $N_{\Gk}(\Zk ) / \Zk_{c}$.   The action of $\myCoxeter_{\Psi}$ is transitive on the chambers of $\ScptA$ (see [{\reTa}:{\S}1.8]). 
\medskip
For $x \in \ScptA$, define
\smallskip
\begin{equation}
\myCoxeter_{x} \ := \ \begin{cases}
\begin{tabular}{p{4.0in}}
symmetry group of $\ScptA$ generated by reflections across affine hyperplanes $H_{\psi}$ containing $x$, i.e., $\psi (x) = 0$.  \\
\end{tabular}
\end{cases}
\end{equation}
\smallskip
\noindent{Set} $\Phi \, := \, \{ \, \mygrad (\psi) \ | \ \psi \in \Psi \, \}$, a possibly nonreduced root system, and let $\myCoxeter_{\Phi}$ denote the Coxeter group of $\Phi$.  We recall a point $x$ is called \ {\it special} \ if $W_{x}$ and $W_{\Phi}$ are the same.  Special points always exists (see [{\reTa}:{\S}1.9]).  When $x$ is special, the group $N_{\Gk}(\Zk )/\Zk_{\text{\rm{o}}}$ is the semidirect product of the group $W_{x}$ and the normal subgroup of translations $X = \Zk/\Zk_{c}$.  Similarly, $W_{\Psi}$ is  semidirect product of the group of $W_{x}$ and translation subgroup of $W_{\Psi}$.  Let $\Phi^{\myred}$ denote the reduced root system of $\Phi$.  Fix a chamber $C_{0}$ of $\ScptA$.  For $\pm \gamma \in \Phi^{\myred}$ define $H^{\pm \gamma}_{C_{0}}$ as in \eqref{bruhatheightpieces}.  Then Lemma \eqref{bruhatheight} holds.
\vskip 0.05in
\item[$\bullet$]  \quad The proof of Proposition \eqref{keypermissible} used only the affine hyperplanes $H_{\pm \psi}$, and it is valid in the nonsplit situation; so, Proposition \eqref{keypermissible} holds.  We use \eqref{csector} to define the $C_{0}$-based sector $S(C_{0},\Phi^{+})$.
\medskip
\item[$\bullet$]  \quad  The results of [{\reMPb}] stated in section \eqref{generaldepthzero} hold in the nonsplit situation.  

\medskip

\end{itemize}

\begin{thm}\label{maindepthzero}   Suppose $\mG$ is a connected absolutely quasisimple $\mk$-group.  Let $\Gk = \mG (\mk )$, and let $\ScptB = \ScptB (\Gk )$ be the Bruhat--Tits building.   Suppose $F$ is a facet of $\ScptB$, and $\sigma$ is the inflation to $\Gk_{F}$ of an irreducible cuspidal representation of $\Gk_{F}/\Gk^{+}_{F}$.  Take $\tau \in {\mathcal E}(\sigma)$ as in Proposition \eqref{compactinductiona}, and define a $\Gk$-equivariant system of idempotents as in \eqref{parahoricidemponent}.  \ Then,

\begin{itemize}
\item[$\bullet$] \ The alternating sum
\begin{equation}\label{nonsplitalternatingsum}
 P \ = \  {\underset {L \subset \ScptB (\Gk) } \sum} (-1)^{\dim (L)} e_{\tau , L} 
\end{equation}
\noindent{over} the facets of $\ScptB (\Gk)$ defines a $\Gk$-invariant essentially compact distribution.
\smallskip
\item[$\bullet$] With Levi subgroup $\Mk$ defined as in \eqref{attachedlevi}, the distribution $P$ is the projector to the Bernstein component of $(\Mk , {\text{\rm{c-Ind}}}^{\Gk}_{{\mathcal F}_{F}} (\tau ))$.
\end{itemize}

\end{thm}

\smallskip

\begin{proof} \quad The proof of Theorem \eqref{maindepthzero} is adapted from those of Theorems \eqref{dplus} and \eqref{mainiwahori}.  We fix a chamber $C_{0}$.    Suppose $D \, (\neq C_{0})$ is any other chamber.  For an arbitrary fixed open compact subgroup $J$ we need to show $P \star e_{J} \in C^{\infty}_{c}(\Gk )$.  It suffices to show
\smallskip  
$$
  {\underset {E \in {\mathcal F}_{+}(D)} \sum } \ (-1)^{\dim (E)} \, e_{\tau , E} \ \star e_{J} \ = \ {\text{\rm{zero function}}}
$$
\smallskip
\noindent{when} $\myht_{C_{0}}(D)$ is sufficiently large.  Fix $x_{\text{\rm{o}}} \in C_{0}$, and take $\rho \in \bR_{>0}$ sufficiently large so that $\Gk_{x_{\text{\rm{o}}},\rho } \subset J$.  We replace $J$ with  $\Gk_{x_{\text{\rm{o}}},\rho }$
\smallskip
Take $\SamS \supset \mS$ ($\mS =   \SamS^{\myGal (\mK/\mk )}$), as above, so that $D$ is in $\ScptA = \ScptA (\Sk )$, and take a positive system of roots $\Phi^{+} \subset \Phi (\Sk )$, so that $S(C_{0}, \Phi^{+})$ contains $D$.  
\medskip
Let $N$ be a sufficiently large integer so that for any simple root $\alpha \in \Delta ((\Phi^{\myred})^{+})$, when $C_{0}$ and $D$ are separated by $N$ affine hyperplanes perpendicular to $\alpha$, then:
\smallskip
\begin{equation}
\begin{tabular}{p{5.2in}}
$\forall$ \ $\psi \in \Psi = \Psi (\Sk )$, satisfying $\lambda_{\alpha}(\mygrad (\psi )) > 0$, i.e., when $\mygrad (\psi )$ is expressed as a linear (nonnegative) combination simple roots the $\alpha$ coefficient is nonzero, then $\forall$ $x \in C_{0}$ and $\forall$ $y \in D$, we have $(\psi (y) - \psi (x)) \, > \, \rho$. \\
\end{tabular}
\end{equation}
\noindent{If} such a $\psi$ vanishes on $D_{+}$, and we take $y \in D_{+}$, we see $-\psi (x) \, = \, (\psi (y) - \psi (x) ) > \rho$,  This means $(\Gk_{x,\rho} \cap \Uk_{-\gamma} ) \supset \Xk_{-\psi}$, where $\gamma = \mygrad (\psi )$.   We note in the nonsplit situation, the root group $\Uk_{-\gamma}$ and the affine root group $\Xk_{-\psi}$ may be noncommutative.
\medskip
Define ${\frcR}_{N}$ as in \eqref{recessgroupone}; roughly the set of chambers which are Bruhat distance at least $N$ from the walls of $S(C_{0},\Phi^{+})$.  The above says for any chamber $D$ in  ${\frcR}_{N}$, the subgroup  $\Gk_{x,\rho}$ contains the subgroup
\smallskip
\begin{equation}\label{nonsplitrgtwo}
  V \ := \ {\underset { \psi \in \Psi (D_{+},\Phi^{+}) }
            \prod } \Xk_{-\psi} \, , 
\end{equation}

\noindent{and} $V \Gk^{+}_{D_{+}}/\Gk^{+}_{D_{+}} \subset \Gk_{D_{+}}/\Gk^{+}_{D_{+}}$ is the unipotent radical of the Borel subgroup $\Gk_{D_{+}}/\Gk^{+}_{D_{+}}$.  We then have the analogue of \eqref{recessgroupthree} and \eqref{recessgroupfour}, i.e.,

$$
{\underset 
{E  \,  \in  \, {\mathcal F}_{+}(D)}
\sum } 
(-1)^{{\text{\rm{dim}}}(E)} \ e_{E} \ \star e_{\Gk_{x_{0},\rho}} \ =  \ {\text{\rm{zero function}}} \, .
$$
\medskip
The situation when $D$ is a chamber in $S(C_{0},\Phi^{+}) \, \backslash \, \frcR_{N}$ is handled by defining sets ${\frcR}_{\{ I, k \}}$ as in \eqref{recessgroupsix}, and adapting the argument.  We omit the very similar details.  This completes the proof the alternating sum \eqref{nonsplitalternatingsum} defines a Bernstein center distribution.

\medskip

To establish that $P$ is the projector, we use Corollary \eqref{finitefieldcora} and adapt the proof of Corollary \eqref{iwahoriprojectorcor} to deduce{\,}:

\begin{itemize}
  
\item[(i)] \quad   ${\underset {E  \, \subset  \, C_{0}} \sum } (-1)^{{\text{\rm{dim}}}(E)} \ e_{\tau , E} \, \star \, e_{\Gk^{+}_{F}} \ = \ e_{\tau , F}$

\smallskip

\item[(ii)] \quad  For any chamber $D \, \neq \, C_{0}$:
$$
{\underset {E  \, \subset  \, {\mathcal F}_{+}(D)} \sum } (-1)^{{\text{\rm{dim}}}(E)} \ e_{\tau , E} \, \star \, e_{\Gk^{+}_{F}} \ = \ {\text{\rm{zero function}}} \, .
$$
\end{itemize}

\noindent{That} $P$ is the projector then follows. 

\end{proof}

\vskip 0.50in

\vskip 0.70in 
 
%%%%%%%%%%%%%%%%%%%%%%%%%%%%%%%%%%%%%%%%%%%%%%%%%%%%%%%%%%%%%%%%%%%%%%%%%
%%%%%%%%%%%%%%%%%%%%%%%%%%%%%%%%%%%%%%%%%%%%%%%%%%%%%%%%%%%%%%%%%%%%%%%%%
%%%%%%%%%%%%%%%%%%%%%%%%%%%%%%%%%%%%%%%%%%%%%%%%%%%%%%%%%%%%%%%%%%%%%%%%%

\section{Appendix I}\label{appendix}

\subsection{{\,}} \quad We show here how the argument to establish the Euler-Poincar{\'{e}} formula for a depth zero Bernstein projector also applies, when $r$ is a positive integer to the depth $r$ projector considered by   Bezrukavnikov--Kazhdan--Varshavsky in [{\reBKV}].  To simplify the exposition we assume the $\mk$-group $\mG$ is split absolutely quasisimple, and leave the necessary minor modifications for the nonsplit setting to the reader.

\medskip

Suppose $r > 0$ is integral, and $F$ is a facet.  Then for any two points $x, \, y \, \in \, \mypfacet (F)$, we have 
$$
\Gk_{x,r}  \ = \ \Gk_{y,r} \qquad {\text{\rm{and}}} \qquad \Gk_{x,r^{+}}  \ = \ \Gk_{y,r^{+}} \ .
$$ 
We therefore, for convenience denote these groups as $\Gk_{F,r}$ and $\Gk_{F,r^{+}}$ respectively.  If $E \subset F$ are two facets of $\ScptB$,  then 
\begin{equation}\label{appendixone}
\Gk_{E,r} \ \supset \ \Gk_{F,r} \ \supset \ \Gk_{F,r^{+}} \ \supset \ \Gk_{E,r^{+}} \ .
\end{equation} 

\medskip

Since we have fixed $r$, when $E$ is a facet, denote by $e_{E,r^{+}}$ the idempotent for the trivial representation of $\Gk_{E,r^{+}}$ (note $r^{+}$ and not $r$).   It is obvious from \eqref{appendixone} that 

\begin{equation}\label{appendixtwo}
  \forall \ \ E \subset F \ \ : \qquad  e_{E,r^{+}} \ \star \ e_{F,r^{+}} \ = \ e_{F,r^{+}} \ .
\end{equation}

\medskip

\begin{lemma}\label{appendixthree}  Suppose $D$ is a chamber and ${\mathcal F}$ is a set of faces of $D$ satisfying $1 \, \le \, \# ({\mathcal F}) \, \le \, \ell$ \ (recall $\ell$ is the rank of $\mG$).  Set

$$
  E \ := \ {\underset {F \in {\mathcal F}} \bigcap } \ F \qquad {\text{\rm{and}}} \qquad {\mathcal C} \ := \ \{ \ {\text{\rm{facet}}} \ K \subset D \ | \ E \subset K \ \} \, .
$$

\smallskip

\noindent{Suppose} $V \in {\mathcal C}$ satisfies $V \neq E$, then
\smallskip
$$
    \Big( \ {\underset {K \in {\mathcal C}} \sum} \ (-1)^{\dim (K)} \ e_{K,r^{+}} \ \Big) \ \star \ e_{V,r^{+}}  \ = \ {\text{\rm{zero function}}} \, .
$$
\end{lemma}

\begin{proof} \quad Follows from \eqref{appendixtwo}.
\end{proof}

\begin{thm}\label{appendixfour}  Suppose $r$ is a positive integer.  The Euler--Poincar{\'{e}} sum

$$
  {\underset {L \subset \ScptB (\Gk) } \sum} (-1)^{\dim (L)} \ e_{L,r^{+}} 
$$

\noindent{over} the facets of $\ScptB (\Gk)$ is a $\Gk$-invariant essentially compact distribution equal to the projector $P_r$ to ${\underset {\rho (\Omega ) \le r} \medcup } \, \Omega$.
\end{thm}

\begin{proof} \quad As already indicated above, the proof is an adaptation of the proof of Theorem \eqref{mainiwahori}.   
  We fix a base chamber $C_{0}$, and $x_{0} \in \mypfacet{(C_{0})}$.  We show, for any integer $\rho \ge 1$ and $J := \Gk_{x_{0},\rho}$, the convolution $\big( \, {\underset {L \subset \ScptB (\Gk) } \sum} (-1)^{\dim (L)} e_{L,r^{+}} \, ) \star e_{J}$ is in $C^{\infty}_{c}(\Gk )$.  To do this, suppose $D \, ( \, \ne C_{0} \, )$ is a chamber.  It suffices to show the convolution $\big( \, {\underset {K \subset {\mathcal F}_{+}(D) } \sum} (-1)^{\dim (K)} e_{K,r^{+}} \, ) \star e_{J}$ is the zero function provided $\myht_{C_{0}}(D)$ is sufficiently large.  \ To do this, we choose a maximal split torus $\Sk$ so that $C_{0}$ and $D$ belong to the apartment $\ScptA = \ScptA (\Sk )$.  In the proof of Theorem \eqref{mainiwahori}, we considered the set ${\frcR}_{\rho + 2}$ defined in \eqref{recessgroupone}.  It works here too.   Let $c(D)$, with respect to $C_{0}$, be the child faces of $D$, and set{\,}:
\begin{itemize}
\item[(i)] \ \ $D_{+}$ to be the intersection of all the faces of $c(D)$.
\smallskip
\item[(ii)] \ \ ${\mathcal F}_{+}(D)$ to be the set of facets of $D$ which contain $D_{+}$.  
\end{itemize}

\noindent{Suppose} $y \in \mypfacet{(D_{+})}$, and $\psi (y) =r$, i.e., $\Gk_{y,r} \supset \Xk_{\psi}$.  Then, 
\begin{itemize}
\item[(i)] \ \ $(-\psi + 2r)(y) = r$, so $\Gk_{y,r} \supset \Xk_{-\psi+2r}$ too.
\smallskip
\item[(ii)] \ \ $(\psi (y) - \psi (x_{0})) > \rho$; therefore, $(-\psi + 2r)(x) \ge \rho + r \ge r$; thus, $\Gk_{x_{0},\rho} \supset \Xk_{-\psi+2r}${\,},
\end{itemize}
\noindent{and} so,

$$
V_{y} \ := \ {\underset {\psi_{|_{D_{+}}} \equiv {\ } r \ , \ \mygrad(\psi) > 0} \prod } \Xk_{-\psi+2r} \ 
$$

\noindent{is} a subgroup contained in $\Gk_{x_{0}, \rho}$, and  $V_{y}{\,}\Gk_{y, r^{+}} = \Gk_{D,r^{+}}$. Then,  

$$
\aligned
{\underset 
{K \,  \in  \, {\mathcal F}_{+}(D)}
\sum } 
(-1)^{{\text{\rm{dim}}}(K)} \ e_{K,r^{+}} &\ \star \ e_{V_{y}} \ =  {\underset 
{K  \,  \in  \, {\mathcal F}_{+}(D)}
\sum } \ (-1)^{{\text{\rm{dim}}}(K)} \ ( \ e_{K,r^{+}} \ \star \ e_{\Gk_{y,r^{+}}} ) \ \star \ e_{V_{y}} \\
&\ = {\underset 
{K  \,  \in  \, {\mathcal F}_{+}(D)}
\sum } \ (-1)^{{\text{\rm{dim}}}(K)} \ e_{K,r^{+}} \ \star \ ( \ e_{\Gk_{y,r^{+}}} \ \star \ e_{V_{y}} \ ) \\
&\ = {\underset
{K  \,  \in  \, {\mathcal F}_{+}(D)}
\sum } \ (-1)^{{\text{\rm{dim}}}(K)} \ e_{K,r^{+}} \ \star \ ( \, e_{(\Gk_{y,r^{+}} \, V_{y} )} \, ) \\
&\ = {\underset
{K  \,  \in  \, {\mathcal F}_{+}(D)}
\sum } \ (-1)^{{\text{\rm{dim}}}(K)} \ e_{K,r^{+}} \ \star \ ( \, e_{\Gk_{D,r^{+}}} \, ) \\
&\ =  \ {\text{\rm{ zero function}}} \, . \\ 
\endaligned
$$

\noindent{So}, under the assumption $D \subset \frcR_{(\rho+2)}$, we see

$$
\aligned
{\underset 
{K  \,  \in  \, {\mathcal F}_{+}(D)}
\sum } 
(-1)^{{\text{\rm{dim}}}(K)} \ e_{K,r^{+}} &\ \star e_{\Gk_{x_{0},\rho}} \ =  \ {\text{\rm{ zero function}}} \, . \\ 
\endaligned
$$

\medskip

To address the situation when $D$ is in $\Sk (C_{0} , \Phi^{+} ) \backslash {\frcR}_{\rho + 2}$ we partition $S(C_{0}, \Phi^{+} )$ using the subsets in \eqref{recessgroupsix} with $k \, = \, \rho+2$.  With the obvious modifications, the proof there then applies.  We conclude \eqref{recessgroupfour} holds provided  $\myht_{C_{0}}(D)$ is sufficiently large.  The completion of the proof the Theorem using the `ball of radius m' defined in  \eqref{ballofradiusm} is clear.

\bigskip

It still remains to identify the distribution $P$ of Theorem \eqref{appendixfour} is the depth $r$ projector.  Suppose $( \pi , V_{\pi} )$ is a smooth irreducible representation:

\smallskip

\begin{itemize}
\item[$\bullet$]  Since  $( \pi ( e_{F,r^{+}} )) (V_{\pi})  \subset V^{\Gk_{_{F,r^{+}}}}_{\pi}$, it is clear  $\pi (P)$ is zero  when $\rho (\pi) > r$.   

\smallskip

\item[$\bullet$] Suppose $\rho (\pi ) \le r$.  Choose a nonzero $v \in V_{\pi}$ which is fixed by some $\Gk_{K,r^{+}}$.  We also take $M$ sufficiently large so that $K \subset \myBall (C_{0},M)$, and for all $j \ge 0${\,}:
\smallskip  
$$
\aligned
    {\underset {F \subset \myBall (C_{0},M+j)} \sum } (-1)^{\dim (F)} \, e_{F,r^{+}} \star e_{K,r^{+}} \ &= \ {\underset {F \subset \myBall (C_{0},M)} \sum } (-1)^{\dim (F)} \, e_{F,r^{+}} \star e_{K,r^{+}} \  \\
    &= \ P \star e_{K,r^{+}} \ .
\endaligned    
$$
\smallskip
\noindent{Then}, $\pi (P) (v) = \pi (P \star e_{K,r^{+}} ) (v)$.  The system of idempotents satisfy the three properties axiomatized in Definition 2.1 of [{\reMS}]; hence, 
$$
{\underset {F \subset \myBall (C_{0},M)} \sum } (-1)^{\dim (F)} \, \pi (e_{F,r^{+}}) \quad {\text{\rm{projects to}}} \quad  {\underset {x \in \myBall {(C_{0},M)}^{\text{\rm{o}}} } \sum } \, (\pi (e_{x,,r^{+}})) (V_{\pi}) \ .
$$  

\end{itemize}
\smallskip
\noindent{We} deduce $\pi (P) (v) = v$, and so $P$ must be the depth $r$ projector.

\end{proof}

%%%%%%%%%%%%%%%%%%%%%%%%%%%%%%%%%%%%%%%%%%%%%%%%%%%%%%%%%%%%%%%%%%%%%%%%%%%%%%%
%%%%%%%%%%%%%%%%%%%%%%%%%%%%%%%%%%%%%%%%%%%%%%%%%%%%%%%%%%%%%%%%%%%%%%%%%%%%%%%
%%%%%%%%%%%%%%%%%%%%%%%%%%%%%%%%%%%%%%%%%%%%%%%%%%%%%%%%%%%%%%%%%%%%%%%%%%%%%%%
%%%%%%%%%%%%%%%%%%%%%%%%%%%%%%%%%%%%%%%%%%%%%%%%%%%%%%%%%%%%%%%%%%%%%%%%%%%%%%%
%%%%%%%%%%%%%%%%%%%%%%%%%%%%%%%%%%%%%%%%%%%%%%%%%%%%%%%%%%%%%%%%%%%%%%%%%%%%%%%
 
\vskip 0.70in 
 
%%%%%%%%%%%%%%%%%%%%%%%%%%%%%%%%%%%%%%%%%%%%%%%%%%%%%%%%%%%%%%%%%%%%%%%%%
%%%%%%%%%%%%%%%%%%%%%%%%%%%%%%%%%%%%%%%%%%%%%%%%%%%%%%%%%%%%%%%%%%%%%%%%%
%%%%%%%%%%%%%%%%%%%%%%%%%%%%%%%%%%%%%%%%%%%%%%%%%%%%%%%%%%%%%%%%%%%%%%%%%

\section{Acknowledgments}

\medskip

Parts of this work was done during visits by the first and third authors to the University of Utah Mathematics Department in Summer 2015, by the first author to the HKUST Mathematics Department in Summer 2016, and by the third author to the University of Oxford Mathematics Department in Summer 2016.  The Departments are thanked for their hospitality.

%%%%%%%%%%%%%%%%%%%%%%%%%%%%%%%%%%%%%%%%%%%%%%%%%%%%%%%%%%%%%%%%%%%%%%%%%%%
%%%%%%%%%%%%%%%%%%%%%%%%%%%%%%%%%%%%%%%%%%%%%%%%%%%%%%%%%%%%%%%%%%%%%%%%%%%
%%%%%%%%%%%%%%%%%%%%%%%%%%%%%%%%%%%%%%%%%%%%%%%%%%%%%%%%%%%%%%%%%%%%%%%%%%%
%%%%%%%%%%%%%%%%%%%%%%%%%%%%%%%%%%%%%%%%%%%%%%%%%%%%%%%%%%%%%%%%%%%%%%%%%%%
%%%%%%%%%%%%%%%%%%%%%%%%%%%%%%%%%%%%%%%%%%%%%%%%%%%%%%%%%%%%%%%%%%%%%%%%%%%
 
\vskip 0.70in 
 
%%%%%%%%%%%%%%%%%%%%%%%%%%%%%%%%%%%%%%%%%%%%%%%%%%%%%%%%%%%%%%%%%%%%%%%%%
%%%%%%%%%%%%%%%%%%%%%%%%%%%%%%%%%%%%%%%%%%%%%%%%%%%%%%%%%%%%%%%%%%%%%%%%%
%%%%%%%%%%%%%%%%%%%%%%%%%%%%%%%%%%%%%%%%%%%%%%%%%%%%%%%%%%%%%%%%%%%%%%%%%

\vfill
\vfill \eject
%%%%%%%%%%%%%%%%%%%%%%%%%%%%%%%%%%%%%%%%%%%%%%%%%%%%%%%%%%%%%%%%%%%%%%%%%
%%%%%%%%%%%%%%%%%%%%%%%%%%%%%%%%%%%%%%%%%%%%%%%%%%%%%%%%%%%%%%%%%%%%%%%%%
%%%%%%%%%%%%%%%%%%%%%%%%%%%%%%%%%%%%%%%%%%%%%%%%%%%%%%%%%%%%%%%%%%%%%%%%%
 
}} 
%%%%% matches {\large{ 
 
\end{document}